\documentclass[reqno, 11pt]{amsart}
\usepackage{amsmath,amsthm,amsfonts,amssymb}
\usepackage{euscript}
\usepackage[english]{babel}
\usepackage{graphicx}
\usepackage{pgf,tikz}
\usepackage[backref, colorlinks, linkcolor=blue, citecolor=blue]{hyperref}
\usetikzlibrary{arrows}
\date{}

\usepackage{placeins}

\usepackage{comment}
\usepackage{tikz}

\usepackage[]{todonotes}

\newtheorem{theorem}{Theorem}[section]
\newtheorem{lemma}[theorem]{Lemma}
\newtheorem{prop}[theorem]{Proposition}
\newtheorem{corollary}[theorem]{Corollary}

\theoremstyle{definition}
\newtheorem{definition}[theorem]{Definition}
\newtheorem{remark}[theorem]{Remark}

\newcommand{\Z}{\mathbb{Z}}
\newcommand{\Q}{\mathbb{Q}}

\newcommand{\R}{\mathbb{R}}

\newcommand{\CC}{\mathbb{C}}

\newcommand{\HH}{\mathbb{H}}

\newcommand{\calG}{\mathcal{G}}

\newcommand{\Qbar}{\overline{\mathbb{Q}}}

\newcommand{\OO}{\mathrm{O}}
\newcommand{\OOO}{\mathcal{O}}

\newcommand{\Comm}{\mathrm{Comm}}

\newcommand{\PO}{\mathrm{PO}}

\newcommand{\G}{\mathbf{G}}

\newcommand{\I}{\mathbf{i}}
\newcommand{\J}{\mathbf{j}}
\newcommand{\K}{\mathbf{k}}

\DeclareMathOperator*{\rk}{rk}

\DeclareMathOperator*{\Isom}{Isom}

\DeclareMathOperator*{\Stab}{Stab}

\DeclareMathOperator*{\tr}{tr}

\DeclareMathOperator*{\Fix}{Fix}

\newcounter{z}

\pgfmathsetmacro{\minsize}{0.2cm}

\makeatletter
\def\@tocline#1#2#3#4#5#6#7{\relax
  \ifnum #1>\c@tocdepth 
  \else
    \par \addpenalty\@secpenalty\addvspace{#2}%
    \begingroup \hyphenpenalty\@M
    \@ifempty{#4}{%
      \@tempdima\csname r@tocindent\number#1\endcsname\relax
    }{%
      \@tempdima#4\relax
    }%
    \parindent\z@ \leftskip#3\relax \advance\leftskip\@tempdima\relax
    \rightskip\@pnumwidth plus4em \parfillskip-\@pnumwidth
    #5\leavevmode\hskip-\@tempdima
      \ifcase #1
       \or\or \hskip 1em \or \hskip 2em \else \hskip 3em \fi%
      #6\nobreak\relax
    \dotfill\hbox to\@pnumwidth{\@tocpagenum{#7}}\par
    \nobreak
    \endgroup
  \fi}
\makeatother

\title{Subspace stabilisers in hyperbolic lattices}

\author[M. Belolipetsky]{Mikhail Belolipetsky}
\address{IMPA, Estrada Dona Castorina 110, 22460-320 Rio de Janeiro, Brazil}
\email[]{mbel@impa.br}

\author[N. Bogachev]{Nikolay Bogachev} 
\address{Department of Computer and Mathematical Sciences, University of Toronto Scarborough, 1095 Military Trail, Toronto, ON M1C 1A3, Canada}
\email{n.bogachev@utoronto.ca}

\author[A. Kolpakov]{Alexander Kolpakov}
\address{CSTEM, University of Austin, Austin, TX 78701, USA}
\email[]{akolpakov@uaustin.org}

\author[L. Slavich]{Leone Slavich}
\address{Dipartimento di Matematica (DIMA), Universit\`a di Genova, Via Dodecaneso 35, 16146 Genova, Italy}
\email[]{leone.slavich@unige.it}

\begin{document}
\begin{abstract}
This paper shows that immersed totally geodesic $m$-dimen\-sional suborbifolds of $n$-dimensional arithmetic hyperbolic orbifolds  correspond to finite subgroups of the commensurator whenever $m \geqslant \frac{n-1}{2}$. We call such totally geodesic suborbifolds finite centraliser subspaces (or  fc-subspaces) and use them to formulate an arithmeticity criterion for hyperbolic lattices. 

We show that a hyperbolic orbifold $M$ is arithmetic if and only if it has infinitely many fc-subspaces, and exhibit examples of non-arithmetic orbifolds that contain non-fc subspaces of codimension one. We provide an algebraic characterization of totally geodesically immersed suborbifolds of arithmetic hyperbolic orbifolds by analysing Vinberg's commensurability invariants. This allows us to construct examples with the property that the adjoint trace field of the geodesic suborbifold properly contains the adjoint trace field of the orbifold. The case of special interest is that of exceptional trialitarian $7$-dimensional orbifolds. We show that every such orbifold contains a totally geodesic arithmetic  hyperbolic $3$-orbifold of exceptional type.

Finally, we study arithmetic properties of orbifolds that descend to their totally geodesic suborbifolds, proving that all suborbifolds in a (quasi-)arithmetic orbifold are (quasi-)arithmetic.
\end{abstract}

\maketitle

\tableofcontents

\section{Introduction}\label{section:intro}

Let $\HH^n$ be the real hyperbolic
$n$-space and $G = \mathbf{PO}_{n,1}(\R) = \Isom(\HH^n)$ be its isometry group.
Here and below a \textit{hyperbolic lattice} is a discrete subgroup of $\Isom(\HH^n)$ having finite covolume with respect to the Haar measure or, equivalently, admitting a fundamental polytope $P\subset \mathbb{H}^n$ of finite volume. In addition, a lattice is called \textit{uniform} if it is cocompact or, equivalently, it admits a compact fundamental polytope. Otherwise, a lattice is called \textit{non-uniform}. 

Given a lattice $\Gamma < G$, the associated quotient space $M = \HH^n/\Gamma$ is a finite volume \textit{hyperbolic orbifold}. It is a manifold when $\Gamma$ is torsion-free. 

In this paper we attempt to give a more particular description of finite volume totally geodesic immersed suborbifolds ({\itshape totally geodesic subspaces}, for short) of hyperbolic orbifolds and manifolds $\HH^n/\Gamma$, with $\Gamma$ being a lattice in $G$. Our special interest lies with the cases of \textit{arithmetic}, \textit{quasi-arithmetic} and \textit{pseudo-arithmetic} lattices \cite{Vin67,EM20}. Their exact definitions, as well as other properties and results that we shall essentially use, will follow in Section~\ref{sec:prelim}.

We shall distinguish three types of lattices that provide an exhaustive description of all arithmetic lattices in $\mathbf{PO}_{n,1}(\R)$. Essentially, \textit{type-I} arithmetic lattices come from admissible quadratic forms, \textit{type-II} arithmetic lattices come from skew-Hermitian forms (or, equivalently, Hermitian forms) over quaternion algebras, and \textit{type-III} arithmetic lattices comprise one exceptional family in $\mathbf{PO}_{3,1}(\R)$ and another in $\mathbf{PO}_{7,1}(\R)$ (the so-called ``trialitarian lattices'', cf. Section~\ref{sec:type3}).

Let $\mathrm{Comm}(\Gamma)$ be the commensurator of $\Gamma$ in $G$, and for $F < \mathrm{Comm}(\Gamma)$ let $\Fix(F) = \{x \in \HH^n\, | \, gx = x, \,\, \forall g \in F\}$ be the fixed point set of $F$ in $\mathbb{H}^n$.

\begin{definition}
An immersed totally geodesic suborbifold $N$ of a hyperbolic orbifold~$M=~\HH^n/\Gamma$ is called a \textit{finite centraliser subspace} (an \textit{fc-subspace} for short) if there exists a finite subgroup $F < \mathrm{Comm}(\Gamma)$ such that $U = \Fix(F)$ and $N = U/\Stab_\Gamma(U)$. 
\end{definition}

The definition of fc-subspaces is motivated by two classes of examples. Let $P$ be a finite-volume hyperbolic Coxeter $n$-polytope, $U \subset \mathbb{H}^n$ a totally geodesic subspace which supports a face $S$ of $P$ and denote by $\Gamma$ the Coxeter group generated by reflections in the facets of $P$. Then $U$ is fixed by a finite subgroup of $\Gamma$ and the quotient of $U$ under the action of its stabiliser in $\Gamma$ has finite covolume whenever $m = \mathrm{dim}(U) \geq 2$ \cite{AMN,Al13}. 
In a similar fashion, one can construct further examples of totally geodesic subspaces by considering the fixed-point set $U$ of a finite group $F$ of symmetries of a hyperbolic orbifold $M=\mathbb{H}^n/\Gamma$. In this setting the group $F$ is a subgroup of the normaliser of $\Gamma$ in $\mathrm{Isom}(\mathbb{H}^n)$, and $U$ projects to an fc-subspace $N \subset M$ under the action of its stabiliser in $\Gamma$.

The following result is akin to a well-known fact for uniform lattices in Lie groups: see Lemma~4.4 in \cite{OV} or combine Theorem~1.13 with Lemma~1.14 in \cite{Raghunathan}: 
\begin{theorem}\label{theorem:cent}
Let $\Gamma < \Isom(\HH^n)$ be a (uniform) lattice and $F < \Isom(\HH^n)$ be a finite subgroup, such that $U = \mathrm{Fix}(F)$ is an $m$-dimensional subsapce in $\HH^n$, with $m\geq 2$. If $F < \mathrm{Comm}(\Gamma)$, then the stabiliser $\mathrm{Stab}_{\Gamma}(U)$ of $U$ in $\Gamma$ is a (uniform) lattice acting on $U$. 
\end{theorem}

We shall assume throughout that the dimension of an fc-subspace is positive. Note that we do not exclude one-dimensional fc-subspaces. 
The central result of this paper is the following theorem.

\begin{theorem}\label{theorem-fc}
Let $M = \HH^n/\Gamma$ be a finite-volume hyperbolic $n$-orbifold. We have:
\begin{itemize}
    \item[(1)] If $M$ is arithmetic, then it contains infinitely many fc-subspaces of positive dimension. Moreover, all totally geodesic suborbifolds of $M$ of dimension $m \geqslant \frac{n-1}{2}$ which are not $3$-dimensional type III are fc-subspaces.
    \item[(2)] If $M$ is non-arithmetic, then it has finitely many fc-subspaces, their number being bounded above by $c\cdot\mathrm{vol}(M)$, with a positive constant $c = \mathrm{const}(n)$ depending only on $n$.
\end{itemize}
\end{theorem}
\noindent Notice that the condition that the totally geodesic subspace is not $3$-dimen\-sional type III in (1) is only needed when $4\leq n\leq 7$. We can actually show that the condition is only needed for $6\leq n \leq 7$. Indeed, the algebraic properties of $3$-dimensional type III lattices imply that the corresponding orbifolds can only be geodesically immersed as suborbifolds of codimension $\geq 3$.

Let us put this theorem in a more general perspective. Parts (1) and (2) of the theorem show a dichotomy for the fc-subspaces of arithmetic and non-arithmetic hyperbolic orbifolds. In fundamental recent works by Margulis and Mohammadi \cite{MM20} (for dimension $n = 3$ only) and independently by Bader, Fisher, Miller, and Stover \cite{BFMS20}, the arithmeticity of hyperbolic manifolds is established in terms of the existence of infinitely many maximal totally geodesic subspaces of dimension at least $2$ (thereby excluding $1$-dimensional geodesics). These results provide a sufficient condition for the arithmeticity of hyperbolic orbifolds, but do not detect arithmetic hyperbolic $2$-orbifolds (obviously) and $3$-dimensional orbifolds of type II and III, as they contain no totally geodesic immersed $2$-dimensional orbifolds. In the $3$-dimensional case, we provide an alternative proof of a result of Lackenby--Long--Reid \cite[Proof of Theorem 1.2]{LLR08} (see also \cite[Lemma 2.1]{CLR}). This allows us to detect the arithmeticity of these orbifolds by exploiting the existence of infinitely many fc–geodesics:

\begin{corollary}
Let $M$ be a finite-volume hyperbolic $3$-orbifold. Then $M$ is arithmetic if and only if $M$ contains infinitely many fc-geodesics. Moreover, in the arithmetic case, all immersed totally geodesic surfaces and curves are fc-subspaces.
\end{corollary}

In this regard, part (1) of Theorem \ref{theorem-fc} can be also viewed as a generalisation to arbitrary dimensions of the aforementioned result of Lackenby--Long--Reid.
Our methods also allow to detect arithmeticity of hyperbolic surfaces: 
\begin{corollary}
A finite area hyperbolic surface is arithmetic if and only if all of its infinitely many closed geodesics are fc-subspaces.
\end{corollary}
In proving Theorem \ref{theorem-fc}, we construct fc-subspaces of dimension $\geq 3$ in all arithmetic hyperbolic orbifolds of dimension $\geq 4$. The most interesting case is that of type-III lattices in $\mathrm{PO}_{7,1}$ for which we are able to prove the following property:

\begin{theorem}\label{teo:3-dim-type-III-in-7-dim-type-III}
    Every $7$-dimensional type-III orbifold $M$ contains a $3$-dimen\-sional type-III totally geodesic fc-subspace.
\end{theorem}

This theorem implies that the main result in \cite{BFMS20} becomes an arithmeticity criterion detecting all arithmetic orbifolds of dimension $\geq 4$:

\begin{corollary}
Let $M = \HH^n/\Gamma$ be a finite-volume hyperbolic $n$-orbifold, with $n\geq 4$. Then $M$ is arithmetic if and only if it contains infinitely many maximal immersed totally geodesic subspaces of dimension $\geq 3$. If this is the case then $M$ contains infinitely many $3$-dimensional totally geodesic fc-subspaces.
\end{corollary}

The proofs in \cite{BFMS20} and \cite{MM20} are based on new powerful superrigidity theorems. In this regard our work only relies on the classical Margulis superrigidity that appears in his proof of the arithmeticity theorem \cite[Chapter~IX]{Margulis-book}. On the other hand, we require a much more detailed analysis of the algebraic structure of arithmetic subgroups. An advantage of fc-subspaces compared to the general totally geodesic subspaces is that they are more concrete and amenable for constructive arguments.

In his study of lattices in $\mathbf{PO}_{n,1}(\mathbb{R})$, Vinberg introduced two commensurability invariants: the {\itshape adjoint trace field} $k$, which is an algebraic number field, and the {\itshape ambient group}, an algebraic $k$-group  $\mathbf{G}$ whose identity component $\mathbf{G}^{\circ}$ is a $k$-form of the real group $\mathbf{PO}_{n,1}$ if $n$ is even, or $\mathbf{PSO}_{n,1}$ if $n$ is odd (see Section \ref{sec:types-arithmetic-lattices}). We analyse the relation between the adjoint trace field of a hyperbolic orbifold $M=\mathbb{H}^n/\Gamma$ (i.e.\ the adjoint trace field of the lattice $\Gamma$) and the adjoint trace field of a totally geodesic immersed suborbifold $N=\mathbb{H}^m/\Lambda$, proving that (quasi-)arithmeticity is inherited by totally geodesic suborbifolds:

\begin{theorem}\label{theorem:geod}
Let $M$ be a quasi-arithmetic hyperbolic orbifold with adjoint trace field $k$, and $N \subset M$ be a finite-volume totally geodesic suborbifold of dimension $m \ge 2$ with adjoint trace field $K$. Then $N$ is hyperbolic and quasi-arithmetic. If $N$ is type-I or type-II, the field inclusion $k\subset K$ holds true. Moreover, if $M$ is arithmetic, then $N$ is arithmetic as well. 
\end{theorem}

\noindent The case where $M$ and $N$ are arithmetic and the totally geodesic subspace $N$ is $3$-dimensional type III is slightly more intricate, but we remark here that there is still a field inclusion of the form $k\subset L$, where $L$ is taken to be the invariant trace field of $N$. This is the analogue of the adjoint trace field, but computed using the representation of $\Lambda\cong \pi_1(N)$ as a subgroup of $\mathrm{PSL}_2(\mathbb{C})$. We refer the reader to Section \ref{sec:hereditary properties} for the precise statements in this case.

The inclusion of trace fields expressed by Theorem \ref{theorem:geod} is slightly counter-intuitive. If $M=\mathbb{H}^n/\Gamma$ and $U\cong \mathbb{H}^m$ is a lift of $N$ to $\mathbb{H}^n$, then the stabiliser $\mathrm{Stab}_{\Gamma}(U)$ of $U$ in $\Gamma$ is a subgroup of $\Gamma$ and the field generated by the traces of the adjoint action of $\mathrm{Stab}_{\Gamma}(U)$ on the Lie algebra of $\mathbf{O}_{n,1}(\mathbb{R})$ is a subfield of the adjoint trace field $k$ of $M$. However, in order to compute the adjoint trace field $K$ of $N$ one has to factor out the action of $\mathrm{Stab}_{\Gamma}(U)$ on the orthogonal bundle of the subspace $U$. 

In other words, the stabiliser of $U$ in $\mathbf{O}_{n,1}(\mathbb{R})$ is isomorphic to  the product $\mathbf{O}_{m,1}(\mathbb{R}) \times \mathbf{O}_{n-m}(\mathbb{R})$, and the adjoint trace field $K$ of $N$ is obtained by extending $\mathbb{Q}$ with the traces of the projection to the Lie algebra of $\mathbf{O}_{m,1}(\mathbb{R})$ of the adjoint action of $\mathrm{Stab}_{\Gamma}(U)$. The proper inclusion of trace fields $k \subset K$ then follows from Borel's Density Theorem and the analysis of the algebraic properties of the projection map $\mathrm{Stab}_{\Gamma}(U)\rightarrow \mathbf{O}_{m,1}(\mathbb{R})$.

There are several situations in which $k=K$, i.e.\ the trace field of $M$ coincides with that of the totally geodesic subspace $N$. For instance, this happens whenever $M$ is quasi-arithmetic and $N$ has codimension one.
In the case of a quasi-arithmetic \textit{reflection} group, Coxeter faces of the corresponding fundamental polytope are quasi-arithmetic over the same field of definition \cite{BK20}. 

The work of Emery and Mila \cite{EM20} shows that any Gromov--Piatetski-Shapiro manifold contains a totally geodesic subspace with a smaller adjoint trace field. Combining this with Theorem~\ref{theorem:geod} we obtain an alternative proof of the non-arithmeticity of these manifolds (see Remark \ref{rem:GPS-smaller-trace-field}). This argument for verifying non-arithmeticity may apply to other locally symmetric spaces as well.  

The fact that the adjoint trace field of a geodesic submanifold $N$ can be larger than the one of the arithmetic ambient manifold $M$ appears to be a previously unknown phenomenon and has some profound consequences which we explore thoroughly in the rest of this paper.

Following Theorem~\ref{theorem:geod}, we define a totally geodesic subspace $N$ of an arithmetic hyperbolic orbifold $M$ to be a {\itshape subform subspace} if its adjoint trace field coincides with that of $M$, and further refine the analysis of Vinberg's commensurability invariants in the arithmetic case by proving the following:

\begin{theorem}\label{teo:embeddings-hyperbolic-orbifolds}
Let $N=\mathbb{H}^m/\Lambda$ be a totally geodesic subspace of an arithmetic hyperbolic orbifold $M=\mathbb{H}^n/\Gamma$. Suppose that  $N$ is not a $3$-dimensional type-III orbifold and that $[K:k]=d\geq 1$, where $K$ (resp.\ $k$) denotes the adjoint trace field of $\Lambda$ (resp. $\Gamma$). Then there exists a unique minimal subform subspace $S\subseteq M$ of dimension $(m+1)\cdot d-1$ such that $N \subseteq S$, and there is no proper subform subspace of $S$ which contains $N$.
\end{theorem}

\noindent Theorem \ref{teo:embeddings-hyperbolic-orbifolds} does not extend naturally to the case where the totally geodesic subspace $N$ is $3$-dimensional type III.  We refer the reader to the end of Section \ref{sec:hereditary properties} for a more in-depth discussion of this particular case.

In order to prove Theorems \ref{theorem:geod} and \ref{teo:embeddings-hyperbolic-orbifolds} we analyse the relation between the ambient group $\mathbf{G}$ of an arithmetic orbifold $M$ and the ambient group $\mathbf{L}$ of the totally geodesic suborbifold $N\subset M$. Denote by $U$ a lift of $N$ to the universal cover $\mathbb{H}^n$ of $M$. The group $\mathbf{G}$ contains a closed, admissible, $k$-defined subgroup $\mathbf{H}$ such that $\mathbf{H}(\mathbb{R})<\mathrm{Stab}_{\mathbf{G}(\R)}(U)$, and $\mathbf{H}$ is $k$-isogenous to a product $\mathbf{C}\times \mathrm{Res}_{K/k}(\mathbf{L})$, where $\mathbf{C}$ is a $k$-group such that $\mathbf{C}(\mathbb{R})$ is compact and $\mathrm{Res}_{K/k}(\mathbf{L})$ denotes the Weil restriction from $K$ to $k$ of the group $\mathbf{L}$.

This fact has as a consequence that a totally geodesic immersion $N \subseteq M$ of arithmetic hyperbolic orbifolds is a composition of two totally geodesic immersions $N \subseteq S \subseteq M$. The immersion $N \subseteq S$ is determined by the factor $\mathrm{Res}_{K/k}(\mathbf{L})$, and $N$ is a {\itshape Weil restriction subspace} of $S$, while the immersion of $S$ into $M$ as a subform subspace depends on the compact factor $\mathbf{C}$. We provide a description of totally geodesic immersions of arithmetic hyperbolic orbifolds obtained as either Weil restriction subspaces or subform subspaces in Section \ref{sec:kinds-of-subspaces}.
In combination with Theorem \ref{teo:embeddings-hyperbolic-orbifolds}, this gives a complete classification of geodesic immersions between non-exceptional arithmetic hyperbolic orbifolds.

The work of Bergeron and Clozel \cite{BC17} actually implies that $7$-dimensional type-III orbifolds do not admit totally geodesic immersions into higher dimensional arithmetic hyperbolic orbifolds: if such a space were totally geode\-si\-cally immersed in a type-I or type-II space, then by passing to a sufficiently large congruence cover one would create a non-zero first homology class in the covering manifold (cf. \cite[Corollary~1.8]{BC13}), and by restriction (and injectivity of the stable restriction map in cohomology) one would contradict the main result of \cite{BC17}.
Theorem \ref{teo:3-dim-type-III-in-7-dim-type-III} shows that certain $3$-dimensional type-III orbifolds immerse in $7$-dimensional type-III spaces. 
The question remaining open is if there exist any other immersions between arithmetic hyperbolic orbifolds which involves type-III spaces.

It is worth stressing the fact that all codimension-$1$ totally geodesic suborbifolds in an arithmetic $n$-orbifold  are fc-subspaces. In Section \ref{sec:non-fc-arithmetic} we show how to build examples of non-fc subspaces (of high codimension) in arithmetic hyperbolic orbifolds. For non-arithmetic lattices we can give examples of totally geodesic subspaces of codimension $1$ which are not fc-subspaces. Note that by \cite{BFMS20} the number of such subspaces is always finite, although we do not have an effective upper bound for their number.

\begin{theorem}\label{teo:existence-not-fc-2}
Any non-arithmetic Gromov--Piatetski-Shapiro hyperbolic manifold contains a non-fc codimension-$1$ totally geodesic subspace.
\end{theorem}

The following fact follows from the recent results of Le--Palmer \cite{LP20} and some previous work of Reid--Walsh \cite{RW}.

\begin{theorem}\label{teo:existence-not-fc-twist-knots}
There exists a sequence of non-arithmetic hyperbolic $3$-manifolds $M_2,M_3, \ldots$ such that each $M_k$ contains exactly $k$ totally geodesic immersed surfaces and all of them are non-fc. 
\end{theorem}

In the last section, we apply the techniques introduced here to build some interesting examples of immersed totally geodesic subspaces of hyperbolic orbifolds. 
We begin by constructing an explicit example of a type-I Weil-restriction subspace in a type-I arithmetic orbifold (Section \ref{sec:invol-quadr-extension}). We then build an example of a type-I Weil-restriction subspace in a type-II arithmetic orbifold (Section \ref{sec:type-I-in-type-II}). 
Finally, we study totally geodesic sublattices of a particular non-arithmetic lattice generated by reflections in the facets of a non-compact Coxeter $5$-simplex (Section \ref{sec:examples}). We show that this example contains a $2$-dimensional arithmetic fc-subspace, arising as the fixed point set of its (unique) non-trivial symmetry.

The notion of fc-subspaces that we define and study in this paper for real hyperbolic orbifolds applies to the other locally symmetric spaces as well. In particular, it might be interesting to consider these subspaces in the complex hyperbolic case. For instance, in a recent paper \cite{D23}, Deraux carried out a detailed analysis of a set of examples of fc-subspaces of complex hyperbolic triangle orbifolds. As we observed in the real hyperbolic case, a systematic study of these subspaces may reveal new unexpected phenomena.

\subsection*{Structure of the paper}
In Section~\ref{sec:prelim} we recall some basic facts about algebraic groups, restriction of scalars, arithmetic lattices in semi-simple Lie groups and the definition of arithmetic, quasi-arithmetic and pseudo-arithmetic hyperbolic lattices. In Sections~\ref{sec:type1}, \ref{sec:type2}, \ref{sec:3-dim-arithmetic} and \ref{sec:type3} we review the classification of arithmetic hyperbolic lattices and describe the involutions in their commensurators. In Section~\ref{sec:kinds-of-subspaces} we describe the two main techniques to construct totally geodesic immersion of arithmetic hyperbolic orbifolds: subform subspaces and Weil restriction subspaces. The proofs of the main theorems are contained in Sections~\ref{sec:hereditary properties} and \ref{sec:proof-of-main-theorems}. In Sections~\ref{sec:non-fc} and \ref{sec:Le-Palmer} we exhibit examples of non-arithmetic lattices containing codimension-one non-fc subspaces and in Section~\ref{sec:non-fc-arithmetic} we show how to build examples of non-fc subspaces in arithmetic hyperbolic orbifolds. Finally, in Section~\ref{sec:example-section} we exhibit some interesting examples of totally geodesic immersions of hyperbolic orbifolds. 

\subsection*{Funding}
The work of M.B. was partially supported by 
CNPq, FAPERJ and by the MPIM in Bonn. The work of N.B. was partially supported by Russian Federation Government grant no.~075-15-2019-1926 and by the Theoretical Physics and Mathematics Advancement Foundation ``BASIS''. The work of A.K. was supported by the SNSF project no.~PP00P2-170560. The work of L.S. was supported by the PRIN project ``Geometry and topology of manifolds'' project no.~F53D23002800001 and by the INdAM Institute.

\subsection*{Acknowledgements}
We thank Nicolas Bergeron, Vincent Emery, and Leonid Potyagailo for their comments on an earlier version of this paper. We also thank Philippe Gille for suggesting the proof of Proposition \ref{prop:inner-involutions-7-dim-type-III}. N.\,B. is grateful to Sami Douba for fruitful discussions on arithmetic groups and hyperbolic geometry, as well as useful comments on the manuscript. L.\,S.\  thanks Giuseppe Ancona for the many invaluable discussions on the theory of algebraic groups. We are also grateful to Uri Bader for useful remarks on the later version of our paper. Last but not least, we thank an anonymous referee for a thorough reading of our paper and for the many useful remarks and suggestions.

\subsection*{Notation}
Let us introduce the following standard notation for the whole paper (unless stated otherwise):
\begin{itemize}
    \item If $k$ is a number field, then $k^\times$ denotes its multiplicative group, and $\OOO$ or $\OOO_k$ denotes its ring of integers. The algebraic closure of $k$ is denoted by $\overline{k}$. 
    \item Bold capital letters $\mathbf{G}$, $\mathbf{H}$, $\mathbf{O}_n$, etc., denote algebraic groups. Since we will be working exclusively with fields of characteristic $0$, all algebraic groups will be understood as subgroups of $\mathbf{GL}_n(\mathbb{C})$ for some $n>0$. By $\G^{\circ}$ we denote the connected component of the identity element of $\mathbf{G}$ in the Zariski topology. If $R \subset \mathbb{C}$ is a ring and $\mathbf{G}<\mathbf{GL}_n(\mathbb{C})$, then $\mathbf{G}(R)$ denotes the group of $R$-points, i.e.\ the subgroup of $\mathbf{G}$ consisting of matrices with entries in $R$.
    \item If $k \subset \mathbb{C}$ is a field and $\mathbf{G}<\mathbf{GL}_n(\mathbb{C})$ is a $k$-group, then $\mathbf{PG}$ denotes the adjoint group of $\mathbf{G}$ and $\mathbf{PG}(k)$ denotes the $k$-points of $\mathbf{PG}$ (now seen as a subgroup of $\mathbf{GL}(\mathfrak{g})$, where $\mathfrak{g}$ is the Lie algebra of $\mathbf{G}$). In general, the restriction $\mathbf{G}(k) \to \mathbf{PG}(k)$ of the adjoint map to the $k$-points is non-surjective.
    \item {\it Warning}: for an algebraic group $\G$, the non-bold notation $\mathrm{PG}(R)$ means the projectivization of $\G(R)$, i.e., $\mathrm{PG}(k) = \G(k)/Z(\G(k))$. For example, $\mathrm{PO}_f(k) = \mathbf{O}_f(k)/\{\pm I\}$, which is not the same as $\mathbf{PO}_f(k)$. 
    \item If $K/k$ is a finite field extension (both fields of characteristic zero) and $\mathbf{G}$ is an algebraic $K$-group, $\mathrm{Res}_{K/k}\,\mathbf{G}$ denotes the algebraic $k$-group obtained from $\mathbf{G}$ through the Weil restriction of scalars from $K$ to $k$. Its real points are denoted by $\mathrm{Res}_{K/k}\,\mathbf{G} (\mathbb{R})$. The set of field embeddings $\sigma:K \rightarrow \mathbb{C}$ that restrict to the identity on $k$ is denoted by $S^{\infty}_{K/k}$.
    \item The capital letters $G$, $H$, $\mathrm{O}_n$, etc., denote real Lie groups, and $G^{\circ}$ denotes the connected component of the identity of the group $G$ in the manifold topology. We warn the reader not to confuse $\mathbf{G}^{\circ}(\mathbb{R})$ (the Lie group consisting of the real points of the connected algebraic group $\mathbf{G}^{\circ}$) and $\mathbf{G}(\R)^{\circ}$ (the identity component of the Lie group $\mathbf{G}(\R)$).
    \item By $\Gamma$, $\Lambda$, etc., we shall denote \textit{lattices} in real Lie groups.
    \item A lattice $\Gamma$ is called \textit{uniform} in $G$ if $G/\Gamma$ is compact.
    \item Two subgroups $\Gamma_1 $ and $\Gamma_2$ of a group $G$ are called \textit{commensurable} and denoted $\Gamma_1 \sim \Gamma_2$ if the group $\Gamma_1 \cap \Gamma_2$ is a subgroup of finite index in each of them.
    \item Two subgroups $\Gamma_1$ and $\Gamma_2$ of $G$ are \textit{commensurable in the wide sense} if $\Gamma_1 \sim g \Gamma_2 g^{-1}$ for some $g \in G$ (we shall usually understand commensurability in the wide sense);
    \item \textit{The commensurator} of $\Gamma$ in $G$ is the group $$\mathrm{Comm}_G (\Gamma) = \{g \in G \mid g \Gamma g^{-1} \sim \Gamma\}.$$ 
\end{itemize}

\section{Preliminaries}\label{sec:prelim}

\subsection{Algebraic groups}
Some parts of this work require a considerable amount of the theory of algebraic groups. We give a short overview below and refer to \cite{Pla-Rap} for a comprehensive introduction.

Let $\Omega$ denote an algebraically closed field of characteristic zero. For the purpose of this work we may assume that $\Omega$ is either $\mathbb{C}$ or the field $\Qbar$ of algebraic numbers. A {\itshape linear algebraic group} is a Zariski-closed subgroup of the general linear group $\mathrm{GL}_n(\Omega)$. As such it is an algebraic subvariety $\mathbf{G}$ of $\mathrm{GL}_n(\Omega)$ such that the morphisms 
\begin{align}
&\mathbf{G} \times \mathbf{G} \ni (x,y) \mapsto x\cdot y \in \mathbf{G},\\
&\mathbf{G} \ni x \mapsto x^{-1} \in \mathbf{G},
\end{align} are algebraic and satisfy the group axioms. A morphism $\phi:\mathbf{G} \rightarrow \mathbf{H}$ is a morphism of algebraic varieties which is also a group homomorphism. An {\itshape isogeny} is an epimorphism with finite kernel.

If $k \subset \Omega$ is a subfield and $\mathbf{G}<\mathrm{GL}_n(\Omega)$ is a linear algebraic group, we say that $\mathbf{G}$ is {\itshape defined over $k$} (or that $\mathbf{G}$ is a {\itshape $k$-group}) if the ideal $\mathcal{I}$ of polynomial 
functions vanishing on $\mathbf{G}$ is generated by the intersection of $\mathcal{I}$ with the algebra of polynomials with coefficients in $k$. A morphism $\phi:\mathbf{G} \rightarrow \mathbf{H}$ of algebraic $k$-groups is defined over $k$ (i.\,e.\ it is a {\itshape $k$-morphism}) if it can be expressed via polynomials with coefficients in $k$. 
If $\mathbf{G}$ is an algebraic $k$-group and $\mathbf{H}$ is a normal $k$-subgroup of $\mathbf{G}$, then the quotient $\mathbf{G}/\mathbf{H}$ is a $k$-group and the quotient map $\mathbf{G} \rightarrow \mathbf{G}/\mathbf{H}$ is a $k$-morphism.
A {\itshape $k$-isogeny} is an isogeny which is defined over $k$.

If $\mathbf{G}\subset \mathrm{GL}_n(\Omega)$ is a $k$-group, its group of $k$-points is the intersection
$$\mathbf{G}(k) = \mathbf{G} \cap \mathrm{GL}_n(k).$$
If $\Omega=\mathbb{C}$ and $\mathbf{G}$ is a linear algebraic group, then $\mathbf{G} < \mathrm{GL}_n(\mathbb{C})$ is endowed with a complex Lie group structure.
If $k\subset \mathbb{R}$ and $\mathbf{G}$ is a $k$-group, then
$$\mathbf{G}(\mathbb{R})=\mathbf{G} \cap \mathrm{GL}_n(\mathbb{R})$$ is a real Lie group.

Given field extensions $k\subset K \subset \Omega$, we can regard a $k$-group $\G$ as a $K$-group $\mathrm{Ext}_{K/k}{\G}$ which is said to be obtained from $\mathbf{G}$ via extension of scalars from $k$ to $K$. If $\mathbf{H}$ is a $K$-group and $\mathrm{Ext}_{K/k}{\G}$ is $K$-isomorphic to $\mathbf{H}$, we say that the $k$-group $\mathbf{G}$ is a $k$-form of $\mathbf{H}$.

Given an algebraic $k$-group $\mathbf{G}<\mathrm{GL}_n(\Omega)$, we denote the connected component of the identity (in the Zariski topology) by $\mathbf{G}^{\circ}$. The identity component $\mathbf{G}^{\circ}$ is
a finite-index normal $k$-subgroup of $\mathbf{G}$, and $\mathbf{G}$ is {\itshape connected} if $\mathbf{G}=\mathbf{G}^{\circ}$. If $\mathbf{G}$ is a $k$-group then $\mathbf{G}^{\circ}(k)$ is Zariski dense in $\mathbf{G}^{\circ}$.

A {\itshape torus} is a connected algebraic group $\mathbf{T}$ for which there exists an isomorphism $\mathbf{T} \cong (\mathbb{G}_m)^d$, where $\mathbb{G}_m \cong \mathbf{GL}_1(\Omega)$ denotes the multiplicative group of $\Omega$ and $d=\mathrm{dim}(\mathbf{T})$ is the dimension of $\mathbf{T}$. A {\itshape character} of a torus $\mathbf{T}$ is a morphism $\chi:\mathbf{T} \rightarrow \mathbb{G}_m$. The characters of a torus $\mathbf{T}$ form a commutative group $\mathbf{X}^*(\mathbf{T})$ under the operation $(\chi_1 + \chi_2)(g)=\chi_1(g)\cdot \chi_2(g)$. If $\mathbf{T}$ has dimension $d$, then $\mathbf{X}^*(\mathbf{T})$ is isomorphic to $\mathbb{Z}^d$. A $k$-defined torus $\mathbf{T}$ that admits a $k$-defined isomorphism $\mathbf{T} \cong (\mathbb{G}_m)^d$ is said to be {\itshape $k$-split}, and this is equivalent to all the characters in $\mathbf{X}^*(\mathbf{T})$ being defined over $k$.

Let $\mathbf{G}$ be a connected algebraic group.
The maximal connected normal solvable subgroup of $\mathbf{G}$ is called the {\itshape radical} of $\mathbf{G}$ and is denoted by $\mathrm{R}(\mathbf{G})$. The radical of a $k$-group is always defined over $k$. If $\mathrm{R}(\mathbf{G})=\{e\}$ the group $\mathbf{G}$ is said to be {\itshape semisimple}. The quotient of any connected $k$-group $\mathbf{G}$ by its radical is a semisimple $k$-group. A disconnected algebraic group $\mathbf{G}$ is semisimple if $\mathbf{G}^{\circ}$ is semisimple.

All maximal tori in a connected semisimple algebraic group $\mathbf{G}$ are conjugate under $\mathbf{G}(\Omega)$ and thus all have the same dimension $d$, which we call the {\itshape rank} of $\mathbf{G}$. If $\mathbf{G}$ is a $k$-group, all maximal $k$-split tori are conjugate under $\mathbf{G}(k)$, and their dimension $s$ is the {\itshape $k$-rank} of $\mathbf{G}$. A maximal $k$-split torus $\mathbf{S}$ is always contained in a maximal $k$-defined torus $\mathbf{T}$.

If $s=0$ the $k$-split tori are trivial and the group $\mathbf{G}$ is said to be {\itshape anisotropic}. If $s=d$ then there exists a $k$-defined maximal torus and $\mathbf{G}$ is said to be {\itshape split}.

A {\itshape Borel subgroup} of a connected algebraic group $\mathbf{G}$ is a maximal connected solvable subgroup $\mathbf{B} < \mathbf{G}$. All Borel subgroups are conjugate under $\mathbf{G}(\Omega)$. It is not necessarily true that a $k$-group $\mathbf{G}$ has a Borel subgroup defined over $k$. If there exists a $k$-defined Borel subgroup, then $\mathbf{G}$ is said to be {\itshape quasi-split}.

A non-commutative connected algebraic group is {\itshape absolutely almost simple} if it has no nontrivial, connected, normal subgroups. To any connected semisimple group $\mathbf{G}$ one can associate the (finite) set $\{\mathbf{G}_1,\dots,\mathbf{G}_r\}$ of minimal connected normal subgroups of $\mathbf{G}$. Each of these subgroups is absolutely almost simple and $\mathbf{G}$ is an almost direct product of $\mathbf{G}_1, \dots,\mathbf{G}_r$, meaning that the map 

\begin{equation}\label{eq:absolutely-simple-isogeny}\textstyle{\prod_{i=1}^r} \mathbf{G}_i \ni (x_1,\dots,x_r) \mapsto x_1 \cdot\ldots \cdot x_r \in \mathbf{G}\end{equation} is an isogeny.

A connected algebraic $k$-group $\mathbf{G}$ is {\itshape almost $k$-simple} if it has no nontrivial, connected, $k$-defined normal subgroups. A connected $k$-group $\mathbf{G}$ is an almost direct product of its finitely many minimal connected normal $k$-defined subgroups, which are all almost $k$-simple. If $\mathbf{G}$ is almost $k$-simple, there exists a field $k'$ containing $k$, and an absolutely almost simple $k'$-group $\mathbf{E}$ such that $\mathbf{G}$ is isomorphic to the group $\mathrm{Res}_{k'/k}(\mathbf{E})$ obtained from $\mathbf{E}$ via {\itshape restriction of scalars} from $k'$ to $k$ (see \cite[\S 3.1.2.]{Tits66} and Section \ref{sec:restriction of scalars}).

The Lie algebra $\mathfrak{g}$ of an algebraic $k$-group $\mathbf{G}$ is defined as the algebra of left invariant derivations on the algebra of regular functions of $\mathbf{G}$ (Lie bracket given by the commutator). The group $\mathbf{G}$ acts by conjugation on $\mathfrak{g}$ via Lie algebra automorphisms, yielding the adjoint representation
$$\mathrm{Ad}:\mathbf{G} \rightarrow \mathbf{GL}(\mathfrak{g}). $$
The kernel of the adjoint representation is the centraliser $\mathrm{Z}(\mathbf{G}^{\circ})$ of the identity component, which is a finite, normal $k$-subgroup of $\mathbf{G}$. A $k$-group is adjoint if the adjoint representation is faithful. If $\mathbf{G}$ is semisimple, the quotient map $$ \mathbf{G} \rightarrow \mathbf{G}/\mathrm{Z}(\mathbf{G}^{\circ})$$ is a $k$-isogeny of $\mathbf{G}$ onto the adjoint $k$-group $\mathbf{PG}=\mathbf{G}/\mathrm{Z}(\mathbf{G}^{\circ})$.
If $\mathbf{G}$ is a connected semisimple adjoint group, it decomposes as a direct product of its absolutely simple factors, i.\ e.\ the isogeny in \eqref{eq:absolutely-simple-isogeny} is an isomorphism \cite[Theorem 2.6.]{Pla-Rap} and each factor of the direct product is simple. If moreover $\mathbf{G}$ is defined over $k$, it decomposes as a direct product of its $k$-simple factors \cite[\S 3.1.2]{Tits66}.

\subsubsection{Tits' classification of semisimple algebraic groups  \cite{Tits66}}
Suppose that $\mathbf{G}$ is a connected semisimple algebraic $k$-group, and denote by $\calG=\mathrm{Gal}(\overline{k}/k)$ the absolute Galois group of $k$.
Let $\mathbf{S}<\mathbf{G}$ be a maximal $k$-split torus and $\mathbf{T} < \mathbf{G}$ be a maximal $k$-defined torus which contains $\mathbf{S}$. Denote by $\Sigma \subset \mathbf{X}^*(\mathbf{T})$ the set of all roots of $\mathbf{G}$ relative to $\mathbf{T}$, by $\mathbf{N}$ the normaliser of $\mathbf{T}$ in $\mathbf{G}$ and by $W=\mathbf{N}/\mathbf{T}$ the {\itshape Weyl group} of $\mathbf{G}$ relative to $\mathbf{T}$. 

In $\mathbf{X}^*(\mathbf{T})\otimes \mathbb{R}$, we choose a scalar product invariant under the natural action of $W$, endowing $\Sigma$ with the structure of a root system. We also choose compatible orders in the character group $\mathbf{X}^*(\mathbf{S})$ and $\mathbf{X}^*(\mathbf{T})$, denote by $\Delta$ the system of simple roots for $\mathbf{G}$ relative to $\mathbf{T}$, and by  $\Delta_{0}$ the subsystem of those roots which vanish on $\mathbf{S}$.

The natural action of the group $\calG$ on $\mathbf{X}^*(\mathbf{T})$ induces an action by automorphisms of the root system $\Sigma$. There is a splitting short exact sequence
\begin{equation}\label{eq:exact-sequence-root-systems}
    1\rightarrow W\rightarrow \mathrm{Aut}(\Sigma)\rightarrow \Theta \rightarrow 1
\end{equation}
where $\Theta=\{\phi \in \mathrm{Aut}(\Sigma)|\,\phi(\Delta)=\Delta\}$ is isomorphic to the group of automorphisms of the Dynkin diagram of the root system $\Sigma$. This readily implies that $\mathrm{Aut}(\Sigma)$ is isomorphic to the semidirect product of  $W$ by $\Theta$ \cite[Section 12.2]{Humphreys}. If the action of $\calG$ on $\Sigma$ takes values in the Weyl group $W$, the group $\mathbf{G}$ is called an {\itshape inner form}. 
The most relevant properties of the action of $\calG$ on the root system $\Sigma$ can be encoded in the {\itshape Tits index} and in the {\itshape Tits symbol}, which are defined as follows.

The action of $\calG$ on $\Sigma$ projects to an action, called the *-action, on the system $\Delta$ of simple roots and on the Dynkin diagram of the root system. The index of a group $\mathbf{G}$ is the data of the Dynkin diagram, together with the *-action of the absolute Galois group $\calG$ on the diagram. The orbits of the vertices in $\Delta \setminus \Delta_0$ under the *-action are the so-called {\itshape distinguished orbits} and are circled.

We notice that it is implicit in Tits' definition via the action of $\calG$ on conjugacy classes of maximal parabolic groups \cite[\S 2.3]{Tits66} that the *-action  of $\calG$ on the Dynkin diagram does not depend on the choice of a maximal $k$-torus $\mathbf{T}$.

By \cite[\S 3.1.2]{Tits66}, if $\mathbf{G}$ is semisimple and defined over $k$ then it decomposes as an almost direct product of $k$-simple groups (a group is $k$-simple if it has no infinite normal subgroup defined over $k$). If $\mathbf{G}$ is $k$-simple, then there exists a finite field extension $K$ of $k$ and an abolutely simple $K$-defined group $\mathbf{H}$ such that $\mathbf{G}$ is isogenous to the group $\mathrm{Res}_{K/k}\,\mathbf{H}$ obtained from $\mathbf{H}$ via restriction of scalars from $K$ to $k$ (see Section \ref{sec:restriction of scalars} for the definition of restriction of scalars). Moreover the field $K\cong k[x]/(p(x))$ is uniquely determined as an abstract field extension  of $k$. By an absolutely simple group we mean a group $\mathbf{H}$ which has no infinite normal subgroup defined over $\mathbb{C}$. 

It follows that the classification of semisimple algebraic groups (up to isogeny) is reduced to the classification of the absolutely simple ones. If $\mathbf{G}$ is assumed to be either simply connected or adjoint, the above decompositions are in fact direct product decompositions with all factors simply connected or adjoint.

The Tits symbol of a group $\mathbf{G}$ is a symbol of the form $^{g}X_{n,r}^{(t)}$, where $X_n$ determines the (absolute) type of the Dynkin diagram of $\mathbf{G}$, $g$ denotes the order of the *-action of the group $\calG$ on the Dynkin diagram and $r$ is the {\itshape relative rank} of $\mathbf{G}$, i.\ e.\ the number of distinguished orbits for the *-action. The ${(t)}$ symbol appears only for groups of classical type and it corresponds to the degree of a certain central division $k$-algebra involved in the definition of the corresponding group (see \cite[pp. 55-61]{Tits66}).

The group $\mathbf{G}$ is anisotropic if and only if $\Delta_0=\Delta$ (equivalently, $r=0$), while it is quasi-split if and only if $\Delta_0=\emptyset$. In the quasi-split case all roots in the Tits index belong to a circled orbit and the action of $\calG$ on the root system $\Sigma$ preserves a system $\Delta' \subset \Sigma$ of simple roots.

Indeed, let us denote by $\mathfrak{t}$ the Lie algebra of $\mathbf{T}$, by $\mathfrak{b}$ the Lie algebra of $\mathbf{B}$ and by $L_{\alpha}=\{x \in \mathfrak{g}\,|\; [t,x]=\alpha(t)\cdot x\ \textrm{for all } t \in \mathfrak{t}\}$ the {\itshape root space} associated to a root $\alpha \in \Sigma$ (where $\alpha$ is now interpreted as an element of the dual $\mathfrak{t}^*$ of $\mathfrak{t}$). We have that 
\begin{equation}\label{eq:Borel-decomposition}\mathfrak{b}=\mathfrak{t} \oplus \sum_{\alpha>0} L_{\alpha},\end{equation}
where the sum on the right hand side ranges over all {\itshape positive} roots with respect to the partial ordering induced by the choice of a set $\Delta'$ of simple roots \cite[Section 16.4]{Humphreys}. Since $\mathbf{T}$ and $\mathbf{B}$ are $k$-defined, it follows that the direct sum decomposition in \eqref{eq:Borel-decomposition} is preserved under the action of $\calG$. Hence, also the system of simple roots $\Delta'$ is preserved by the action of $\calG$.  
By choosing the ordering on $\mathbf{X}^*(\mathbf{T})$ so that $\Delta=\Delta'$ we may assume that the action of $\calG$ on the Dynkin diagram of the root system $\Sigma$ is given by elements of the subgroup $\Theta < \mathrm{Aut}(\Sigma)$. 

If $\mathbf{G}$ is a semisimple algebraic $k$-group, its Tits index and symbol are by definition those of the connected group $\mathbf{G}^{\circ}$.

\subsection{Weil restriction of scalars}\label{sec:restriction of scalars}

In this section we briefly review a classical construction in algebraic geometry which will play an important role throughout the paper: {\itshape Weil's restriction of scalars}.

Suppose that $K/k$ is a finite extension of algebraic number fields of degree $d$ and that $X$ is an algebraic variety over $K$ of dimension $n$. Then $X$ can be interpreted as an algebraic variety $\mathrm{Res}_{K/k}\,X$ over $k$ of dimension $n\cdot d$. Such operation yields a covariant functor from the category of algebraic varieties over $K$ to the category of algebraic varieties over $k$, since for any $K$-morphism $f:X\rightarrow Y$ there exists an induced $k$-morphism $$\mathrm{Res}_{K/k}\,f:\mathrm{Res}_{K/k}\,X\rightarrow \mathrm{Res}_{K/k}\,Y.$$ 

These functorial properties easily follow from the existence of a natural map $p:\mathrm{Res}_{K/k}\,X\rightarrow X$ which is $K$-defined and has the following {\itshape universal property} (cf. \cite[Section 1.7]{Margulis-book}): for any $k$-variety $Y$ and any $K$-morphism $f:Y\rightarrow X$, there exists a unique $k$-morphism $\phi:Y\rightarrow \mathrm{Res}_{K/k}\,X$ such that $f=p\circ \phi$. The map $p$ induces a bijection between the $k$-points of $\mathrm{Res}_{K/k}\,X$ and the $K$-points of $X$. The functor $\mathrm{Res}_{K/k}$ is right adjoint to the {\itshape extension of scalars} functor $\mathrm{Ext}_{K/k}$ that takes an algebraic variety $X$ over $k$ and sees it as an algebraic variety $\mathrm{Ext}_{K/k}\,X$ defined over $K$.

We shall be exclusively concerned with the case where $X$ is an affine variety, in which case restriction of scalars admits a fairly explicit description as we now explain.
We notice that left multiplication is a $k$-linear map on the $k$-vector space $K$.
By fixing a basis $\mathcal{B}=(b_1,\dots,b_{d})$, we construct the left-regular representation of $K$ as a (commutative) $k$-subalgebra $\mathcal{A}(k)$ of the algebra $M_{d \times d}(k)$. The equations that identify $\mathcal{A}(k)$ are $k$-linear in the entries $y_{ij}, i,j=1,\dots,d$, of the matrices in $M_{d \times d}(k)$.
 
Let us suppose that $X\subset \mathbb{A}_K^N$ is defined as the zero locus of a finite set of polynomials $p_1,\dots, p_m \in K[x_1,\dots,x_N].$ 
Using the left regular representation of $K$, we can associate to each equation of the form $p_l(x_1,\dots,x_N)=0$, $l=1,\dots,k$, a system of polynomial equations in $d^2 N$ variables with coefficients in $k$. It is sufficient to interpret each of the coefficients of $p_l$ as a $d\times d$ matrix with coefficients in $k$ and each variable $x_h$ as a $d\times d$ matrix with entries given by variables $y_{ijh}$,  and $i,j=1,\dots,d$. The operation of multiplication in $K$ now translates to row-by-column multiplication of $d\times d$ matrices in $\mathcal{A}(k)$.

Finally, we can form a system of polynomial equations with coefficients in $k$ by adjoining the equations coming from each polynomial $p_l$, $l=1,\dots m$, with the linear equations involving the coefficients $y_{ijh}$, $i,j=1,\dots,d$ which define $\mathcal{A}(k)$ as a subalgebra of $M_{d \times d}(k)$ (we adjoin this set of equations for each variable $x_h$, $h=1,\dots,N$). Then the restriction of scalars $$\mathrm{Res}_{K/k}\,X \subset \mathbb{A}^{d^2N}_k$$ is the affine $k$-variety defined as the zero-locus of this system of equations. It follows from the construction that there is a one-to-one correspondence between the $K$-points of $X$ and the $k$-points of $\mathrm{Res}_{K/k}\,X$.

Assume that we have fixed a field embedding $k \rightarrow \mathbb{R}$. In view of our need to review the connection between restriction of scalars and the construction of arithmetic lattices in semisimple Lie groups, we are particularly interested in describing the group of real points $\mathrm{Res}_{K/k}\,X(\mathbb{R})$ of $\mathrm{Res}_{K/k}\,X$. 

We denote by $S_{K/k}^{\infty}$ the set of field embeddings of $K$ which restrict to the identity on $k$. There are $d=[K:k]$ such embeddings, so that $S_{K/k}^{\infty} = \{\sigma_0,\sigma_1,\dots,\sigma_{d-1}\}$. For each $\sigma \in S_{K/k}^{\infty}$ and $p\in K[x_1,\dots,x_N]$, we denote by $p^{\sigma}$ the polynomial obtained by applying $\sigma$ to each coefficient of $p$. Similarly, we denote by $X^{\sigma}$ the affine algebraic $\sigma(K)$-variety defined as the zero locus of $p_1^{\sigma}, \ldots, p_m^{\sigma}\in \sigma(K)[x]$. 

By the primitive element theorem, the field $K$ is equal to $k(\alpha)$ for some $\alpha \in K$ with minimal polynomial $q(x) \in k[x]$ of degree $d$. Therefore, as an abstract field $K$ is isomorphic to $k[x]/(q(x))$. By extending coefficients to $\mathbb{R}$, we see that $K \otimes_k \mathbb{R}$ is isomorphic to $\mathbb{R}[x]/(q(x))$. Let us assume that $q(x)$ has $r$ real roots and $c=(d-r)/2$ pairs of complex conjugate roots. Choose a set $\mathcal{R}$ of representatives for the infinite places of $K$, i.e.\ $\mathcal{R}$ contains all the real roots $\{\alpha_1,\dots,\alpha_r\}$ of $q$ plus a set $\{\beta_1,\dots,\beta_c\}$ of representatives for each pair of complex conjugate roots. Denote by $\sigma_i$, $i=1,\dots,r$ the embedding of $K$ defined by $x\mapsto \alpha_i$ and by $\tau_j$, $j=1,\dots,c$, the embedding defined by $x\mapsto \beta_j$.  The map 
\begin{equation}\label{eq:tensoring-by-R}
\mathbb{R}[x]/(q(x)) \ni [p(x)] \mapsto (p(\alpha_1),\dots, p(\alpha_r), p(\beta_1), \dots, p(\beta_c))
\end{equation} 
obtained by considering these $r+c$ embeddings simultaneously is an isomorphism between $K \otimes_k \mathbb{R}$ and $\mathbb{R}^r \times \mathbb{C}^c$.

Notice that we have the following ring isomorphisms:
$$K\cong \mathcal{A}(k),\; K\otimes_k \mathbb{R}\cong \mathcal{A}(\mathbb{R}),$$ and thus, by composing with the isomorphism in \eqref{eq:tensoring-by-R}, $\mathcal{A}(\mathbb{R})$ is canonically isomorphic to $\mathbb{R}^r \times \mathbb{C}^c$.
Finally, we interpret the variety $X$ as being defined over the abstract field $k[x]/(q(x))$, while $\mathrm{Res}_{K/k}\,X(\mathbb{R})$ corresponds to $X\left(\mathbb{R}[x]/(q(x)) \right)$, and due to the isomorphism \eqref{eq:tensoring-by-R} we have that $$\mathrm{Res}_{K/k}\,X(\mathbb{R}) \cong  X^{\sigma_1}(\mathbb{R})\times \ldots\times X^{\sigma_{r}}(\mathbb{R}) \times X^{\tau_1}(\mathbb{C})\times \ldots\times X^{\tau_c}(\mathbb{C}).$$ This decomposition is actually defined over the Galois closure $\overline{K}$ of the extension $K/k$.  The $k$-points of $\mathrm{Res}_{K/k}\,X$ correspond to elements of the form $$(\sigma_1(x),\dots,\sigma_r(x), \tau_1(x),\dots,\tau_{c}(x))$$ where $x \in X(K)$. 

Finally, we notice that if $\G$ is an algebraic $K$-group, then $\mathrm{Res}_{K/k}\,\G$ is an algebraic $k$-group. Indeed, the group structure of $\G$ is defined by a $K$-polynomial map $\G \times \G \rightarrow \G$, and the functorial nature of restriction of scalars induces a $k$-polynomial map $\mathrm{Res}_{K/k}\,\G \times \mathrm{Res}_{K/k}\,\G \rightarrow \mathrm{Res}_{K/k}\,\G$ which endows $\mathrm{Res}_{K/k}\,\G$ with a group structure. 

\subsection{Arithmetic lattices}\label{sec:general definitions}
Let $G$ be a non-compact, connected, simple real Lie group (i.e.\ a group whose Lie algebra $\mathfrak{g}$ is simple). 

A subgroup $\Gamma<G$ is an \textit{arithmetic lattice} if there exist:
\begin{enumerate}
\item an algebraic number field $k$,
\item a connected, $k$-simple algebraic $k$-group $\mathbf{G}$,
\item a continuous surjection with compact kernel:
$$p:\mathrm{Res}_{k/\mathbb{Q}}\mathbf{G}(\mathbb{R})^{\circ}\rightarrow G $$
\end{enumerate}
such that $p(\mathrm{Res}_{k/\mathbb{Q}}\,\mathbf{G}(\mathbb{Z}) \cap \mathrm{Res}_{k/\mathbb{Q}}\mathbf{G}(\mathbb{R})^{\circ} )$ is commensurable with $\Gamma$.

Notice how an arithmetic lattice defined as above is indeed a lattice, in the sense that coset space $G/\Gamma$ carries a finite $G$-invariant measure. Indeed it follows from the Theorem of Borel and Harish-Chandra \cite[Theorem~12.3]{BHC62} that $\mathrm{Res}_{k/\mathbb{Q}}\,\mathbf{G}(\mathbb{Z})$ is a lattice in $\mathrm{Res}_{k/\mathbb{Q}}\,\mathbf{G}(\mathbb{R})$. Since being a lattice is invariant under commensurability, it follows that the intersection of $\mathrm{Res}_{k/\mathbb{Q}}\,\mathbf{G}(\mathbb{Z})$ with $\mathrm{Res}_{k/\mathbb{Q}}\mathbf{G}(\mathbb{R})^{\circ}$ is a lattice in this latter connected Lie group. Moreover, since the map $p:\mathrm{Res}_{k/\mathbb{Q}}\mathbf{G}(\mathbb{R})^{\circ}\rightarrow G$ has compact kernel, we get that $p(\mathrm{Res}_{k/\mathbb{Q}}\,\mathbf{G}(\mathbb{Z}))$ and the commensurable group $\Gamma$ are lattices.

Since $G$ is assumed to be simple, it follows that the group $\mathbf{G}$ is \textit{admissible} (for the Lie group $G$), in the sense that $\mathrm{Res}_{k/\mathbb{Q}}\mathbf{G}(\mathbb{R})^{\circ}$ has only one non-compact factor in its decomposition as an almost direct product of $\mathbb{R}$-simple groups, and this factor is isogenous to $G$. By fixing an embedding $\sigma_0:k \to \mathbb{C}$ corresponding to the non-compact factor of $\mathrm{Res}_{k/\mathbb{Q}}\,\mathbf{G}(\mathbb{R})$ as in Section~\ref{sec:restriction of scalars}  and identifying the group $\mathbf{G}$ with $\mathbf{G}^{\sigma}$, we may reformulate the above definition by declaring $\Gamma<G$ to be an arithmetic lattice if there exist an algebraic number field $k\subset \mathbb{C}$ with ring of integers $\mathcal{O}$ and a connected, $k$-simple, admissible algebraic $k$-group $\mathbf{G}$ such that $\Gamma$ is commensurable with $\mathbf{G}(\mathcal{O})$ under the isogeny $\mathbf{G}(\mathbb{R})^{\circ}\xrightarrow{i} G$ (if $k \subset \mathbb{R}$) or $\mathbf{G}(\mathbb{C})\xrightarrow{i} G$ (if $k \subset \mathbb{C}$ is a complex field). In this setting, the number field $k$ and the algebraic $k$-group $\mathbf{G}$ are called the \textit{field of definition} and \textit{group of definition} of the arithmetic lattice $\Gamma$. The embedding $\sigma_0:k \to \mathbb{C}$ is referred to as the \textit{identity embedding}: it is uniquely defined if $\sigma(k)\subset \mathbb{R}$, while it is only defined up to complex conjugation if $\sigma(k)\not \subset \mathbb{R}$. Unless otherwise stated, we will identify the field of definition with its image under the identity embedding.

Notice that the latter case in which $k$ is a complex field can only occur if $G$ is isogenous to a complex Lie group. Moreover when this is the case the field $k$ is forced to have a single complex place, i.e.\ a single pair of complex-conjugate field embeddings. Indeed a complex algebraic group is compact if and only if it is finite  \cite[p.~134, Problem~3]{OV}, therefore if $k$ were to have more than one complex place there would be more than one non-compact factors in the group $\mathrm{Res}_{k/\mathbb{Q}}\mathbf{G}(\mathbb{R})^{\circ}$ and $G$ would not be simple.

By Godement's compactness criterion we have that an arithmetic lattice $\Gamma \sim \mathbf{G}(\mathcal{O})$ is uniform if and only if $\mathbf{G}(\mathcal{O})$ contains no nontrivial unipotent elements. This condition is automatically fulfilled if $\mathbf{G}^{\sigma}(\mathbb{R})$ is compact for some embedding $\sigma:k \rightarrow \mathbb{R}$. Therefore in order for a lattice $\Gamma$ to be non-uniform it is necessary that there are no compact factors in the group $\mathrm{Res}_{k/\mathbb{Q}}\mathbf{G}(\mathbb{R})^{\circ}$. It follows that a necessary condition for $\Gamma$ to be non-uniform is that $k=\mathbb{Q}$ (if $k\subset \mathbb{R}$) or that $k$ is an imaginary quadratic extension of the field of rational numbers (if $k\subset \mathbb{C}$ is a complex field). In general these conditions are not sufficient to guarantee that $\Gamma$ is non-uniform. 

If $\mathbf{G}$ is an algebraic $k$-group as above and $\Gamma<G$ is commensurable with $i(\mathbf{G}(\mathcal{O}))$, then the commensurator $\mathrm{Comm}_G(\Gamma)$ is easily seen to contain the group $i(\G(k))$, and this group is dense in $G$.
Moreover, there is the following arithmeticity criterion by Margulis: if
$\Gamma$ is an irreducible lattice in a connected, semisimple real Lie group $G$, then $\Gamma$ is arithmetic if and only if $\mathrm{Comm}_G(\Gamma)$ is dense in $G$ \cite[Theorem 9]{Margulis77}.

 If $\Gamma<G$, where $G$ is a simple Lie group with a finite number of connected components, we say that $\Gamma$ is arithmetic if $\Gamma^{\circ}=\Gamma \cap G^{\circ}$ is arithmetic. Notice that $\Gamma$ and $\Gamma^{\circ}$ are necessarily commensurable.
If $\Gamma<G$ is non-arithmetic, then $\Comm_G(\Gamma)$ is the maximal (by subgroup inclusion) lattice of $G$ containing $\Gamma$ (see \cite[Theorem 9]{Margulis77} and \cite[Chapter IX, Theorem~B \& Proposition~4.\,22]{Margulis-book}).
The celebrated Margulis' Arithmeticity Theorem \cite{Mar84} states that if $\Gamma$ is an irreducible lattice in a semi-simple Lie group $G$ with $\rk_\R G \ge 2$, then $\Gamma$ is arithmetic.

Now, let us assume that the non-compact simple Lie group $G$ is \textit{algebraic}, in the sense that $G$ is isomorphic to the group of real points of some connected real algebraic group $\mathbf{A}$. We may thus equip $G$ with the Zariski topology on $\mathbf{A}(\mathbb{R})$, i.e.\ a subset $C\subset G$ is closed if it corresponds to the zero-locus of some set of polynomial functions with real coefficients on $\mathbf{A}(\mathbb{R})$. By Borel's Density Theorem \cite{Borel-density}, if $\Gamma<G$ is a lattice, then $\Gamma$ is Zariski-dense in $G$.

The conclusions of Borel's density Theorem can be strengthened: $\Gamma$ is also dense in the Zariski topology of $\mathbf{A}$ considered as a complex algebraic set. To see this, we first notice that $\mathbf{A}(\mathbb{R})$ consists only of smooth points, and is thus Zariski-dense in $\mathbf{A}(\mathbb{C})$. Therefore, any $\mathbb{R}$-polynomial function that vanishes on $\Gamma$ vanishes on all of $\mathbf{A}(\mathbb{C})$.

To conclude, we apply the following lemma. It might be well known to the experts but we could not find it in the literature, so we include a proof. 

\begin{lemma}\label{lem:R-density-implies-C-density}
Let $\mathcal{X}$ be an affine algebraic variety defined over $\mathbb{R}$, and $Y \subset \mathcal{X}(\mathbb{R})$ a Zariski-dense subset of $\mathcal{X}(\mathbb{C})$. Then $Y$ is $\mathbb{C}$-Zariski dense in $\mathcal{X}$.
\end{lemma}

\begin{proof}
Let $p\in \mathbb{C}[x_1,\dots,x_n]$ be a polynomial with complex coefficients such that $p(y_1,\dots,y_n)=0$ for all $(y_1,\dots,y_n) \in Y$. We must prove that $p$ evaluates to $0$ on $\mathcal{X}$.

Denote by $\overline{p}$ the polynomial obtained from $p$ by applying complex conjugation to all its coefficients. Since $Y$ is a subset of $\mathcal{X}(\R)$, it follows that $0=\overline{p}(\overline{y_1},\dots \overline{y_n})=\overline{p}(y_1,\dots,y_n)$ for all $(y_1,\dots,y_n) \in Y$ , i.e.\ also $\overline{p}$ evaluates to $0$ on $Y$.
It follows that both polynomials 
$g_1=p+\overline{p}$ and $g_2=p\cdot\overline{p}$
evaluate to $0$ on $Y$. Moreover both $g_1$ and $g_2$ are invariant under complex conjugation, and thus belong to $\mathbb{R}[x_1,\dots,x_n]$. Since $Y$ is Zariski closed in $\mathcal{X}$ we have that $g_1$ and $g_2$ evaluate to $0$ on all of $\mathcal{X}$.

This means that $p$ and $\overline{p}$, now considered as regular functions on $\mathcal{X}$, satisfy $p=-\overline{p}$ and $p\cdot \overline{p}=0$, implying that $-p^2=0$. It follows that $p=0$, i.e.\ $p$ evaluates to $0$ on $\mathcal{X}$.
\end{proof}

\subsubsection{Vinberg's commensurability invariants}
We now recall Vinberg's construction of commensurability invariants for Zariski-dense subgroups of a semisimple algebraic group. Let $\mathbf{A}$ be an algebraic group over an algebraically closed field $F$, and $\Gamma < \mathbf{A}(F)$ a subgroup. A field $k < F$ is a \textit{field of definition} for $\Gamma$ if there exists a 
basis $\mathcal{B}$ for the Lie algebra $\mathfrak{a}$ of $\mathbf{A}$ such that the image $\mathrm{Ad}\,\Gamma$ of the adjoint action of $\Gamma$ on $\mathfrak{a}$ is represented with respect to $\mathcal{B}$ by matrices with coefficients in $k$.

\begin{theorem}[Vinberg \cite{Vin71}]\label{teo:Vinberg-invariants}
Let $F$ be an algebraically closed field of characteristic zero, $\mathbf{A}$ a semisimple algebraic group over $F$, and $\Gamma<\mathbf{A}$ a Zariski dense subgroup.

\begin{enumerate}
\item There exists a smallest field of definition for $\Gamma$, given by the field $$k = \Q(\{\tr(\mathrm{Ad}\,\gamma) \mid \gamma \in \Gamma\}),$$ where $\tr(\mathrm{Ad}\,\gamma)$ denotes the trace of the adjoint action of $\gamma$ on $\mathfrak{a}$. It is an invariant of the commensurability class of $\Gamma$;
\item The Zariski closure of $\{\mathrm{Ad}\, \gamma \mid \gamma \in \Gamma\}$ in $\mathbf{GL}(\mathfrak{a})$ is an algebraic $k$-group $\mathbf{G}$. The group $\mathrm{Ext}_F(\mathbf{G})$ obtained from $\mathbf{G}$ by extending scalars from $k$ to $F$ is isomorphic to the adjoint group $\mathbf{PA}$ of $\mathbf{A}$ (i.e.\ $\mathbf{G}$ is a $k$-form of $\mathbf{PA}$), and the adjoint image of $\Gamma$ in $\mathbf{PA}$ is contained in the group $\mathbf{G}(k)$;
\item The group $\mathbf{G}^{\circ}$ is uniquely determined up to $k$-isomorphism by the commensurability class of $\Gamma$.
\end{enumerate}
\end{theorem}
We will refer to the pair $(k,\mathbf{G})$ in the statement of Theorem \ref{teo:Vinberg-invariants} as the \textit{Vinberg invariants} of (the commensurability class of) $\Gamma$.

If $\Gamma$ is a lattice in a connected, simple, noncompact algebraic Lie group $G=\mathbf{A}(\mathbb{R})$, we have seen that $\Gamma$ is Zariski dense in $\mathbf{A}$, viewed as a complex algebraic set. We can then apply Theorem \ref{teo:Vinberg-invariants} and define the Vinberg invariants of $\Gamma$ which will be called the \textit{adjoint trace field} $k$ and the \textit{real ambient group} $\mathbf{G}$, respectively. 

If, moreover, the group $G$ has the structure of a complex Lie group (which will happen whenever the group $\mathbf{A}$ is not absolutely simple) we may assume that $G=\mathbf{B}(\mathbb{C})$ with $\mathbf{B}$ a complex algebraic group such that $\mathrm{Res}_{\mathbb{C}/\mathbb{R}}\, \mathbf{B}$ is $\mathbb{R}$-isomorphic to $\mathbf{A}$. It follows that $\Gamma$ is also Zariski dense in $\mathbf{B}$, and we may construct the Vinberg invariants of $\Gamma$ as a subgroup of $\mathbf{B}$. The reader should be careful not to confuse the Vinberg invariants of $\Gamma<\mathbf{A}$ with the Vinberg invariants of $\Gamma<\mathbf{B}$: these are distinct fields and distinct algebraic groups. In order to avoid confusion we will use a different terminology (and notation) for the Vinberg invariants of a lattice in a  complex algebraic group, consistent with the one typically used for lattices in $\mathrm{PSL}_2(\mathbb{C})$: the field $L$ will be referred to as the \textit{invariant trace field} and the algebraic $L$-group will be referred to as the \textit{complex ambient group}.

Finally, we remark that if $\Gamma<G$ is an arithmetic lattice in an algebraic, absolutely simple Lie group, then by \cite[Lemma 2.6]{PR} the algebraic number field $k\subset \mathbb{C}$ and the admissible algebraic $k$-group $\mathbf{G}$ such that $\Gamma \sim \mathbf{G}(\mathcal{O})$ are precisely the Vinberg invariants of $\Gamma$, and are therefore commensurability invariants. In particular if $G=\mathbf{B}(\mathbb{C})$ is a complex Lie group, we have that $k\subset \mathbb{C}$ is a complex field which coincides with the invariant trace field of $\Gamma$, and $\mathbf{G}$ is the complex ambient group of $\Gamma$. If $\mathbf{G}=\mathbf{A}(\mathbb{R})$ with $\mathbf{A}$ an absolutely simple algebraic $\mathbb{R}$-group, then $k \subset \mathbb{R}$ is the adjoint trace field of $\Gamma$ and $\mathbf{G}$ is its real ambient group.

\subsection{Hyperbolic lattices}\label{sec:hyperbolic-lattices}

We denote by $\HH^n$ the hyperbolic space, which is the unique simply connected complete Riemannian $n$-manifold with constant sectional curvature $-1$. The \textit{hyperboloid model} $\HH^n$ for hyperbolic space is defined as follows.

Consider the real vector space $\mathbb{R}^{n+1}$ equipped with the standard quadratic form $f$ of signature $(n,1)$: 
$$f(\mathbf{x})=-x_0^2+x_1^2+\dots+x_n^2.$$

Let $\mathfrak{H}$ be the hyperboloid $$\mathfrak{H}=\{\mathbf{x} \in \mathbb{R}^{n+1} \,|\, f(\mathbf{x}) = -1 \} = \mathfrak{H}^+ \cup \mathfrak{H}^-,$$
where
$$
\mathfrak{H}^+ = \{\mathbf{x} \in \mathfrak{H} \,|\, x_0 > 0\} \text{ and } \mathfrak{H}^- = \{\mathbf{x} \in \mathfrak{H} \,|\, x_0 < 0\}.
$$

By equipping $\mathfrak{H}^+$ with the Riemannian metric induced by restricting $f$ to each tangent space $T_p(\mathfrak{H}^+)$, $p \in \mathfrak{H}^+$, we obtain the hyperboloid model for the $n$-dimensional hyperbolic space $\mathbb{H}^n$.

Let $\mathrm{O}_{n,1} = \mathbf{O}_f(\R)$ be the orthogonal group of $f(\mathbf{x})$, $\mathrm{PO}_{n,1} = \mathbf{PO}_f(\R)$ be its adjoint group, and $\mathrm{O}^+_{n,1} < \mathrm{O}_{n,1}$ be the subgroup preserving $\mathfrak{H}^+$. Thus, we can identify $\mathrm{Isom}(\mathbb{H}^n)$ with $\mathrm{O}^+_{n,1}$.

Notice that $\mathrm{O}^+_{n,1}$ is a simple real Lie group but is not realized as the group of $\mathbb{R}$-points of an algebraic group, since it is not an algebraic subgroup of $\mathrm{O}_{n,1}$. However, $\mathrm{O}^+_{n,1}$ is isomorphic (as a Lie group) to the group of $\mathbb{R}$-points of the adjoint, $\mathbb{R}$-simple real algebraic group $\mathbf{PO}_{n,1}=\mathbf{O}_{n,1}/\{\pm \mathrm{id}\}$, where $\mathbf{O}_{n,1} = \mathbf{O}_f$. The isomorphism $\mathrm{O}^+_{n,1} \cong \PO_{n,1}=\mathbf{PO}_{n,1}(\R)$ is obtained by restricting the adjoint homomorphism $\mathrm{Ad}:\mathrm{O}_{n,1}\rightarrow \mathrm{PO}_{n,1}$ to the subgroup $\mathrm{O}^+_{n,1}$. The algebraic group $\mathbf{O}_{n,1}$ has two connected components and $\mathbf{O}_{n,1}^\circ = \mathbf{SO}_{n,1}$. Then the adjoint algebraic group $\mathbf{PO}_{n,1}$ 
\begin{itemize}
    \item is connected, if $n$ is even ($-I \notin \mathbf{SO}_{n,1}$ and therefore $\mathbf{PO}_{n,1} \cong \mathbf{SO}_{n,1}$, with the isomorphism being induced by the map $\mathbf{O}_{n,1} \rightarrow \mathbf{SO}_{n,1}$ given by $M \mapsto \mathrm{det}(M)^{-1}\cdot M$),
    \item and has two connected components when $n$ is odd (since $-I \in \mathbf{SO}_{n,1}$); in this case its irreducible component is $\mathbf{PSO}_{n,1}$.
\end{itemize}

A {\it hyperbolic lattice} is then defined as a lattice in $\mathrm{PO}_{n,1}$. If $\Gamma < \mathrm{PO}_{n,1}$ is a hyperbolic lattice, the quotient $M=\mathbb{H}^n/\Gamma$ is a finite-volume {\itshape hyperbolic orbifold}. If $\Gamma$ is torsion-free, then $M$ is a Riemannian manifold, and is called a {\itshape hyperbolic manifold}.

If $\Gamma<\mathrm{PO}_{2m,1}$ is a lattice, then its real ambient group is always a $k$-form of $\mathbf{PO}_{n,1}\cong \mathbf{SO}_{n,1}$. On the other hand, if $\Gamma<\mathrm{PO}_{2m+1,1}$ is a lattice, then the real ambient group of $\Gamma$ is a $k$-form of $\mathbf{PSO}_{2m+1,1}$ (if $\Gamma<\mathrm{PSO}_{2m+1,1}$ consists entirely of orientation-preserving isometries) or of $\mathbf{PO}_{2m+1,1}$ (if $\Gamma$ contains orientation-reversing isometries).

Consider now the isometry group $G = \mathrm{PO}_{n,1}$ of the hyperbolic space $\HH^n$, and let $\G$ be an \textit{adjoint} admissible (for $G$) algebraic $k$-group. Then any subgroup $\Gamma$ commensurable with $\G(\mathcal{O})$ is an \textit{arithmetic hyperbolic lattice}. The commensurator $\mathrm{Comm}_G(\Gamma)$ of such  a lattice $\Gamma$ is precisely $\mathbf{G}(k)$, since $\mathbf{G}(\mathbb{C})\cong \mathbf{PO}_{n+1}(\mathbb{C})$ is centreless (see \cite[Remark 5.2.5]{WM} and \cite[Theorem 3(b)]{Borel-adjoint}).

Since the case $\rk_\R G = 1$ admits also non-arithmetic lattices, we introduce some weaker types of arithmeticity of hyperbolic lattices. 

\begin{definition}[Vinberg \cite{Vin67}]\label{def:QA}
A lattice $\Gamma < \mathrm{PO}_{n,1}$ is called \textit{quasi-arithmetic} if there exists an algebraic number field $k$ and an  adjoint, admissible (for $\mathrm{PO}_{n,1}$) algebraic $k$-group $\G$, such that $\Gamma$ is conjugate into $\G(k)$. A lattice $\Gamma$ is \textit{properly quasi-arithmetic} if it is  quasi-arithmetic, but not arithmetic.     
\end{definition}

In essence, a quasi-arithmetic lattice $\Gamma<\mathrm{PO}_{n,1}$ is a lattice contained in the commensurator of an arithmetic lattice. The number field $k$ and the algebraic $k$-group $\mathbf{G}$ are called the \textit{field of definition} and \textit{group of definition} of the quasi-arithmetic lattice $\Gamma$, thereby extending the definitions given in the arithmetic case.

The group $\mathbf{PO}_{n,1}$ is absolutely simple whenever $n\neq 3$, therefore if $\Gamma$ is a quasi-arithmetic lattice in $G=\mathbf{PO}_{n,1}(\mathbb{R})$, the field of definition $k$ is totally real, the group of definition $\mathbf{G}$ is admissible.

For $n=3$ we have the following sequence of $\CC$-isomorphisms of complex algebraic groups: $\mathbf{PSO}_{3,1}\cong \mathbf{PSO}_4 \cong \mathbf{PSO}_3 \times \mathbf{PSO}_3$ and therefore $\mathbf{PSO}_{3,1}$ is $\mathbb{R}$-simple but not absolutely simple.
However, the group $\mathbf{PSO}_{3,1}$ is isomorphic over $\mathbb{R}$ to $\mathrm{Res}_{\mathbb{C}/\mathbb{R}}\,\mathbf{PGL}_2$, and $\mathbf{PGL}_2$ is indeed absolutely simple. This translates to the well known exceptional Lie group isomorphism $\mathrm{PSO}_{3,1}(\mathbb{R})\cong \mathrm{PSL}_2(\mathbb{C})$. Due to this, there is a classical description of arithmetic lattices acting on $\mathbb{H}^3$ as arithmetic subgroups of $\mathbf{PGL}_2(\CC)$. We shall discuss the types of arithmetic lattices and some of their specific properties in more detail in Section \ref{sec:types-arithmetic-lattices}. 

From the general classification of semi-simple algebraic groups by Tits~\cite{Tits66}, it follows that there exist three distinct types of (quasi)-arithmetic groups in $\PO_{n,1}$: 
\begin{itemize}
    \item type I, associated with quadratic forms (in all dimensions $n\geq 2$),
    \item type II, associated with unitary groups of skew-Hermitian forms with coefficients in quaternion algebras (in odd dimensions $n\geq 3$),
    \item and type III, related to the exceptional isomorphism in dimension $n=3$ (discussed above) or to the triality phenomenon in dimension $n=7$.
\end{itemize}

As for adjoint trace fields and fields of definition of quasi-arithmetic lattices, the following fact is widely used, although we did not find it in the literature. We provide a simple argument for the reader's convenience.

\begin{prop}\label{prop:adj-trace-of-QA}
    Let $\Gamma < \G(k)$ be a quasi-arithmetic lattice acting on $\mathbb{H}^n$, as in Definition~\ref{def:QA}. If $k$ is totally real then $k$ is the adjoint trace field of $\Gamma$ and $\G$ is its real ambient group. If $n=3$ and $k$ has one complex place, then $k$ is the invariant trace field of $\Gamma$ and $\G$ is its complex ambient group. 
\end{prop}
\begin{proof}
Let $k'$ be the adjoint trace field of $\Gamma < \G(k)$, and $\G'$ be its ambient $k'$-group. Denote by $\Gamma' \cong \Gamma$ the isomorphic image of $\Gamma$ such that $\Gamma' < \G'(k')$. The field $k$ is also a field of definition for $\Gamma$, so Vinberg's Theorem~\ref{teo:Vinberg-invariants} implies $k' \subset k$. 

Now we argue by contradiction, assuming that $k'$ is a proper subfield of $k$. If $k$ is totally real, than both groups $\mathbf{G}'$ and $\mathbf{G}$ are forms of $\mathbf{PO}_{n,1}$ and are therefore isomorphic over $\mathbb{R}$. We claim that they are isomorphic over $k$ (i.e., that the group $\mathrm{Ext}_{k/k'}\,\mathbf{G}'$ is $k$-isomorphic to $\mathbf{G}$). Indeed, the $\mathbb{R}$-isomorphism $\phi: \mathbf{G}'\rightarrow \mathbf{G}$ can be chosen so that it maps $\Gamma'<\mathbf{G}'(k')<\mathbf{G}'(k)$ to $\Gamma<\mathbf{G}(k)$. Therefore, $\phi$ commutes with the action of the absolute Galois group $\mathcal{G}=\mathrm{Gal}(\mathbb{C}/k)$ on $\Gamma$. Since $\Gamma'$ is Zariski-dense in $\mathbf{G}'$, we have that $\phi$ commutes with the $\mathcal{G}$-action on all of $\mathbf{G}'$ (i.e. $\phi$ is $\mathcal{G}$-equivariant), and thus gives rise to a $k$-defined isomorphism $\phi: \mathrm{Ext}_{k/k'}\,\mathbf{G}' \to \G$.

Now, let $\sigma:k\rightarrow \mathbb{C}$ be any non-trivial field embedding of $k$ that restricts to the identity on $k'$. We have that $\mathbf{G}^{\sigma}(\mathbb{R})\cong \mathbf{G}(\mathbb{R})$ is non-compact, contradicting the admissibility of $\mathbf{G}$.   

If $k\subset \mathbb{C}$ but $k \not\subset \R$, then $\mathbf{G}$ and $\mathbf{G}'$ are forms of $\mathbf{PGL}_2$ and are therefore isomorphic over $\mathbb{C}$. The same argument used in the totally real case shows that $\mathbf{G}$ and $\mathbf{G}'$ are isomorphic over $k$. Since $k$ has one pair of complex conjugate embeddings and $k'$ is a proper subfield of $k$, we see that $k'$ is totally real. It follows that $\Gamma<\mathbf{G}'(k')$ is at the same time a lattice in $\mathbf{PGL}_2(\mathbb{C})$ and a subgroup of $\mathbf{G}'(\mathbb{R})$ which, being the group of $\mathbb{R}$-points of an $\mathbb{R}$-form of $\mathbf{PGL}_2$, is either isomorphic to $\mathbf{PGL}_2(\mathbb{R})\cong \mathbf{SO}_{2,1}(\mathbb{R})$ or to the compact group $\mathbf{PSU}_2(\mathbb{R})\cong \mathbf{SO}_3(\mathbb{R})$. In both cases we see that $\Gamma$ is contained in a proper algebraic subgroup of $\mathrm{Res}_{\mathbb{C}/\mathbb{R}}\,\mathbf{PGL}_2 \cong \mathbf{PSO}_{3,1}$, which gives us a contradiction with Borel's Density Theorem.
\end{proof}

\noindent Proposition \ref{prop:adj-trace-of-QA} shows that if $\Gamma' < G$ is commensurable with some quasi-arithmetic lattice $\Gamma < \G(k)$, then $\Gamma'$ is also a subgroup of $\G(k)$ and therefore quasi-arithmetici\-ty is an invariant property of commensurability classes of hyperbolic lattices.

Let us stress the fact that the notion of quasi-arithmeticity is distinct from the usual arithmeticity only for Lie groups of real rank $1$, due to Margulis' Arithmeticity Theorem \cite{Mar84}. Moreover, the superrigidity results of Corlette~\cite{Cor92} and Gromov--Schoen \cite{GS92} for $\mathrm{F}_4^{-20}$ and $\mathrm{PSp}_{n,1}$ together with the fact that for lattices in $\mathrm{PU}_{n,1}$, $n\ge 2$, traces are always integral (as proven in \cite[Theorem 1.5]{BFMS23} or \cite[Theorem 1.3.1]{BU23} based on work of Esnault--Groechenig \cite[Theorem~1.1]{EG18}) imply that, up to isogeny, the groups $\PO_{n,1}$ are the only non-compact semi-simple real Lie groups containing properly quasi-lattices lattices. Vinberg \cite{Vin67} constructed first properly quasi-arithmetic lattices in $\Isom(\HH^n)$ for low dimensions $2 \le n \le 5$ using reflection groups. On the other hand, Agol's construction \cite{Ag06}, extended by Belolipetsky--Thomson \cite{BT11} and Bergeron--Haglund--Wise \cite{BHW11}, yields properly quasi-arithmetic lattices in $\Isom(\HH^n)$ for all $n \ge 2$, as observed by Thomson \cite{Tho16}.

Furthermore, a lattice $\Gamma < \G(\R)$ is called \textit{pseudo-arithmetic} over $K/k$ if $\Gamma < \G(K)$, $\G$ is admissible over $k$, and $K$ is a multiquadratic extension of $k$ (i.e. $K = k(\sqrt{a_1}, \ldots, \sqrt{a_m})$, for some $a_1, \ldots, a_m \in k$). All the currently known examples of lattices in  $\mathrm{PO}_{n,1}$ for $n>3$ are pseudo-arithmetic~\cite{EM20}.

\subsubsection{Totally geodesic subspaces}

Let $M=\mathbb{H}^n/\Gamma$ and $N=\mathbb{H}^m/\Lambda$, $m < n$, be finite-volume hyperbolic orbifolds. A map $i:N \rightarrow M$ is a {\itshape totally geodesic immersion} if any of its lifts $\widetilde{i}:\mathbb{H}^m \rightarrow \mathbb{H}^n$ maps $\mathbb{H}^m$ isometrically to an $m$-dimensional totally geodesic subspace $U\subseteq \mathbb{H}^n$. Equivalently, there is an $m$-dimensional totally geodesic subspace $U\subseteq \mathbb{H}^n$ such that its stabiliser $\mathrm{Stab}_{\Gamma}(U)< \mathbf{PO}_{m,1}(\R) \times \mathbf{O}_{n-m}(\R)$ acts as lattice on $U$, and its projection to $\mathbf{PO}_{m,1}(\R)$ is conjugate to $\Lambda$. In this setting, we say that $N$ is a {\itshape totally geodesic subspace} of $M$, and that $\Lambda$ is a {\itshape totally geodesic sublattice} of $\Gamma$.

If the map $i:N \rightarrow M$ above is an embedding (i.e.\ all the lifts of $N$ to $\mathbb{H}^n$ are pairwise disjoint) we say that $i$ is a {\itshape totally geodesic embedding},  $N$ is a {\itshape totally geodesic embedded subspace} of $M$, and that $\Lambda$ {\em embeds geodesically} in $\Gamma$.

\section{Arithmetic hyperbolic lattices and involutions}\label{sec:types-arithmetic-lattices}
\subsection{Arithmetic lattices of type I}\label{sec:type1}

Arithmetic lattices of type I in $\HH^n$ are also called arithmetic lattices of simplest type, and correspond to the following indices in Tits' classification \cite{Tits66}: $B_{m}$, if $n$ is even, and either ${^1}D^{(1)}_{m}$ or ${^2}D^{(1)}_m$, if $n = 2m-1$ is odd. Here ${^1}D^{(1)}_{m}$ can only occur if the field of definition is $\mathbb{Q}$ and $m$ is odd \cite[Proposition 13.6]{Emery}. 

\subsubsection{Admissible quadratic forms}

Let $k \subset \R$ be a totally real number field with the ring of integers $\OOO$. Let $f$ be a quadratic form defined over $k$. We say that $f$ is {\itshape admissible} if it has signature $(n, 1)$ and for any non-identity field embedding $\sigma: k \rightarrow \mathbb{R}$ the form $f^{\sigma}$ is positive definite.
The fact that $f$ is admissible as defined above clearly implies that the $k$-defined algebraic group $\mathbf{O}_f$ is admissible for $\mathrm{O}_{n,1}$. It follows that the group $\mathrm{PO}_f(\mathcal{O})$ is an arithmetic lattice in $\mathrm{PO}_{n,1}$ (see \cite{BHC62}).

A quadratic form $f$ defined over $k$ is called \textit{anisotropic} if it does \textit{not} represent $0$ over $k$. Otherwise $f$ is called \textit{isotropic}. It is easy to see that any admissible quadratic form $f$ defined over $k \neq \mathbb{Q}$ is anisotropic. Meyer's theorem \cite[Chapter 3.2, Corollary 2]{Serre} implies that any quadratic form defined over $\mathbb{Q}$ of signature $(n, 1)$ with $n\geq 4$ is isotropic. 

\begin{definition}\label{def:type-I}
Any group $\Gamma$ obtained as $\mathrm{PO}_f(\mathcal{O})$ in the way described above and any subgroup of $\mathrm{PO}_{n,1}\cong\mathbf{PO}_f (\mathbb{R})$ commensurable to such a group in the wide sense, is called an arithmetic lattice of simplest type, or \textit{an arithmetic lattice of type I}. The field $k$ is called the \textit{field of definition} of $\Gamma$. A hyperbolic orbifold $M=\mathbb{H}^n/\Gamma$ is said to be arithmetic of type~I if the group $\Gamma$ is an arithmetic lattice of type~I.

In the same setting, if $\Gamma<\mathrm{PO}_{n,1}$ is a lattice commensurable in the wide sense with a subgroup of $\mathbf{PO}_f(k)$ we say that $\Gamma$ and the orbifold $M=\mathbb{H}^{n}/\Gamma$ are \textit{quasi-arithmetic of type~I}.
\end{definition}

\begin{remark}
Given a totally real field $k$, two admissible, $k$-defined quadratic forms $f_1$ and $f_2$ define the same commensurability class of arithmetic hyperbolic lattices if and only if $f_1$ is equivalent over $k$ to $\lambda \cdot f_2$ for some $\lambda \in k^\times$. 
\end{remark}

Godement's compactness criterion \cite{God62} implies that $\Gamma$ is uniform if and only if $f$ is anisotropic. Thus, for any totally real number field $k \neq \Q$, the resulting orbifold $\HH^n/\Gamma$ is compact. If $k = \Q$, then the orbifold $\HH^n/\Gamma$ is compact only if the quadratic form $f$ does \textit{not} rationally represent $0$. This can only happen if $f$ has signature $(n,1)$ with $n \leq 3$, by Meyer's theorem. 

\begin{remark}
The classification of semi-simple algebraic groups by Tits~\cite{Tits66} implies that all arithmetic lattices acting on $\mathbb{H}^{2n}$, $n\geq 1$, are of simplest type. The same applies to all non-uniform arithmetic lattices acting on $\mathbb{H}^{n}$, for $n\geq 2$. 
\end{remark}

\subsubsection{Constructing $k$-involutions}\label{sec:k-involutions-type1} We wish to build examples of fc-subspa\-ces associated to a single involution in $\mathrm{Comm}(\Gamma)$, where $\Gamma<\mathrm{Isom}(\mathbb{H}^n)$ is an arithmetic hyperbolic lattice. If $k$ is the adjoint trace field of $\Gamma$ and $\mathbf{G}$ is its ($k$-defined, admissible) ambient group, then $\mathrm{Comm}(\Gamma)=\mathbf{G}(k)$, so we are left with the task of analysing the fixed-point set of involutions in $\mathbf{G}(k)$.

Here we examine the case of type-I lattices: $\mathbf{G}=\mathbf{PO}_f$, where $f$ is an admissible Lorentzian quadratic form of signature $(n,1)$ over a totally real field $k$.

As a first step, we characterize all involutions in the group of $k$-points $\mathbf{O}_f(k)$.
Let us fix a subspace $V_1 \subset k^{n+1}$ of dimension $m+1$ such that $f_{|V_1}$ has signature $(m, 1)$, and let $V_1^{\perp} = V_{-1}$ be its orthogonal complement with respect to the form $f$. Clearly $k^{n+1}$ decomposes as a direct sum of $V_1$ and $V_{-1}$. Let $N$ denote the linear transformation which acts as the identity on $V_1$ and as multiplication by $-1$ on $V_{-1}$. It is easy to verify that $N$ belongs to $\mathrm{O}_f(k)$ and that $N^2 = \mathrm{id}$.

Also, all order-$2$ elements in $\mathrm{O}_f(k)$ arise through the above construction. Suppose that $N \in \mathrm{O}_f(k)$ has order $2$. Then $N$ has eigenvalues $1$ and $-1$ and the corresponding eigenspaces $V_1,V_{-1} \subset k^{n+1}$ are orthogonal with respect to the form $f$. Up to multiplication by $\pm \mathrm{id}$, we may assume that the eigenspace $V_1$ intersects the upper sheet $\mathbb{H}^n$ of the hyperboloid. In this case, the restriction $f_{\mid V_1}$ of $f$ to $V_1$ will be an admissible Lorentzian quadratic form of signature $(m,1)$, where $\mathrm{dim}(V_1)=m+1$. If $\Gamma<\mathbf{PO}_f$ is an arithmetic lattice (i.e.\ $\Gamma$ is commensurable with $\mathrm{PO}_f(\mathcal{O})$) and $i=[N]\in \mathbf{PO}_f(k)$ is an involution with $N \in \mathrm{O}_f(k)$ as above, then $\mathrm{Stab}_{\Gamma}(V_1)$ corresponds to a totally geodesic type-I arithmetic subspace of $\mathbb{H}^n/\Gamma$ associated to the admissible $k$-form $g=f_{|V_1}$. 

If $n$ is even, then $\mathbf{PO}_f$ is $k$-isomorphic to $\mathbf{SO}_f=\mathbf{O}_f \cap \mathbf{SL}_{n+1}$, and therefore the latter is adjoint. The proof of \cite[Lemma~4.2]{ERT} then shows that all $k$-involutions in $\mathbf{PO}_f(k)$ correspond to involutions in $\mathrm{PO}_f(k)$.  For odd $n\geq 3$, there may be involutions in $\mathbf{PO}_f(k)$ that do not belong to $\mathrm{PO}_f(k)$. Due to \cite[Lemma~4.3]{ERT}, we have a complete description of all involutions in this case, too. 

Let us denote by $A$ the matrix which represents the form $f$ with respect to the standard basis of $k^{n+1}$, and define the {\itshape general orthogonal group} of the form $f$ as $$\mathbf{GO}_f = \{ M \in \mathrm{GL}_{n+1}(\mathbb{C}) \,|\, M^{t} A M = \mu \cdot A, \text{ for some } \mu \in \mathbb{C}\}.$$ An involution $N \in \mathbf{PO}_f(k)$ can be represented by a linear transformation $M \in \mathbf{GO}_f(k)$ such that $M^2 = \mu \cdot \mathrm{id}$. Here, $M$ is defined up to multiplication by $\lambda \in k^{\times}$. Notice that we always have $0 < \mu \in k^\times$. Moreover, $\mu$ is totally positive in the sense that $\sigma(\mu)$ is positive for all field embeddings $\sigma:k \rightarrow \mathbb{R}$.  

Since we consider the projective model of $\mathbb{H}^n$, we may identify such $N \in \mathbf{PO}_f(k)$ directly with $M$, as they describe the same projective transformation with the same (projective) fixed point set. The matrix $M$ has eigenvalues $\pm \sqrt{\mu}$. The eigenspace $V_{\sqrt{\mu}}$ for the eigenvalue $\sqrt{\mu}$ of $M$ is called the \textit{positive eigenspace}. Similarly, the eigenspace $V_{-\sqrt{\mu}}$ for the eigenvalue $-\sqrt{\mu}$ of $M$ is called the \textit{negative eigenspace}. 

Let $K$ denote the field $k(\sqrt{\mu})$, and let us assume that the number $\mu$ is not a square, so that $[K:k]=2$. Notice that $K$ is totally real due to the fact that $\mu$ is totally positive. The positive and negative eigenspaces are defined over $K$, in the sense that they are described as the set of solutions of a homogeneous linear system of equations with coefficients in $K$. Up to multiplication by $-\mathrm{id}$, we can assume that the restriction $g$ of the form $f$ to the positive eigenspace has signature $(m,1)$, for some $m>0$, while the restriction $h$ of $f$ to the negative eigenspace is positive definite. Notice that the forms $g$ and $h$ are also defined over $K$.

We now claim that the form $g$ is admissible, and that if $\Gamma$ is an arithmetic lattice in $\mathbf{O}_f(\mathbb{R})$, the stabiliser in $\Gamma$ of the positive eigenspace projects to an arithmetic lattice in $\mathbf{O}_g(\mathbb{R})$. The claim follows by studying the restriction of scalars $\mathrm{Res}_{K/k}\,\mathbf{O}_g$ of the orthogonal group $\mathbf{O}_g$, and proving that it is isomorphic to the centraliser $\mathbf{H}$ of $N$ in $\mathbf{O}_f$. Notice that that $\mathbf{H}$ is the fixed point set of the $k$-automorphism of $\mathbf{O}_f$ given by conjugation by $N$, and so is indeed a $k$-subgroup of $\mathbf{O}_f$. An element of $\mathbf{O}_f$ commutes with $N$ if and only if it preserves the positive and negative eigenspaces, and therefore $\mathbf{H}$ is $K$-isomorphic to $\mathbf{O}_g \times \mathbf{O}_h$. By the discussion in Section \ref{sec:restriction of scalars} we have the following $K$-defined isomorphism:
\begin{equation*}
    \mathrm{Res}_{K/k}\,\mathbf{O}_g\cong \mathbf{O}_g \times \mathbf{O}_{g^{\sigma}},
\end{equation*}
where $g^{\sigma}$ is the form obtained by applying the non-trivial automorphism $\sigma \in \mathrm{Gal}(K/k) \cong \mathbb{Z}_2$ to the coefficients of $g$. We note that $g$ is defined as the restriction of the form $f$ to the positive eigenspace for $N \in \mathbf{GL}_{n+1}(K)$. The Galois automorphism $\sigma$ sends $N=1/\sqrt{\mu} M$ to $N^{\sigma}=-1/\sqrt{\mu} M=-N$, therefore exchanging the positive and negative eigenspaces. Since the form $f$ is defined over $k$, we have that $f^{\sigma}=f$. We may thus identify the form $g^{\sigma}$ with the (positive definite) restriction $h$ of the form $f$ to the negative eigenspace of $N$ so that the group $\mathbf{O}_{g^{\sigma}}$ is $K$-isomorphic to $\mathbf{O}_h$. We remark that the discussion above implies that the positive and negative eigenspaces of $N$ have the same dimension, equal to $(n+1)/2$.

We have thus established that 
\begin{equation}\label{eq:type-I-centraliser}
    \mathrm{Res}_{K/k}\,\mathbf{O}_g\cong \mathbf{O}_g \times \mathbf{O}_h \cong \mathbf{H}
\end{equation} where all isomorphisms are defined over $K$.
Now note that the group of $k$-points $\mathbf{H}_k=\mathbf{GL}_{n+1}(k)\cap \mathbf{H}$ corresponds to pairs of the form $(M,M^{\sigma})$, where $M \in \mathbf{O}_g(K)$ and $M^{\sigma}\in \mathbf{O}_h(K)$ is obtained from $M$ by applying to all of its entries the Galois automorphism $\sigma$. These are precisely the $k$-points of $\mathrm{Res}_{K/k}\,\mathbf{O}_g$. As such the isomorphism \eqref{eq:type-I-centraliser} induces a one-to-one correspondence between the Zariski dense subgroups $\mathrm{Res}_{K/k}\,\mathbf{O}_g(k)$ and $\mathbf{H}_k$, implying that 
$\mathrm{Res}_{K/k}\,\mathbf{O}_g$ and $\mathbf{H}$ are $k$-isomorphic.

Therefore if $\Gamma < \mathbf{O}_f(k)$ is an arithmetic lattice, then $$\mathrm{Stab}_{\Gamma}\, V_1 = \Gamma \cap \mathbf{H}_k$$ is an arithmetic subgroup of $\mathbf{O}_g(\mathbb{R})\times \mathbf{O}_h(\mathbb{R})$. By modding out the compact factor $\mathbf{O}_h(\mathbb{R})$ we see that the projection $\Lambda$ of $\mathrm{Stab}_{\Gamma}\, V_1$ to $\mathbf{O}_g(\mathbb{R})$ is an arithmetic subgroup of $\mathbf{O}_g(\mathbb{R})$.

Finally, the admissibility of $g$ follows easily from the admissibility of~$f$. Indeed, for each non-identity embedding $\eta: k \rightarrow \mathbb{R}$, the group $(\mathrm{Res}_{K/k}\,\mathbf{O}_g)^{\eta}(\R)$ decomposes as the product of the orthogonal groups of the restrictions of the positive definite form $f^{\eta}$ to the positive and negative eigenspaces for $N^{\eta}$, and therefore all the resulting factors are compact. We thus see that the involution $[N] \in \mathbf{PO}_f(k)$ yields a type-I arithmetic fc-subspace of dimension $(n-1)/2$ with adjoint trace field $K$ and ambient group $\mathbf{PO}_g$.

We conclude this subsection with a remark that might have applications.

\begin{remark}
    The above discussion combined with the work \cite{KRS} shows that all odd-dimensional type-I arithmetic hyperbolic lattices $\Gamma < \mathrm{PO}_{2m+1,1}$ {\em generated by involutions with fixed point set of dimension $\ne m$} embed geodesically as codimension-one sublattices {\em without passing to a finite-index subgroup}. More generally, this holds true whenever $\Gamma$ lies in the image $\mathrm{PO}_f(k)$ of $\mathbf{O}_f(k)$ inside $\mathbf{PO}_f(k)$, see \cite[Propositions 2.1 and 4.1]{KRS}.
\end{remark}

\subsection{Arithmetic lattices of type II}\label{sec:type2}
\subsubsection{Quaternion algebras} Let $k$ be a field of characteristic $\neq 2$. A quaternion algebra over $k$ is a $4$-dimensional central simple algebra $D$. Any quaternion algebra is isomorphic to the algebra $D(a, b)$ with $\I^2 = a$, $\J^2 = b$, and $\I \J = -\J \I = \K$, for some choice of non-zero $a, b \in k$. These relations imply that $\K^2 = -ab$. Over any field $k$, the quaternion algebra $D(1,1)$ is isomorphic to the algebra $M_2(k)$ of $2 \times 2$ matrices with coefficients in $k$ through the following homomorphism:
\begin{eqnarray*}
1 \mapsto 
\begin{pmatrix}
1 & 0\\
0 & 1
\end{pmatrix},\; &
\I \mapsto 
\begin{pmatrix}
1 & 0\\
0 & -1
\end{pmatrix},\; &
\J \mapsto 
\begin{pmatrix}
0 & 1\\
1 & 0
\end{pmatrix},\; \
\K \mapsto 
\begin{pmatrix}
0 & 1\\
-1 & 0
\end{pmatrix}.
\end{eqnarray*}

As a corollary of Wedderburn's structure theorem \cite[Theorem 2.9.6]{Mac-Reid}, any quaternion algebra $D$ over $k$ is either isomorphic to the algebra $M_2(k)$ of $2 \times 2$ matrices (which always has zero divisors), or it is a division algebra (i.e. any non-zero element has a multiplicative inverse). 

In any quaternion algebra $D$ there is an involutory anti-automorphism $q \mapsto q^*$ (called the standard involution) whose fixed-point set coincides with the field $k$. It is obtained by multiplying $\I$, $\J$ and $\K$ by $-1$ and satisfies $(pq)^* = q^* p^*$, for all $p, q \in D$. If $D \cong M_2(k)$, the standard involution can be written as
\begin{equation}\label{eq:standard-involution-split}
  \begin{pmatrix}
    a & b \\
    c & d
    \end{pmatrix}
    \mapsto
   \begin{pmatrix}
    d & -b \\
    -c & a
    \end{pmatrix}. 
\end{equation}

For any element $q \in D$, its norm is defined as $N(q)=q q^* \in k$ and its trace is defined as $\mathrm{Tr}(q)=q+q^* \in k$. If $D\cong M_2(k)$, the norm and trace of an element are respectively the determinant and trace of the corresponding matrix. The set of invertible quaternions (i.e.\ the set of quaternions with non-zero norm) forms a multiplicative group which we denote by $\mathrm{GL}_D$, and this is the group of $k$-points of the algebraic $k$-group $\mathbf{GL}_D$, which is a $k$-form of $\mathbf{GL}_2$. Similarly, the set of quaternions $q \in D$ with unit norm forms a multiplicative group which we denote by $\mathrm{SL}_D$, and this is the group of $k$-points of the algebraic $k$-group $\mathbf{SL}_D$, which is a $k$-form of $\mathbf{SL}_2$. If $D \cong M_2(k)$, then we have $\mathrm{GL}_D = \mathbf{GL}_D(k)\cong \mathbf{GL}_2(k)$ and $\mathrm{SL}_D = \mathbf{SL}_D(k) \cong \mathbf{SL}_2(k)$.

Now, let $k$ be an algebraic number field, $\OOO$ its ring of integers and $D$ a quaternion algebra over $k$. An order in $D$ is an $\OOO$-submodule $O\subset D$ such that $O$ is a subring containing $1$ which generates $D$ over $k$. The subring $M_2(\OOO)$ is always an order in $M_2(k)$. If $a, b \in \OOO$, the quaternions with coefficients in $\mathcal{O}$ always form an order in $D(a, b)$. Given an order $O \subset D$, its group of units is the group $\mathrm{SL}_D(O)= \mathrm{SL}_D \cap O$. The groups of units of any two orders in a quaternion algebra $D$ are commensurable.

\subsubsection{Arithmetic hyperbolic lattices from quaternion algebras}\label{sec:unitary-lattices}

Let $D$ be a quaternion algebra over the field $k$ with the ring of integers $\OOO$, and let
$$
F(x,y) = \sum^m_{i,j = 1} x^*_i\, a_{ij}\, y_j,\; (a_{ij}\in D, a_{ij}= -a^*_{ji})
$$
be a non-degenerate skew-Hermitian form on the right $D$-module $D^m$. Let $\mathrm{U}_F(D)$ denote the group of automorphisms of $D^m$ preserving the form $F$. This is the group $\mathbf{U}_F(k)$ of $k$-points of the algebraic group $\mathbf{U}_F$. If $k$ is an algebraic number field and $O$ is some order in $D$, let $\mathrm{U}_F(O)$ denote the subgroup of $\mathrm{U}_F(D)$ preserving the $\OOO$-lattice $O^m$.

If $D \cong M_2(k)$, then $\mathbf{U}_F$ is $k$-isomorphic to the orthogonal group $\mathbf{O}_f$ of a $k$-defined form $f$ of rank $2m$. Indeed, $D^m$ is a $4m$-dimensional vector space over $k$. Assuming that 
$D = D(1, 1)$, we set
$$
D^m_{\pm} = \{x\in D^m \mid x\, \I = \pm x\}.
$$
Then $D^m = D^m_+ \oplus D^m_-$ and $D^m_- = D^m_+\, \J$, so that $\dim D^m_+ = \dim D^m_- = 2m$.
One can notice that, if $x, y \in D_+^m$, then 
$$
F(x,y) = f(x,y)\, (\I-1)\, \J,
$$
where $f$ is a symmetric non-degenerate $k$-bilinear form on $D^m_+$. Indeed, for all $x, y \in D^m_+$ we have that 
$$
0 = F(x,y(\I-1)) = F(x,y) (\I-1), \text{ and } 0 = F(x(\I-1),y) = (-\I-1)F(x,y).
$$ 
By a straightforward computation, the only quaternions $q \in D$ that satisfy~$q(\I-1)=(-\I-1)q=\!0$ are precisely those of the form $\lambda(\I-1)\J$ for some scalar $\lambda \in k$ (the latter depending on $x, y \in D^m_+$). The fact that the form $f$ is symmetric follows from $F$ being skew-Hermitian.

Notice that for every $A \in \mathrm{U}_F(D)$ we have that $A(D^m_+)=D^m_+$. Indeed, for every $x \in D^m_+$ we have $(Ax)(\I-1)=A(x(\I-1))=A(0)=0$. Moreover $A:D^m_+\rightarrow D^m_+$ belongs to $\mathbf{O}_f(k)$. Indeed for every $x,y \in D^m_+$ we have
$$f(x,y)(\I-1)\J=F(x,y)=F(Ax,Ay)=f(Ax,Ay)(\I-1)\J$$ and since both $f(x,y)$ and $f(Ax,Ay)$ are elements of $k$ we have that they coincide.

By the discussion above we see that the map $A \mapsto A|_{D^m_+}$ defines an isomorphism $$\Phi \colon \mathbf{U}_F(k)=\mathrm{U}_F(D) \to \mathbf{O}_f(k),$$ where the form $f$ is determined by $F$ up to a scalar. The surjectivity can be easily established by noticing that every $B\in \mathbf{O}_f(k)$ extends uniquely to an element $A \in \mathrm{U}_F(D)$ by setting $A(x+y\J)=Bx+(By)\J$ for every $x,y \in D^m_+$. For $k = \R$, let us define the signature of $F$ to be the signature of $f$. 

Now let $k \subset \R$ be a totally real number field and $D$ be a quaternion algebra such that $D^\sigma \otimes \R \cong M_2(\R)$ for all embeddings $\sigma \colon k \hookrightarrow \R$. A skew-hermitian $D$-form $F$ is \textit{admissible} if, regarded as a form on $(D\otimes\R)^m$, it has signature $(2m-1, 1)$ while $F^\sigma$ has signature $(2m, 0)$ for all non-identity embeddings $\sigma \colon k \hookrightarrow \R$. Given an admissible form $F$ we have that $$\mathrm{U}_F (D \otimes \R)=\mathbf{U}_F({\mathbb{R}}) \cong \mathbf{O}_f({\R}) \cong \OO_{2m-1,1},$$ and $\mathrm{U}_F(D)$ is naturally identified with a subgroup of $\mathrm{O}_{2m-1,1}$.

Thus, for every order $O$ in $D$, we have that the projection $\mathrm{PU}_F(O)$ of $\mathrm{U}_F(O)<\mathbf{O}_f({\R})$ to $\mathbf{PO}_f({\R}) \cong \mathrm{PO}_{2m-1,1}$  is an arithmetic lattice in $\Isom(\HH^{2m-1})$. 

When $D \cong M_2(k)$, we have that $\mathrm{U}_F(D)$ is $k$-isomorphic to $\mathrm{O}_f(k)$, with $f$ being an admissible $k$-defined form. Thus we do not obtain anything new: in this case all arithmetic lattices in $\mathbf{U}_F(k)$ yield type-I lattices.

If $D$ is a division algebra we obtain a new class of hyperbolic arithmetic lattices.

\begin{definition}\label{def:type-2-lattice}
Let $F$ be an admissible skew-Hermitian form defined over a division quaternion algebra $D$ over $k$, and $O$ be an order in $D$.
All subgroups $\Gamma < \mathrm{PO}_{2m-1, 1}$ which are commensurable in the wide sense with a group of the form $\mathrm{PU}_F(O)$ as above are called \textit{arithmetic lattices of type~II}. A hyperbolic orbifold $M=\mathbb{H}^{2m-1}/\Gamma$ is a {\em type-II orbifold} if $\Gamma$ is an arithmetic lattice of type~II.

In the same setting, if $\Gamma<\mathrm{PO}_{2m-1,1}$ is a lattice commensurable in the wide sense with a subgroup of $\mathbf{PU}_F(k)$, we say that $\Gamma$ (and the orbifold $M=\mathbb{H}^{n}/\Gamma$) are \textit{quasi-arithmetic of type~II}.
\end{definition}

\begin{remark}
Given a totally real field $k$ and a division algebra $D$ over $k$ satisfying the hypotheses above, two admissible skew-Hermitian forms $F_1$ and $F_2$ on $D^m$ define the same commensurability class of hyperbolic lattices if and only if $F_1$ is equivalent (up to change of basis in $D^m$) to $\lambda \cdot F_2$ for some $\lambda \in k^\times$. 
\end{remark}

\begin{remark}
If $k = \Q$ and $F$ is isotropic over $D$ then the orbifold $\HH^{2m-1}/\Gamma$ is non-compact, otherwise $\Gamma$ is a uniform lattice, as follows from Godement's criterion \cite{God62}. 
\end{remark}
Also, in the non-uniform case the quaternion algebra $D$ is necessarily isomorphic to $M_2(\Q)$, so that $\Gamma$ is a type-I lattice. More generally, we prove the following.
\begin{prop}\label{prop:non-isotropic-over-division-algebra}
If $k\subset \mathbb{R}$ is a real algebraic number field, $D$ is a quaternion division algebra over $k$ which splits over $\R$ and $F$ is a skew-Hermitian form on the right $D$-module $D^m$ of signature $(2m-1,1)$, then $F$ cannot represent $0$ over $D$.
\end{prop}
\begin{proof}
Indeed, suppose that there exists a non-zero $v \in D^m$ such that $F(v, v)=0$. Let us consider the real quaternion algebra $D\otimes \R \cong M_2(\mathbb{R})$. Then $v = x + y\J$ for some $x, y \in (D\otimes \R)^m_+$, and 
\begin{equation}\label{eq:skew-hermitian-decomposition}
0 = F(v, v) = f(x, x)(\I-1)\J + f(x, y)(\I-1) + f(x, y)(\I+1) + f(y, y)(\I+1)\J.
\end{equation}

Note that $\{(\I-1)\J,\, \I-1,\, \I+1,\, (\I+1)\J\}$ is an $\R$-basis of $D\otimes \R$, and thus $f(x, x) = f(x, y) = f(y, y) = 0$. Since $f$ is a form of signature $(2m-1, 1)$, we have that $y = \mu x$ for some $\mu \in \R$, so that $v = x(1 + \mu \J)$.

If $\mu = \pm 1$, then we have $(1 + \mu \J)(1 - \mu \J) = 0$, and thus $v(1-\mu \J) = 0$, so that the coordinates of $v$ understood as elements of $(D\otimes \R)^m_+$ are zero divisors in $D\otimes \R$. Suppose that $q \in D$ is a non-zero coordinate of $v$. If $D$ were a division algebra, then $q$ would be invertible in $D$ and therefore also in $D\otimes \R$. This is a contradiction, since $D\otimes \R$ is a non-trivial associative algebra, and no element can be invertible and a zero divisor at the same time. If $\mu \neq \pm 1$, then $1 + \mu \J$ is invertible, since $N(1 + \mu \J)\neq 0$. Then $v(1 + \mu \J)^{-1}(\I-1)=x(\I-1)=0$, and again all the coordinates of $v$ are zero divisors. Thus, $D$ cannot be a division algebra.
\end{proof}

\begin{remark}
Let us note that in the classification of algebraic semi-simple groups by Tits \cite{Tits66}, type-II lattices as defined above belong to the indices ${^1}D^{(2)}_{m, 0}$ or ${^2}D^{(2)}_{m, 0}$, as also follows from Proposition~\ref{prop:non-isotropic-over-division-algebra} (see also \cite[Proposition 13.6]{Emery}). We would also like to mention the fact that the discussion in \cite[p.~220]{Vin-Geometry-II} apparently does not take into account type-II arithmetic lattices in dimension $3$ (cf. Section~\ref{sec:3-dim-arithmetic}).
\end{remark}

\subsubsection{Constructing $k$-involutions}\label{sec:k-involutions-type2}
Below we show how to construct all the $k$-involutions in any type II group of the form $\mathbf{PU}_F (k)$ where, as in the previous section, $D$ is a quaternion division algebra over $k$ and $F$ is an admissible skew-Hermitian form on $D^m$ of signature $(2m-1, 1)$.

Let $D_1$ be a proper right submodule of $D^m$. Since $D_1$ is a module over a division algebra, it has to be free, so that $D_1\cong D^l$ for some $l < m$. We can restrict the form $F$ to a  skew-Hermitian form $G$ on $D_1$. Let us suppose that $G$ is non-degenerate with signature $(2l-1, 1)$. Now, let us consider the orthogonal complement $$D_1^{\perp}=\{x \in D^m \,|\,F(x, y) = 0\, \text{ for all } y \in D_1\}.$$ 
Notice that since $F$ is skew-Hermitian, the derived orthogonality relation is symmetric.

We observe that $D^m = D_1 \oplus D_1^{\perp}$. Indeed, $D_1\cap D_1^{\perp}$ is trivial by Proposition~\ref{prop:non-isotropic-over-division-algebra}.

In order to prove that $D_1 + D_1^{\perp}=D^m$, we proceed as follows: first, fix a basis $(x_1, \dots, x_l)$ of $D_1$ and complete it to a basis $(x_1, \dots, x_m)$ of $D^m$. Next, apply the Gram--Schmidt orthogonalisation process by setting
\begin{small}
\begin{eqnarray*}
\mathrm{proj}_y(x)=yF(y, y)^{-1}F(y, x),& y_1=x_1, & y_i= x_i - \sum_{j=1}^{i-1}\mathrm{proj}_{y_j}(x_i),\; i=2,\dots,m.
\end{eqnarray*}    
\end{small}
Note that $\mathrm{proj}_y(x)$ is well-defined whenever $y\in D^m$ is non-zero. Indeed $F(y, y)$ is non-zero by Proposition \ref{prop:non-isotropic-over-division-algebra} and it is invertible since $D$ is a division algebra. Also, a straightforward check shows that $F(y, x - \mathrm{proj}_y(x)) = 0$ for any $x \in D^m$.

We thus obtain an orthogonal basis $(y_1, \dots, y_m)$ for $D^m$ such that the first $k$ vectors are a basis for $D_1$ and the remaining belong to $D_1^{\perp}$. This is clearly enough to conclude that $D_1 + D_1^{\perp}=D^m$, and moreover we see that $D_1^{\perp}\cong D^{m-l}$ and that the restriction of $F$ to $D_1^{\perp}$ has signature $(2(m-l), 0)$.

Finally, we notice that if $N$ is a linear transformation of $D^m$ that acts as the identity on $D_1$ and as multiplication by $-1$ on $D_1^{\perp}$, then $N$ is an involution in $\mathrm{U}_F(D)=\mathbf{U}_F (k)$. Moreover, the fixed point set for the action of $N$ on $\mathbb{H}^{2m-1}$ is precisely a totally geodesic subspace of dimension $2l-1$. 

It is not difficult to see that all the $k$-involutions of $\mathrm{U}_F(D)$ arise through the construction above. Indeed, let us suppose that $N \in \mathrm{U}_F(D)$ has order~$2$. Since $D$ is a division algebra, the eigenvalues of $N$ as an element of $\mathrm{GL}(D^m)$ can only be equal to $+1$ or $-1$. The corresponding right $D$-modules $D_1 = \{ x\in D^m \,|\, Nx = x \}$ and $D_{-1} = \{ x \in D^m \,|\, Nx = -x \}$ are orthogonal with respect to the form $F$ (i.e.\ $D_{-1} = D_1^{\perp}$). The involution $N$ corresponds to an element of $\mathrm{Isom}(\mathbb{H}^n)$ if and only if the restriction $F_{\mid {D_1}}$ of the form $F$ to $D_1$ has signature $(2l-1, 1)$ for $l = \mathrm{dim}\, D_1$. 

Finally, suppose that $i=[N]\in \mathrm{PU}_F(D)$, with $N$ as above, and that $\Gamma$ is a type-II arithmetic lattice in $\mathrm{PU}_F(D\otimes \mathbb{R})=\mathbf{PU}_F (\mathbb{R})$. By the discussion above, it follows that the fixed-point set of $i$ projects to a type-II totally geodesic subspace in $M=\mathbb{H}^n/\Gamma$ with associated admissible skew-hermitian $k$-defined form $G=F_{|D_1}$.

We still need some work to classify the involutions in the group $\mathbf{PU}_F (k)$ of $k$-points of the algebraic $k$-group $\mathbf{PU}_F$. The issue is the same as in the case of type-I lattices acting on $\mathbb{H}^n$ with $n$ odd (see Section~\ref{sec:k-involutions-type1}), namely that the group $\mathbf{U}_F$ is not adjoint.
Fortunately, the argument from \cite[Lemma 4.3]{ERT} with small modifications applies to this case as well.

Notice that $D\otimes \mathbb{C} \cong M_2(\mathbb{C})$ so that the elements of $\mathrm{U}(F, D \otimes \mathbb{C})=\mathbf{U}_F$ can be represented by elements of $\mathbf{GL}_{2m}(\mathbb{C})$.
We define the \textit{general unitary group} of the form $F$:
\begin{equation*}
    \mathbf{GU}_F=\{B \in \mathbf{GL}_{2m}(\mathbb{C}) \,|\, B^* F B = \mu F\; \mathrm{for\; some}\; \mu \in \mathbb{C}\},
\end{equation*} 
where $F$ is now understood as a $2m \times 2m$ matrix with complex coefficients. The matrix $B^*$ is obtained from $B$ by transposing the ($D\otimes \mathbb{C}$)-coefficients (each being represented by a $2 \times 2$ block) and then applying to each block the standard involution \eqref{eq:standard-involution-split}. The corresponding group of $k$-points is $$\mathbf{GU}_F (k) = \mathbf{GU}_F \cap \mathbf{GL}_{2m}(k).$$

Recall that by the argument of Section \ref{sec:unitary-lattices} there exists an isomorphism $$\Phi:\mathbf{U}_F (\mathbb{R}) = \mathrm{U}_F(D\otimes \R) \rightarrow \mathbf{O}_f(\mathbb{R}),$$ where $f$ is an admissible $k$-form of signature $(2m-1, 1)$. This isomorphism extends to an isomorphism $\Phi: \mathbf{GU}_F \rightarrow \mathbf{GO}_f$ of real algebraic groups, where $\mathbf{GO}_f$ is the general orthogonal group of $f$ \cite[p. 7]{ERT}. 

Notice that the element $\mu \in \mathbb{C}$ is uniquely determined by $B \in \mathbf{GU}_F$ so that one can unambiguously write $\mu=\mu(B)$. Moreover $B$ represents an equivalence between the forms $f$ and $\mu f$, which both have signature $(2m-1, 1)$. By an exactly the same argument as in the type-I case (see \cite[p. 7]{ERT}) we have that $\mu \in \mathbb{R}$ and $\mu > 0$. Moreover, if $m \geq 2$ then $\mu$ is \textit{totally positive} and thus $K = k(\sqrt{\mu})$ is totally real.

Let us define the group of ``scalar'' $2m \times 2m$ matrices $\mathbf{C} = \{c \cdot \mathrm{id} \,|\, c\in \mathbb{C}^{\times} \}$.
Then the \textit{projective general unitary group} of $F$ is
$\mathbf{PGU}_F=\mathbf{GU}_F/\mathbf{C}$. The isomorphism $\Phi$ above
 naturally descends to an isomorphism $\mathbf{PGU}_F \to \mathbf{PGO}_f$, where $\mathbf{PGO}_f = \mathbf{GO}_f/\mathbf{C}$ is the projective general orthogonal group of $f$, as defined in \cite[p.~7]{ERT}. 
 
Suppose that $B \in \mathbf{GU}_F (k)$. We claim that $\mu(B) \in k_{>0}=k \cap \mathbb{R}_{>0}$. Notice that $B^* F B = \mu(B) \cdot F$, hence $\mu(B) \in D$. Moreover, $\mu(B) \in \mathbb{R}_{>0}$ and therefore belongs to the centre of $D\otimes \mathbb{R}$. These two facts together imply that $\mu(B)$ is in the centre of $D$, which is precisely $k$. 

The following result is analogous to \cite[Lemma 4.3]{ERT} in the setting of type-II lattices.

\begin{lemma}\label{lemma:ERT-type-II}
Let $F$ be a skew-Hermitian form on the right $D$-module $D^m$, where $D$ is a quaternion algebra over a totally real number field $k$ satisfying the hypotheses of Definition \ref{def:type-2-lattice}. Let $\mathbf{PU}_F (k)$ be the group of $k$-points of the adjoint group $\mathbf{PU}_F$. Then

\begin{equation}\label{eq:k-points-type-2}
\mathbf{PU}_F (k) = \left\{ \pm \frac{1}{\sqrt{\mu(B)}}\, B\;|\; B \in \mathbf{GU}_F (k) \right\}.   
\end{equation}
\end{lemma}

\begin{proof}
The proof follows closely the argument of \cite[Lemma 4.3]{ERT}, with the main difference being that the groups $\mathbf{O}_f$ and $\mathbf{GO}_f$ are replaced by their type II counterparts $\mathbf{U}_F$ and $\mathbf{GU}_F$, respectively.

Let $\pi:\mathbf{U}_F \rightarrow \mathbf{PU}_F$ and $\eta: \mathbf{GU}_F \rightarrow \mathbf{PGU}_F$ be the quotient maps. Both $\pi$ and $\eta$ are $k$-homomorphisms of $k$-algebraic groups by \cite[Theorem 6.8]{Borel}. The inclusion map $$\nu: \mathbf{U}_F \rightarrow \mathbf{GU}_F$$ is a $k$-homomorphism as well. By the universal mapping property \cite[p.~94]{Borel} the inclusion $\nu$ of $\mathbf{U}_F$ into $\mathbf{GU}_F$  induces a $k$-homomorphism $\overline{\nu}: \mathbf{PU}_F\rightarrow \mathbf{PGU}_F$ such that $\overline{\nu}\pi=\eta \nu$.

If $B \in \mathbf{U}_F$, then $\overline{\nu}(\pm B)= \mathbf{C}B$ and so $\overline{\nu}$ is a monomorphism. Now, assume that $B \in \mathbf{GU}_F$. Then $B^*FB = \mu F$, and thus 
$$
(1/\sqrt{\mu}\cdot B)^*\, F\, (1/\sqrt{\mu}\cdot B) = F
$$ which implies $\pm 1/\sqrt{\mu}\cdot B \in \mathbf{U}_F$. We have that $\overline{\nu}(\pm \frac{1}{\sqrt{\mu}}\cdot B)=\mathbf{C}B$ and thus $\overline{\nu}$ is an isomorphism.

There is a short exact sequence of algebraic $K$-groups
$$1 \rightarrow \mathbf{C} \rightarrow \mathbf{GU}_F\xrightarrow{\eta} \mathbf{PGU}_F \rightarrow 1$$ which determines an exact sequence of Galois cohomology groups
$$1 \rightarrow \mathbf{C}_k \rightarrow \mathbf{GU}_F (k) \xrightarrow{\eta} \mathbf{PGU}_F (k) \rightarrow H^1(k,\mathbf{C})$$ by \cite[Prop. 1.17 and Corollary 1.23]{Borel-Serre}. We have that $H^1(k,\mathbf{C})=0$ by \cite[p. 72, Prop. 1]{Serre-Gal-cohomology}, therefore $\eta(\mathbf{GU}_F (k))=\mathbf{PGU}_F (k)$, and
$\mathbf{PGU}_F (k)=\{\mathbf{C}B\;|\; B \in \mathbf{GU}_F (k)\}.$
Therefore
$$\mathbf{PU}_F (k)= \overline{\nu}^{-1}(\mathbf{PGU}_F (k))= \left\{\pm\frac{1}{\sqrt{\mu(B)}} B \;|\; B \in \mathbf{GU}_F (k) \right\}.$$
\end{proof}

Now that we have a complete description of $k$-involutions in the adjoint $k$-group $\mathbf{PU}_F$, we can discuss the structure of the corresponding fc-subspaces. The case where $N$ belongs to the image of $\mathrm{U}_F(D)$ in $\mathbf{PU}_F$ has already been treated, and we get that the corresponding fc-subspaces are arithmetic of type II, with the same adjoint trace field $k$.

We now suppose that $\pm N$ is an involution in $\mathbf{PU}_F (k)$ of the form $\pm N=\pm\frac{1}{\sqrt{\mu(B)}} B$, with $B \in \mathbf{GU}_F (k)$ and $\mu = \mu(B)$ is totally positive and not a square in $k$. Let $K=k(\sqrt{\mu})$. The discussion proceeds in very much the same way as for odd-dimensional type-I lattices (see Section \ref{sec:k-involutions-type1}). Notice that both $N$ and $B$ are defined up to multiplication by $-\mathrm{id}$. Without loss of generality, we may assume that $N=(1/\sqrt{\mu})B$. We define the \textit{positive eigenspace} $D_1$ of $N$ and similarly the \textit{negative eigenspace} $D_{-1}$ of $N$. Notice that now $D_1$ and $D_{-1}$ are right $(D \otimes K)$-modules, and that the quaternion algebra $D \otimes K$ may now split.

Up to multipliciation by $-\mathrm{id}$, we may assume that the restriction $G$ of the form $F$ to $D_1$ has signature $(n-1,1)$, while the restriction $H$ of $F$ to $D_{-1}$ has signature $(l,0)$, with $n= \mathrm{dim}(D_1)$, $l=\mathrm{dim}(D_{-1})$ and $n+l=m$.

We claim that the form $G$ is admissible. Notice that the nontrivial automorphism $\sigma \in \mathrm{Gal}(K/k)$ acts naturally on the elements of the quaternion algebra $D \otimes K$ by sending $q \otimes \lambda$ to $q \otimes \lambda^{\sigma}$. By computing the restriction of scalars of $\mathbf{U}_G$ we obtain the $K$-isomorphism:
$$\mathrm{Res}_{K/k}\,\mathbf{U}_G\cong \mathbf{U}_G \times \mathbf{U}_{G^{\sigma}},$$
where $G^{\sigma}$ is obtained by applying $\sigma$ to the coefficients of $G$. 

The action of $\sigma$ sends $N=(1/\sqrt{\mu})B$ to $-N=-(1/\sqrt{\mu})B$, thus exchanging the positive and negative eigenspaces $D_1$ and $D_{-1}$. It follows that $n=\mathrm{dim}(D_1)=\mathrm{dim}(D_{-1})=l$ and that $m=n+l$ is even. The form $F$ is defined over $k$, and because of this we have that $F^{\sigma}=F$. We may therefore identify the form $G^{\sigma}$ with the positive definite form $H=F_{|D_{-1}}$ so that $\mathbf{U}_{G^{\sigma}}$ is isomorphic to group $\mathbf{U}_H$.
We obtain $K$-defined isomorphisms
\begin{equation}\label{eq:res-of-scalars-type-II} \mathrm{Res}_{K/k}\,\mathbf{U}_G \cong \mathbf{U}_G \times \mathbf{U}_H \cong \mathbf{H}, \end{equation} where the $k$-group $\mathbf{H}$ is the centraliser of $N$ in $\mathbf{U}_F$. Moreover the isomorphism in \eqref{eq:res-of-scalars-type-II} induces a bijection between the $k$-points of $\mathrm{Res}_{K/k}\,\mathbf{U}_G$ and those of $\mathbf{H}$, implying that these two groups are actually $k$-isomorphic.
As in the case of type-I lattices, the admissibility of $G$ now follows easily from the admissibility of $F$ over $k$.

These facts together imply that if $\Gamma<\mathbf{U}_F (k)$ is a type-II arithmetic lattice, then $\mathrm{Stab}_{\Gamma}(V_1)=\Gamma \cap \mathbf{H}(k)$ is an arithmetic lattice in the group $\mathbf{U}_G(\mathbb{R}) \times \mathbf{U}_H(\mathbb{R})$. By modding out the the compact factor $\mathbf{U}_H(\mathbb{R})$ we obtain that the projection $\Lambda$ of $\mathrm{Stab}_{\Gamma}(V_1)$ to $\mathbf{U}_G(\mathbb{R}) \cong \mathrm{O}_{m-1,1}$ is an arithmetic hyperbolic lattice with adjoint trace field $K$. 

Finally, we want to understand whether this totally geodesic sublattice $\Lambda<\mathbf{U}_F (K)$ is of type I or type II. By the discussion in Section \ref{sec:unitary-lattices}, if $D \otimes K \cong M_2(K)$ then $\Lambda$ is a type-I lattice, while if $D \otimes K$ is a division algebra, $\Lambda$ is a type-II lattice.

\begin{remark}
It is a well-known fact that all hyperbolic lattices acting on $\mathbb{H}^n$, with $n$ even, are type-I lattices. To construct {\itshape all} arithmetic lattices acting on $\mathbb{H}^n$, where $n\neq 3,7$ is an odd number, one has to consider both 
type-I and type-II lattices \cite[p.~222]{Vin-Geometry-II}. In the remaining cases $n=3,7$, one has to add two exceptional families of arithmetic lattices which we now introduce.
\end{remark}

\newpage
\subsection{(Quasi-)arithmetic hyperbolic lattices in dimension $3$}\label{sec:3-dim-arithmetic}

\subsubsection{Exceptional isomorphism between $\mathbf{PGL}_2(\mathbb{C})$ and $\mathbf{PSO}_{3,1}(\R)$}

It is well-known that $\Isom^+(\HH^3)$ can be viewed both as $\mathbf{PGL}_2(\CC) = \mathrm{PSL}_2(\CC)$ and $\mathbf{PSO}_{3,1}(\R)$ due to the exceptional isomorphism between $\mathbf{PGL}_2(\mathbb{C})$ and $\mathbf{PSO}_{3,1}(\R)$. This gives a way to construct all arithmetic lattices acting on $\mathbb{H}^3$ through arithmetic subgroups of $\mathrm{PSL}_2(\mathbb{C})$, as follows:
\begin{enumerate}
    \item Fix a complex number field $L$ with one complex place (up to complex conjugation) and a finite (possibly empty) set of real places.
    \item Choose a quaternion algebra $A$ over $L$ such that $A$ is ramified at all real places of $L$.
    \end{enumerate}
    
If $O$ is an order in $A$, then the central quotient of the groups of units of $O$, denoted by $\mathrm{PSL}_A(O)$, is an arithmetic lattice acting on $\mathbb{H}^3$. Moreover, all arithmetic lattices acting on $\mathbb{H}^3$ are commensurable in the wide sense with a group of this form. 
In this setting both the field $L$ and the quaternion algebra $A$ are invariants of the commensurability class of the lattice $\Gamma$ and correspond to the {\itshape invariant trace field} and the {\itshape invariant quaternion algebra} of $\Gamma$ as defined in \cite[Section 3.3]{Mac-Reid}. As it is explained in \cite{Vin95}, the field $L$ is simply the invariant trace field of $\Gamma$ regarded as a lattice in $\mathbf{PGL}_2(\mathbb{C})$, i.e. $L=\mathbb{Q}(\{\mathrm{tr}(\mathrm{Ad}(\gamma))\,|\, \gamma \in \Gamma\})$ where $$\mathrm{tr}(\mathrm{Ad}(\gamma)) = \mathrm{tr}(\gamma)^2-1= \mathrm{tr}(\gamma^2)+1 $$ is the trace of the adjoint representation of $\gamma \in \mathbf{PGL}_2(\mathbb{C})$. The complex ambient group of $\Gamma<\mathbf{PGL}_2(\mathbb{C})$ is $\mathbf{PGL}_A$: the central quotient of the group of invertible elements of the invariant quaternion algebra $A$. Notice that the group of $L$-points $\mathbf{PGL}_A(L)$ is isomorphic to the group $A^*/L^*$ defined as the quotient of the group $\mathrm{GL}_A=A^*$ by the multiplicative action of $L^*$ \cite[Theorem 8.4.4]{Mac-Reid}.

If the lattice is not uniform, then necessarily $L=\mathbb{Q}(\sqrt{-d})$, where $d > 0$ is square-free, $A=M_2(L)$, and the resulting lattice is commensurable with $\mathrm{PSL}_2(\mathcal{O}_d)$, where $\mathcal{O}_d$ is the ring of integers of $L$. These are the so-called Bianchi groups. In all other cases $A$ is a division algebra \cite[Section 8.2]{Mac-Reid}.

We wish to understand how the invariant trace field of a lattice $\Gamma<\mathbf{PGL}_2(\mathbb{C})$ relates to the adjoint trace field of the image of $\Gamma$ under the exceptional isomorphism $i:\mathbf{PGL}_2(\mathbb{C}) \rightarrow \mathbf{PSO}_{3,1}(\mathbb{R})$.  We therefore turn our attention to the exceptional isomorphism $i:\mathbf{PGL}_2(\mathbb{C}) \rightarrow \mathbf{PSO}_{3,1}(\mathbb{R})$, which is explicitly described in \cite[Section 10.2]{Mac-Reid}. For a complex number $z=x+iy$, $x,y \in \mathbb{R}$, denote by $z^*$ its conjugate $z^*=x-iy$. We prove the following:

\begin{prop}\label{prop:adjoint-rep-trace}
Let $i$ denote the exceptional isomorphism from $\mathbf{PGL}_2(\mathbb{C})$ to $\mathbf{PSO}_{3,1}(\mathbb{R})$. For any element $g \in \mathbf{PGL}_2(\mathbb{C})$, $\mathrm{tr}(\mathrm{Ad}(i(g)))= \mathrm{tr}(\mathrm{Ad}(g))+\mathrm{tr}(\mathrm{Ad}(g))^*$.
\end{prop}

\begin{proof}
The isomorphism $i$ can be conveniently interpreted as an isomorphism of {\itshape real} Lie groups between the $\mathbb{R}$-points of $\mathrm{Res}_{\mathbb{C}/\mathbb{R}}\,\mathbf{PGL}_2$ and the group $\mathbf{PSO}_{3,1}(\mathbb{R})$. Being an isomorphism of real Lie groups, it induces an isomorphism of the real Lie algebras and these induce an isomorphism of the corresponding adjoint representations. Thus, we opt to work directly in the group $\mathrm{Res}_{\mathbb{C}/\mathbb{R}}\,\mathbf{PGL}_2(\mathbb{R})$ and compute the trace of the adjoint representation of its elements. In order to describe this group, we use the regular representation of complex number as $2 \times 2$ matrices with real coefficients of the form 
\begin{equation}\label{eq:complexrep}\mathbb{C} \ni z = x + iy \mapsto \begin{pmatrix} x &-y \\y &x \end{pmatrix} \in M_2(\mathbb{R}),\end{equation} where $x$ and $y$ correspond respectively to the real and imaginary parts of $z \in \mathbb{C}$. By doing so we can describe the Weil restriction of $\mathbf{SL_2(\mathbb{C})}$ as:
\begin{small}
$$ 
\mathrm{Res}_{\mathbb{C}/\mathbb{R}}\,\mathbf{SL}_2(\mathbb{R})=\left\{\begin{pmatrix}A & B \\ C & D \end{pmatrix}\Big|\; AD-BC=\begin{pmatrix}1&0\\0& 1 \end{pmatrix},\; A,B,C,D\; \textrm{of the form}\, \eqref{eq:complexrep}\right\},
$$    
\end{small}

\noindent which is a $6$-dimensional real Lie group. The group $\mathrm{Res}_{\mathbb{C}/\mathbb{R}}\,\mathbf{PGL}_2$ is the quotient of $\mathrm{Res}_{\mathbb{C}/\mathbb{R}}\,\mathbf{SL}_2$ by $\{\pm \mathrm{id}\}$ and its Lie algebra is the Weil restriction from $\mathbb{C}$ to $\mathbb{R}$ of the Lie algebra of $\mathbf{SL}_2(\mathbb{C})$, i.e. is obtained by regarding the complex $3$-dimensional Lie algebra $\mathfrak{sl}_2$ as a $6$-dimensional real Lie algebra $\mathrm{Res}_{\mathbb{C}/\mathbb{R}}\,\mathfrak{sl}_2$. 

As such, if $B=(v_1,v_2,v_3)$ is a basis of $\mathfrak{sl}_2$, a basis of $\mathrm{Res}_{\mathbb{C}/\mathbb{R}}\,\mathfrak{sl}_2$ is given by $B'=(v_1,i\cdot v_1, v_2,i\cdot v_2, v_3,i\cdot v_3)$. Suppose that $g \in \mathbf{SL}_2(\mathbb{C})$ and $\mathrm{Ad}(g)\in \mathbf{GL}(\mathfrak{sl}_2)$ is represented with respect to the basis $B$ by a matrix 
\begin{equation}\label{eq:complex-adjoint-matrix}M=\begin{pmatrix}a_{1,1}&a_{1,2}&a_{1,3} \\ a_{2,1}&a_{2,2}&a_{2,3} \\a_{3,1} &a_{3,2} &a_{3,3}  \end{pmatrix} \in \mathbf{GL}_3(\mathbb{C})\end{equation} 
with trace $\mathrm{tr}(\mathrm{Ad}(g))=a_{1,1}+a_{2,2}+a_{3,3}$. We denote by $\mathrm{Res}_{\mathbb{C}/\mathbb{R}}(g)$ the element which correponds to $g$ in the group $\mathrm{Res}_{\mathbb{C}/\mathbb{R}}\,\mathbf{SL}_2$. Its adjoint action $\mathrm{Ad}(\mathrm{Res}_{\mathbb{C}/\mathbb{R}}(g)) \in \mathbf{GL}(\mathrm{Res}_{\mathbb{C}/\mathbb{R}}\,\mathfrak{sl}_2)$ is represented with respect to the basis $B'$ by the matrix 

\begin{equation}\label{eq:real-adjoint-matrix}\mathrm{Res}_{\mathbb{C}/\mathbb{R}}(M)=\begin{pmatrix}A_{1,1}&A_{1,2}&A_{1,3} \\ A_{2,1}&A_{2,2}&A_{2,3} \\A_{3,1} &A_{3,2} &A_{3,3}  \end{pmatrix} \in \mathbf{GL}_6(\mathbb{R}),\end{equation} where each $2 \times 2$ real submatrix $A_{i,j}$ is obtained from the corresponding coefficient of $M$ as in $\eqref{eq:complexrep}$. We thus see that 
\begin{multline}
\label{eq:trace-restriction}
\mathrm{tr}(\mathrm{Ad}(i(g)))=\mathrm{tr}(\mathrm{Ad}(\mathrm{Res}_{\mathbb{C}/\mathbb{R}}(g)))= \mathrm{tr}(A_{1,1})+ \mathrm{tr}(A_{2,2})+\mathrm{tr}(A_{3,3})=\\
= \mathrm{tr}(\mathrm{Ad}(g))+ \mathrm{tr}(\mathrm{Ad}(g))^*.     
\end{multline}
\end{proof}

We note the following important corollary of Proposition \ref{prop:adjoint-rep-trace}:
\begin{corollary}\label{cor:subfield}
If $\Gamma<\mathbf{PGL}_2(\mathbb{C})$ is an arithmetic lattice with invariant trace field $L$ that contains a subfield $k$ such that $[L:k]=2$, then the adjoint trace field of  $i(\Gamma)<\mathbf{PSO}_{3,1}$ is precisely $k$, and thus is totally real. Moreover, the identity component $\mathbf{G}^{\circ}$ of the real ambient group $\mathbf{G}$ of $i(\Gamma)$ is $k$-isomorphic to the group $\mathrm{Res}_{L/k}\,\mathbf{PGL}_A$, where $A$ is the invariant quaternion algebra of $\Gamma$.
\end{corollary}
\begin{proof}
Since the field $L$ has one complex place, all of its subfields are totally real. Therefore $L$ is an imaginary quadratic extension of the totally real field $k$ and for any $z \in L$, $z+z^* \in k$. By Proposition \ref{prop:adjoint-rep-trace}, the adjoint trace field of $i(\Gamma)$ is $$\mathbb{Q}(\{\mathrm{tr}(\mathrm{Ad}(i(\gamma))\,|\, \gamma \in \Gamma\})=\mathbb{Q}(\{\mathrm{tr}(\mathrm{Ad}(\gamma))+ \mathrm{tr}(\mathrm{Ad}(\gamma))^*\,|\, \gamma \in \Gamma\})\subset k.$$

By \cite[Remark 7 (b)]{Vin95}, there exists $\gamma \in \Gamma$ such that $\beta=\mathrm{tr}(\mathrm{Ad}(\gamma)) \in \mathbb{R}$ and $k=\mathbb{Q}(\beta)$. This implies that $\mathrm{tr}(\mathrm{Ad}(i(\gamma)))=2\beta$, and therefore $k$ is precisely the adjoint trace field of $i(\Gamma)$. The statement about the ambient groups follows easily: the exceptional isomorphism $i:\mathbf{PGL}_2(\mathbb{C})\rightarrow \mathbf{PSO}_{3,1}(\mathbb{R})$ maps the Zariski-dense subgroup $\Gamma<\mathbf{PGL}_A(L)$ to the Zariski-dense subgroup $i(\Gamma)<\mathbf{G}^{\circ}(k)$. Now, the group $\mathbf{PGL}_A(L)$ corresponds to the group of $k$-points of $\mathrm{Res}_{L/k}\,\mathbf{PGL}_A$. The desired $k$-isomorphism follows from the same argument as in the proof of Proposition \ref{prop:adj-trace-of-QA}.
\end{proof}

\medskip

\subsubsection{Classification of (quasi)-arithmetic lattices in $\mathrm{PSL}_2(\CC)$: types I-III}\label{sec:3-dim-types}
It is natural to ask how to characterize type-I and type-II lattices among the $3$-dimensional ones with the above description. As mentioned in \cite[p.~366]{Li-Millson}, type-I and type-II lattices correspond to the case where the field $L$ contains a totally real subfield $k$ such that the degree of the extension $[L:k]$ is $2$, i.e. type-I and type-II lattices are those to which Corollary \ref{cor:subfield} applies. In this case, the distinction between type I and type II can be recovered as follows \cite[pp.~199--200]{Schwerner}: if the norm form $N_{L/k}(A)$ splits over $k$ then we obtain type-I lattices, otherwise we get type-II lattices. In the case of type-I lattices, the corresponding admissible form $f$ defined over $k$ can be explicitly recovered from the data of the field $L$ and the quaternion algebra $D$ \cite[Section 10.2]{Mac-Reid}. 

The exceptional family of $3$-dimensional hyperbolic lattices arises when the field L does not contain a subfield $k$ such that $[L:k]=2$.
The corresponding arithmetic lattices are not of type I or II, thus we assign them to type III. 

\begin{definition}\label{def:3-dim-type-3}
    An arithmetic lattice $\Gamma<\mathrm{PSL}_2(\mathbb{C})$ with invariant trace field $L$ that does {\it not} contain a subfield $k$ such that $[L:k]=2$ is called a {\it $3$-dimensional type-III} arithmetic lattice.
\end{definition}
\noindent An explicit example of such a lattice is discussed, for instance, in \cite[p.~22]{LLR08}.

\medskip
We also record the following geometric characterization: type-I and type-II lattices are precisely the arithmetic lattices acting on $\mathbb{H}^3$ that contain a pure translation along a geodesic. Indeed, \cite[Remark 7(c)]{Vin95} implies that type-III lattices do not contain such translations, and the fact that type-I and type-II have such translations can be justified by the restriction of an admissible quadratic or skew-Hermitian form onto a subspace of signature $(1,1)$. All loxodromic elements of type-III lattices act as a composition of a translation and a rotation along the same axis with the rotation angle $\theta$ that is not a rational multiple of $2\pi$.

Summarizing, if $\Gamma<\mathrm{PSL}_2(\mathbb{C})$ is an arithmetic lattice with invariant trace field $L$, then $\Gamma$ is:
\begin{itemize}
    \item type I if $L$ contains a subfield $k$ such that $[L:k]=2$ and the norm form $N_{L/k}(A)$ splits over $k$;
    \item type II if $L$ contains a subfield $k$ such that $[L:k]=2$ and the norm form $N_{L/k}(A)$ does not split over $k$;
    \item type III if $L$ does not contain a subfield $k$ such that $[L:k]=2$.
\end{itemize}

We show how one may distinguish type I and II by examining the ramification set of the invariant quaternion algebra $A$.

\begin{prop}\label{prop:ramification-norm-form}
    For an arithmetic lattice $\Gamma<\mathbf{PGL}_2(\mathbb{C})$ with invariant trace field $L$ such that $[L:L\cap \mathbb{R}=k]=2$ and invariant quaternion algebra $A$, the following are equivalent:
    \begin{enumerate}
        \item $\Gamma$ is type I;
        \item $A=B\otimes_k L$ for some quaternion $k$-algebra $B$;
        \item The finite ramification set of $A$ consists of $s\geq 0$ pairs of prime ideals $P_i, P'_i \subset \mathcal{O}_L$, $i=1,\dots,s$, lying above a splitting prime ideal $Q_i\subset \mathcal{O}_k$ (i.e.\ $Q_i\cdot \mathcal{O}_L= P_i \cdot P'_i$ with $P_i \neq P'_i$);
        \item $A$ possesses an involution $\sigma$ of the second kind that leaves $k$ elementwise invariant;
        \item The norm form $N_{L/k}(A)$ splits.
    \end{enumerate}
\end{prop}

\begin{proof}
The equivalence of (1), (2) and (3) follows from \cite[Sections 9.5, 10.1, 10.2]{Mac-Reid}.
The equivalence of (4) and (5) is a result of Albert--Riehm--Scharlau (see \cite[Theorem 3.1]{INV}).
The implication (4)$\Rightarrow$(2) is a result of Albert \cite[Proposition 2.2]{INV}. Note that this implication is not valid for arbitrary central simple algebras, and holds true for quaternion algebras due to the fact that the standard involution is their unique symplectic involution.

(2)$\Rightarrow$(4) is easily checked. The involution of the second kind is given by $\sigma(q\otimes z)=\tau(q) \otimes z^{\ast}$, where $\tau$ is the standard involution on $B$ and $z^{\ast}$ denotes the complex conjugate of $z \in L$.
\end{proof}

It is also rather natural to define the notion of a quasi-arithmetic lattice $\Gamma<\mathrm{PSL}_2(\mathbb{C})$ as a lattice whose invariant trace field $L$ has one complex place and whose invariant quaternion algebra ramifies at all the real places, and extend to this more general setting the classification into types.
Notice that by Corollary \ref{cor:subfield} quasi-arithmetic lattices of types I and II in $\mathrm{PSL}_2(\mathbb{C})$ correspond to quasi-arithmetic lattices of the same type in $\mathrm{PO}_{3,1}$, as defined in Definitions \ref{def:type-I} and \ref{def:type-2-lattice}.

\subsubsection{Type-II lattices as arithmetic lattices in $\mathrm{PSL}_2(\mathbb{C})$}\label{sec:Nori's-argument}
Let us show how to construct a quaternion algebra $D$ over $k$, and the corresponding skew-Hermitian form $F$ on $D^2$ in the case of $3$-dimensional type II-lattices. This argument is not new; however, we could not find it anywhere in the literature.

We briefly recall the construction of the norm form (also called \textit{corestriction}) $N_{L/k}(A)$ of the quaternion algebra $A$. Given a central simple algebra $A$ of degree $m$ over $L$ and a subfield $k\subset L$ such that $[L:k]=2$, its conjugate algebra $A^c = \{a^c \,|\, a \in A\}$ is defined by the following operations:
\begin{equation*}
a^c+b^c = (a+b)^c,\; a^c \cdot b^c= (a\cdot b)^c,\; \lambda\cdot a^c= (c(\lambda)\cdot a)^c,
\end{equation*}
where $a, b \in A$, $\lambda \in L$, and $c(\lambda)$ denotes the  Galois conjugate of $\lambda$ (relatve to the subfield $k$). The switch map $s: A^c \otimes_L A \rightarrow A^c \otimes_L A$ defined by $a^c\otimes b \mapsto b^c \otimes a$ is $c$-semilinear over $L$ and is a $k$-algebra automorphism. The $k$-subalgebra 
\begin{equation*}
N_{L/k}(A) = \{ z \in A^c \otimes_L A \,|\, s(z) = z \}
\end{equation*}
of elements fixed by $s$ is a central simple $k$-algebra of degree $m^2$.
This construction induces a homomorphism of the respective Brauer groups \cite[Proposition 3.13]{INV}:
\begin{align*}
N_{L/k}: \mathrm{Br}(L) &\rightarrow \mathrm{Br}(k), \\ 
[A] &\mapsto [N_{L/k}(A)].
\end{align*}

Now, let $A$ be the quaternion algebra over a complex field $L$ associated with an arithmetic lattice $\Gamma$ in $\mathbf{PGL}_2(\mathbb{C})$. Since $A$ is a quaternion algebra, it has order $1$ (if $A = M_2(L)$) or $2$ (if $A$ is a division algebra) in the Brauer group of $L$. There are two possible cases:
\begin{itemize}
    \item[(1)] $\mathrm[N_{L/k}(A)]$ has order $1$ in $\mathrm{Br}(k)$, $N_{L/k}(A)\cong M_4(k)$ and $\Gamma$ is a type-I lattice.
    \item[(2)] $\mathrm[N_{L/k}(A)]$ has order $2$ in $\mathrm{Br}(k)$ and $N_{L/k}(A)\cong M_2(D)$, where $D$ is a division quaternion algebra over $k$.
\end{itemize}
Notice that the case where $N_{L/k}(A)$ is a division algebra of degree $4$ is excluded: since $k$ is an algebraic number field the order of $[N_{L/k}(A)]$ as an element of the Brauer group of $k$ is equal to the degree of the division algebra that is Brauer-equivalent to $N_{L/k}(A)$ \cite[p. 359]{Pierce}.

Suppose that we are in case (2) and let us fix a basis $\mathcal{B}$ on the right $D$-module $D^2$. This choice specifies a $D$-algebra isomorphism between $N_{L/k}(A)\cong M_2(D)$ and the algebra $\mathrm{End}_D(D^2)$ of $D$-linear endomorphisms of $D^2$. By \cite[Proposition 4.1 and Theorem 4.2]{INV}, there is a bijection between involutions of the first kind on $\mathrm{End}_D(D^2)$ and non-degenerate skew-Hermitian forms on $D^2$ (up to multiplication in $k^\times$). If $\phi$ is an involution of the first kind on $\mathrm{End}_D(D^2)$ and $f$ is the corresponding skew-Hermitian form, then $f$ and $\phi$ are related as follows:
\begin{equation}\label{eq:form-endomorphism-correspondence}
    f(x, g(y)) = f(\phi(g)(x), y)
\end{equation} 
for all $x$, $y$ in $D^2$ and $g \in \mathrm{End}_D(D^2)$. 
The choice of the basis $\mathcal{B}$ allows us to rewrite (\ref{eq:form-endomorphism-correspondence}) as
\begin{equation}\label{eq:quaternion-matrices}
FG=\phi(G)^*F,
\end{equation}
where $F$ (resp.\ $G$) denotes the matrix that represents the form $f$ (resp.\ the endomorphism $g$) with respect to the basis $\mathcal{B}$. The notation $M^*$ represents the matrix obtained from $M$ by transposing and applying the standard involution of $D$ to all its coefficients.

Let $\sigma$ denote the standard involution on the quaternion algebra $A$. We can define an involution $\phi$ on $A^c \otimes_L A$ by setting $\phi(a^c \otimes b)=\sigma(a)^c \otimes \sigma(b)$. It is easy to see that $\phi$ is an involution of the second kind on $A^c \otimes_L A$ whose fixed point set is the field $k$. Moreover, it commutes with the switch map $s$, so that $\phi(N_{L/k}(A))=N_{L/k}(A)$. Its restriction to $N_{L/k}(A)$, which we still denote by $\phi$ with a slight abuse of notation, is an involution of the first kind with associated skew-Hermitian form $f$ on~$D^2$.

We now define a map $i$ from the quaternion $L$-algebra $A$ to the $k$-algebra $N_{L/k}(A)\cong M_2(D)$ as follows:
\begin{equation*}
i(a)=a^c\otimes a.     
\end{equation*}
The map $i$ is multiplicative and $i(1)=1^c\otimes 1$, so that it becomes a group homomorphism when restricted to the subgroup $\mathrm{SL}_A$ of unit-norm elements of $A$. We claim that the map $i$ induces an isogeny between the group $\mathbf{SL}_A$ and $\mathbf{U}_F^{\circ}$, where $F$ is the skew-Hermitian matrix which represents the form $f$ with respect to the basis~$\mathcal{B}$.

Let us denote by $G\in M_2(D)$ the element $i(a)$, where $a\in \mathrm{SL}_A=\mathbf{SL}_A(L)$. 
Since $1 = a\cdot \sigma(a)$ we see that
\begin{equation}\label{eq:Nori}
\mathrm{id} = i(1) = i(\sigma(a)\cdot a) = i(\sigma(a))\cdot i(a) = \phi(G) \cdot G,
\end{equation}
where the last equality holds because 
\begin{equation*}
i(\sigma(a))=\sigma(a)^c \otimes \sigma(a)=\phi(a^c\otimes a)=\phi(i(a)).
\end{equation*}
By multiplying both sides of (\ref{eq:quaternion-matrices}) on the left by $G^*$ and applying (\ref{eq:Nori}), we obtain
\begin{equation*}
G^*FG = G^*\phi(G)^*F = (\phi(G)\cdot G)^*F = F,
\end{equation*}
and therefore $i(a) = G\in \mathrm{U}(F,D)=\mathbf{U}_F (k)$.

For any maximal order $O < A$, we have that $\mathrm{SL}_A(O)$ is a lattice in $\mathbf{SL}_A(\mathbb{C}) \cong \mathbf{SL}_2(\mathbb{C})$ and the image $i(\mathrm{SL}_A(O))$ is a lattice in $\mathbf{U}_F (\mathbb{R})\cong \mathrm{O}_{n,1}$, and thus Zariski-dense in $\mathbf{U}_F^{\circ}$ by Borel's density theorem. This gives rise to a surjective $k$-morphism $\mathrm{Res}_{L/k}(\mathbf{SL}_A)\to \mathbf{U}_F^{\circ}$. Moreover, $\ker(i)$ is finite as it is given by the elements of $L$ with unit norm over $k$. Indeed, $L$ is an imaginary quadratic extension of $k$ and by Dirichlet's unit theorem the rank of the group of units in the ring of integer elements over $k$ is zero. 
Thus, we have a $k$-isogeny between $\mathrm{Res}_{L/k}(\mathbf{SL}_A)$ and $\mathbf{U}_F^{\circ}$. 

Finally, we wish to go the other way around and express a type-II arithmetic lattice $\Gamma < \mathbf{PU}_F (k)$ acting on $\mathbb{H}^3$ as an arithmetic lattice in $\mathbf{PGL}_2(\mathbb{C})$. Recall that the group $\mathrm{PU}_F(D)$ is naturally identified with a subgroup of $\mathrm{PO}_f(\mathbb{R})$, for some symmetric bilinear form $f$ of signature $(3, 1)$. Hence $\Gamma$ can be regarded as a lattice in $\mathbf{PO}_f (\mathbb{R})$ and its subgroup $\Gamma' = \Gamma \cap \mathbf{PSO}_f(\mathbb{R})$ of index $2$ can be identified with an arithmetic lattice in $\mathbf{PGL}_2(\mathbb{C})$ using the exceptional isomorphism between $\mathbf{PGL}_2(\mathbb{C})$ and $\mathbf{PSO}_f(\mathbb{R}) \cong \mathrm{PO}_{3,1}^{\circ}$. We remark that by Proposition \ref{prop:type-III-not-totally-real} below the invariant trace field of the resulting lattice in $\mathbf{PGL}_2(\mathbb{C})$ is indeed an imaginary quadratic extension of a totally real field (see Remark \ref{rem:Nori-detects-type-II}).

We can therefore associate with $\Gamma'$ its invariant trace field $L$ and invariant quaternion algebra $A$, as defined in \cite[Chapter 3]{Mac-Reid}. Since $\Gamma'$ is an arithmetic lattice, we see that $L$ is a number field with exactly one pair of (conjugate) complex embeddings and $A$ is ramified at all real places of $L$ \cite[Theorem 8.3.2]{Mac-Reid}. Moreover, $\Gamma'$ will be commensurable in the wide sense with a group of the form $\mathrm{PSL}_A(O)$, where $O$ is an order in $A$ \cite[Corollary 8.3.3]{Mac-Reid}.

\subsubsection{The Vinberg invariants of type-III lattices in $\mathbf{PSO}_{3,1}(\mathbb{R})$}

Type-III arithmetic lattices acting on $\mathbb{H}^3$ exhibit a peculiar behaviour when interpreted as lattices in $\mathbf{PSO}_{3,1}(\mathbb{R})$. We begin by proving the following fact:
\begin{prop}\label{prop:type-III-not-totally-real}
Suppose that $\Gamma< \mathbf{PGL}_2(\mathbb{C})$ is a type-III arithmetic lattice, and denote by $i(\Gamma)<\mathbf{PSO}_{3,1}(\mathbb{R})$ its image under the exceptional isomorphism. The adjoint trace field $k$ of $i(\Gamma)$ is not totally real.
\end{prop}

\begin{proof}
Let $L \subset \mathbb{C}$ denote the invariant trace field of $\Gamma$, and let $K=L\cap \mathbb{R}$. By \cite[Remark 7(b)]{Vin95}, there exists $\gamma \in \Gamma$ such that $\alpha=\mathrm{tr}(\mathrm{Ad}(\gamma)) \in \mathbb{C}$ and $L=K(\alpha)$. Notice that $K$ is necessarily a totally real field. Since $\Gamma$ is a type III lattice, the extension $L/K$ is not quadratic and the minimal polynomial of $\alpha$ over $K$ has complex roots $\alpha$, $\alpha^*$ and a non-empty set of real roots $r_1,\dots,r_k$.

In particular, the extension $L/K$ has both real and complex embeddings and thus is not a Galois extension. Let us denote by $\overline{L}$ the Galois closure of $L/K$. Since the extension $L/K$ is not quadratic we have that complex conjugation is not an automorphism of $L$. In particular $\alpha^*$ does not belong to $L$ and thus we have the following sequence of field extensions:
$$K \subsetneq L=K(\alpha) \subsetneq K(\alpha,\alpha^*) \subset \overline{L},$$ where $K(\alpha)$ is a proper subfield of $K(\alpha,\alpha^*)$. It follows that there is a non-trivial field embedding $\sigma:K(\alpha,\alpha^*)\rightarrow \overline{L}$ which is the identity on $L=K(\alpha)$. Following this fact, $\sigma(\alpha)=\alpha$ while $\sigma(\alpha^*)$ is a real number in the set $\{r_1,\dots,r_k\}$ and $\sigma(\alpha+\alpha^*)= \sigma(\alpha)+\sigma(\alpha^*)$ is the sum of a non-real and a real number and is thus non-real. By Proposition \ref{prop:adjoint-rep-trace}, we have that
$\mathrm{tr}(\mathrm{Ad}(i(\gamma)))=\alpha+\alpha^* \in k$, therefore the adjoint trace field $k$ is not totally real.
\end{proof}

We mention the following corollary of the Proposition \ref{prop:type-III-not-totally-real}:
\begin{corollary}\label{cor:type-III-non-admissible}
Suppose that $\Gamma< \mathbf{PGL}_2(\CC)$ is a type III arithmetic lattice, and denote by $i(\Gamma)<\mathbf{PSO}_{3,1}(\mathbb{R})$ its image under the exceptional isomorphism. The real ambient group $\mathbf{G}$ of $i(\Gamma)$ is not admissible. Moreover, the adjoint trace field of $\Gamma$ is not contained in the invariant trace field of $\Gamma$.
\end{corollary}
\begin{proof}
By Proposition \ref{prop:type-III-not-totally-real} the adjoint trace field $k$ of $i(\Gamma)$ is not totally real. Since the invariant trace field $L$ has a single complex place, all of its subfields are totally real. If follows that $k \not \subset L$. Denote by $\sigma:k\rightarrow \mathbb{C}$ a non-real field embedding of $k$. The algebraic group $\mathbf{G}$ is $k$-defined and non-compact at the identity embedding. We claim that the complex points of the ``conjugate'' group $\mathbf{G}^{\sigma}(\mathbb{C})$ form a non-compact group, too. Arguing by contradiction, if $\mathbf{G}^{\sigma}(\mathbb{C})$ is compact it has to be a compact complex algebraic group and so it is finite by \cite[p.~134, Problem~3]{OV}. Since the conjugation map $$\mathbf{G}(k)\ni M \mapsto M^{\sigma}\in \mathbf{G}^{\sigma}(\mathbb{C})$$ is injective, it follows that the group $\mathbf{G}(k)$ is finite. This is impossible, since $\mathbf{G}(k)$ is dense in the non-compact group $\mathbf{G}(\mathbb{R})^{\circ}$.
\end{proof}

\begin{remark}\label{rem:Nori-detects-type-II}
As a consequence of Proposition \ref{prop:type-III-not-totally-real}, all type-II arithmetic lattices in $\mathrm{PSO}_{3,1}$ arise from an arithmetic lattice in $\mathbf{PGL}_2(\mathbb{C})$ via the construction described in Section \ref{sec:Nori's-argument}. This is due to the fact that type I and type-II lattices have a totally real adjoint trace field, and thus their invariant trace field as lattices in $\mathbf{PGL}_2(\mathbb{C})$ is an imaginary quadratic extension of a totally real field.
\end{remark}

\subsubsection{Constructing $L$-involutions}\label{sec:k-involutions-dim3}

As mentioned in the previous section, in the context of arithmetic lattices in $\mathbf{PGL}_2(\mathbb{C})$ the notion of adjoint trace field has to be replaced by that of the invariant trace field $L$, and the role of the  real ambient group is now taken by the complex ambient group $\mathbf{PGL}_A$, where $A$ denotes the invariant quaternion algebra. In order to construct the fc-subspaces we now have to describe the $L$-involutions in $\mathbf{PGL}_A$.

Notice that an element in $\mathbf{PGL}_A(L)$ has order $2$ if and only if it can be represented by an element $q \in A^*$ such that $q^2 \in L^*$. By taking the tensor product $A\otimes \mathbb{\mathbb{C}}$ we see that $A^*$ is mapped injectively into $\mathbf{GL}_2(\mathbb{C})$, and $q$ corresponds to a matrix $N\in \mathbf{GL}_2(\mathbb{C})$ whose square is of the form $z \cdot \mathrm{id}$ for some non-zero $z \in L$. In order for $N$ to be non-trivial in $\mathbf{PGL}_2(\mathbb{C})$ we need that $N$ is not of the form $z \cdot \mathrm{id}$. The latter is equivalent to $\mathrm{tr}\, N = 0$, which means that $q^*=-q$ (i.e. $q$ is a pure quaternion). Hence $L$-involutions correspond to traceless elements of $A^*/L^*$. Their geometric interpretation is that of a rotation of angle $\pi$ about a geodesic in $\mathbb{H}^3$ \cite[Chapter V]{Fenchel}.

There are many traceless elements in the commensurator of an arithmetic lattice \mbox{$\Gamma<\mathbf{PGL}_2(\mathbb{C})$}: 

\begin{prop}\label{prop:Jorgensen-involutions}
Let $\Gamma<\mathbf{PGL}_2(\mathbb{C})$ be an arithmetic lattice, and let $\gamma$ be a loxodromic element of $\Gamma$. There exists an involution in $\mathrm{Comm}(\Gamma)$ that acts as rotation of angle $\pi$ about the axis of $\gamma$.
\end{prop}

A geometric proof of this fact is provided  in the proof of Theorem~1.2 in \cite{LLR08} and makes use of the so-called {\itshape Jorgensen involutions}, which are order $2$ rotations around the common perpendicular to the geodesic axes of two loxodromic elements. However, the proof provided there requires a modified argument when $\Gamma$ is not cocompact. The problem in this case is that if $\gamma \in \Gamma$ is a rotation of angle $\pi$ around a geodesic $\alpha$, then there is  no guarantee {\itshape a priori} that $\alpha$ projects to a closed geodesic (see the proof of Theorem \ref{theorem:cent} for a discussion of this phenomenon). We provide here an alternative argument which only makes use of elementary linear algebra.

\begin{proof}
We denote by $L$ the invariant trace field of $\Gamma$, by $A$ its invariant quaternion algebra and by $\Gamma^{(2)}$ the (finite-index) subgroup of $\Gamma$ generated by the squares of its elements. By \cite[Theorem 1.2]{LLR08} the group $\Gamma^{(2)}$ is {\itshape derived from a quaternion algebra}, i.e. is conjugate into a subgroup of $\mathrm{PSL}_A(O)$, where $O$ is an order in the invariant quaternion algebra $A$. The element $\gamma^2$ obviously belongs to $\Gamma^{(2)}$ and is loxodromic with the same axis as $\gamma$.

Suppose that $A = \left(\frac{m,\, n}{L}\right)$ for some $m,n \in L$. We can express the elements $\gamma^2$ as a linear combination 
$$\gamma^2 = a\cdot \mathbf{1}+b\cdot \mathbf{i} + c \cdot \mathbf{j} + d \cdot \mathbf{k},\; a,b,c,d \in L $$ of the standard basis $\{\mathbf{1},\mathbf{i},\mathbf{j},\mathbf{k}\}$ of $A$ which is defined up to multiplication by $-1$. We now look for an invertible, traceless element $$q = x\cdot \mathbf{1}+y\cdot \mathbf{i} + z \cdot \mathbf{j} + w \cdot \mathbf{k} \in A$$ which commutes with $\gamma^2$,  corresponding to the required rotation of angle $\pi$ along the axis $\alpha$.

A manual computation allows to check that the condition $\gamma^2\cdot q = q \cdot \gamma^2$ yields the following homogeneous linear system in the unknowns $x,y,z,w$:

\begin{equation*}
    \begin{cases}
    dz-cw=0;\\ 
    bw-dy=0;\\ 
    bz-cy=0.
    \end{cases}
\end{equation*}

Notice that:
\begin{itemize}
    \item The system does not depend on the values of $\mathbf{i}^2=m$ or $\mathbf{j}^2=n$ or on the real part $a$ of $\gamma^2$;
    \item The unknown $x$, which correponds to the real part of $q$, appears in no equation;
    \item The $3 \times 3$ matrix built out of the coefficients of the unknowns $y,z,w$ has determinant $0$, independently of the choice of $b,c,d$.
\end{itemize}
This implies that there is always a non-zero solution with $x=0$, which corresponds to a non-zero traceless quaternion $q$ which commutes with $\gamma^2$. If $A$ is a division algebra we can immediately conclude that $q$ is also invertible. However we also wish to account for the possibility that $A$ splits, which will happen if $\Gamma$ is a non cocompact lattice.

In this case we notice that $A \otimes \mathbb{C} \cong M_2(\mathbb{C})$ and $\gamma^2$ can be represented by a matrix $M$ in $\mathbf{SL}_2(\mathbb{C})$ with eigenvalues $\lambda \neq 0$ and $\lambda^{-1}$. Suppose that $q$ is non-invertible and that it is represented by $N \in M_2(\mathbb{C})$. Since $N$ has trace $0$ it must have eigenvalue $0$ with multiplicity $2$. Now, the matrices $M$ and $N$ commute, and as such can be brought in an upper triangular form by a simultaneous change of basis and without loss of generality we can assume that they are of the following form:
\begin{equation*}
M=\begin{pmatrix}\lambda & s \\ 0 & \lambda^{-1} \end{pmatrix},\;  N=\begin{pmatrix} 0 & t \\ 0 & 0 \end{pmatrix},    
\end{equation*} 
for some $s,t \in \mathbb{C}$ with $t \neq 0$.
The fact that $M$ and $N$ commute translates to the condition $\lambda t=\lambda^{-1}t\Rightarrow\lambda^2=1\Rightarrow \lambda= \pm 1$. This implies that $\gamma^2$ is a parabolic isometry, which contradicts our assumption on $\gamma$ being loxodromic.

We conclude by noticing that the quaternion $q$ belongs to $A^*$ and therefore its image in $A^*/L^*= \mathrm{Comm}(\Gamma)$ is a rotation by angle $\pi$ about the geodesic axis of $\gamma$.
\end{proof}

\subsection{Exceptional arithmetic lattices in dimension $7$}\label{sec:type3}

As follows from the classification of semisimple algebraic groups by Tits \cite{Tits66}, 
there exist  anisotropic algebraic groups $\mathbf{G}$ defined over any number field $k$ such that $\calG=\mathrm{Gal}(\overline{k}/k)$ induces an order three (``triality'') outer automorphism of $\G(\overline{k})$, where $\overline{k}$ is the algebraic closure of the field $k$. These groups can be described as groups of automorphisms of certain {\itshape trialitarian algebras} (see \cite{Garibaldi} or \cite[Section 43]{INV}). We are interested in those examples where $k$ is totally real and $\mathbf{G}$ is an admissible $k$-form of the real group $\mathbf{PSO}_{7,1}$.

\begin{figure}[h]
    \centering
    \includegraphics[width=3.8cm]{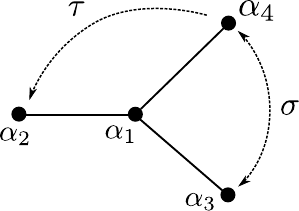}
    \caption{The Tits index of ${^6}D_{4,0}$. The reader should notice that there are no circled roots as the group is totally anisotropic. The action of the absolute Galois group $\mathcal{G} =\mathrm{Gal}(\overline{k}/k)$ is induced by complex conjugation $\sigma$, which exchanges $\alpha_2$ and $\alpha_3$, and the order $3$ trialitarian automorphism $\tau$ which permutes cyclically $\alpha_1,\, \alpha_2$ and $\alpha_3$.}
   \label{fig:Tits-diagrams}
\end{figure}

It follows from \cite[p. 336]{BC17} that, under the hypotheses above, the action of $\calG$ on the Dynkin diagram of the root system of $\G$ relative to a maximal $k$-torus is isomorphic to the symmetric group $\mathfrak{S}_3$. This can also be noticed by observing that the real algebraic group $\mathbf{PSO}_{7,1}$ is an outer form of the complex group $\mathbf{PSO}_8$, with complex conjugation acting on the Dynkin diagram via the ``folding'' automorphisms that exchanges two non-central roots. Moreover, the group $\mathbf{G}$ is always anisotropic \cite[p. 337]{BC17}. Again, this follows easily from the fact that the Tits diagram of the real group $\mathbf{PSO}_{7,1}$ has one circled root, namely the one fixed by the action of complex conjugation. It follows from this information that the Tits symbol of $\mathbf{G}$ is ${^6}D_{4,0}$ (for the Tits index, see Figure~\ref{fig:Tits-diagrams}).

\begin{definition}
Let $\Gamma < \mathbf{PO}_{7,1}(\R)$ be a lattice commensurable with $\G(\OOO)$, for some admissible triality algebraic $k$-group $\G$ as above. Then $\Gamma$ is called \textit{an arithmetic lattice of type~III}. An orbifold $M=\mathbb{H}^7/\Gamma$ is of type~III if the group $\Gamma$ is commensurable in the wide sense with an arithmetic lattice of type~III.
\end{definition}

One can similarly define the notion of a quasi-arithmetic $7$-dimensional lattice of type III, in the same way as quasi-arithmetic lattices of type I and II are defined: we simply require $\Gamma$ to be a lattice which is commensurable in the wide sense with a subgroup of $\mathbf{G}(k)$ for some admissible triality algebraic $k$-group $\mathbf{G}$. We remark that no examples of properly quasi-arithmetic trialitarian lattices are known.

The above discussion of the Tits index shows that by Godement's compactness criterion \cite{God62}, all type-III arithmetic lattices in dimension $7$ are uniform. For an explicit construction of an example see \cite{BC13}. 

\medskip
In this subsection we prove the following result:

\begin{prop}\label{prop:inner-involutions-7-dim-type-III}
Let $k$ be a totally real number field and $\mathbf{G}$ a connected adjoint algebraic $k$-group with Tits index ${^6}D_{4,0}$. Assume that the group $\mathrm{Ext}_{\mathbb{R}/k}\mathbf{G}$ obtained via extension of scalars from $k$ to $\mathbb{R}$ is isomorphic to $\mathbf{PSO}_{7,1}$. There exists an order $2$ element $\theta \in \mathbf{G}(k)$ such that fixed point set for its action on $\mathbb{H}^7$ is a totally geodesic copy of $\mathbb{H}^3$.
\end{prop}

We will apply Proposition \ref{prop:inner-involutions-7-dim-type-III} to prove Theorem \ref{teo:3-dim-type-III-in-7-dim-type-III}. Namely, we will show that
the quotient of $\mathbb{H}^3$ under the action of the centraliser in $\mathbf{G}(k)$ of the involution $\theta$ constructed in the proposition is a $3$-dimensional type-III arithmetic orbifold $N$.

\subsubsection{Constructing $k$-involutions}\label{sec:constructing-k-involutions-triality}
It is not difficult to construct order $2$ elements in the $\overline{k}$-points of an adjoint semisimple algebraic $k$-group $\mathbf{G}$ of absolute type $D_n$, with $n\geq 2$. Indeed, let $\mathbf{T} < \mathbf{G}$ be a maximal $k$-torus. Then $\mathbf{T}(\overline{k})$ is isomorphic to $(\mathbb{G}_m)^n$, where $\mathbb{G}_m$ denotes the multiplicative group of $\overline{k}$. Elements of the form $g=(\pm 1,\dots,\pm 1)$ with at least one negative entry correspond to order $2$ elements in the adjoint group $\mathbf{G}$.

We have that conjugation by $g$, which we denote $\mathrm{Inn}(g)$, is a $k$-automor\-phism of $\mathbf{G}$ if and only if it commutes with the action of the absolute Galois group $\calG$. 
The automorphism $\mathrm{Inn}(g)$ is the identity on $\mathbf{T}$ and on its Lie algebra $\mathfrak{t}$, thus it acts as the identity on the root system $\Sigma$ of $\mathbf{G}$ relative to $\mathbf{T}$. It follows that, for every $\alpha \in \Sigma$,  $\mathrm{Inn}(g)$ acts on each root space $L_{\alpha}$ as a linear transformation of the form $\pm \mathrm{id}$. Moreover by \cite[Theorem, p. 75]{Humphreys}, the action on each root space is uniquely determined by the action on the root spaces for a system of simple roots $\Delta=\{\alpha_1,\dots,\alpha_n\}$.

Indeed, denote by $\Sigma_+$ (resp.\ $\Sigma_-$) the set of {\itshape even} (resp.\ {\itshape odd}) roots given by those $\alpha \in \Sigma$ for which $\mathrm{Inn}(g)$ acts as $\mathrm{id}$ (resp.\ $-\mathrm{id}$) on the root space $L_{\alpha}$. Also, let $\Delta_+=\Sigma_+ \cap \Delta$ and  $\Delta_-=\Sigma_- \cap \Delta$. Every $\alpha \in \Sigma$ is expressed as a linear combination of the roots $\alpha \in \Delta$ with integer coefficients. Then the set $\Sigma_+$ (resp.\ $\Sigma_-$) is the set of roots which can be expressed as 
\begin{equation}\label{eq:even-odd-roots}\sum_{\alpha \in \Delta_+} n_{\alpha}\cdot \alpha + \sum_{\beta \in \Delta_-} m_{\beta}\cdot \beta\end{equation}
with $\sum_{\beta \in \Delta_-}m_{\beta}$ even (resp.\ odd). It follows that every partition of $\Delta=\Delta_+ \cup \Delta_-$ with non-empty $\Delta_{-}$ determines an involution of the form $\mathrm{Inn}(g)$ as above, and all such involutions arise in this way.

Finally, we notice that $\mathrm{Inn}(g)$ commutes with the action of the absolute Galois group $\calG$ if and only if this action preserves the partition $\Sigma=\Sigma_+ \cup \Sigma_-$ into the sets of even and odd roots. Equivalently, we require that the image of $\Delta_+$ (resp.\ $\Delta_-$) under the action of an element $\sigma \in \mathcal{G}$ lies in $\Sigma_+$ (resp.\ $\Sigma_-$). In order to prove Proposition \ref{prop:inner-involutions-7-dim-type-III} we are left with the task of finding a maximal $k$-torus $\mathbf{T}$ in a trialitarian $k$-form $\mathbf{G}$ of $\mathbf{PSO}_{7,1}$ so that the action of $\calG$ is ``small'' enough to preserve a partition of $\Sigma$ into two sets of even and odd roots.

\subsubsection{Proof of Proposition \ref{prop:inner-involutions-7-dim-type-III}}
Denote by $\calG' < \calG$ the kernel of the action of $\calG$ on the Dynkin diagram of $\mathbf{G}$, and by $E$ the fixed field for $\calG'$. The field $E$ is a degree $6$ extension of $k$, and is the smallest field such that the group $\mathrm{Ext}_{E/k}(\mathbf{G})$ obtained via extension of scalars from $k$ to $E$ is an inner form. 

By \cite[Lemma 6.29]{Pla-Rap}, there exists a totally imaginary quadratic extension $L/k$ such that $\mathrm{Ext}_{L/k}\mathbf{G}$ is quasi-split. Moreover $L$ and $E$ are linearly disjoint over $k$, meaning that the natural map $E\otimes_k L \rightarrow EL$ defined by $x \otimes y \mapsto x\cdot y$ is an isomorphism.

Since $\mathrm{Ext}_{L/k}\mathbf{G}$ is quasi-split, it follows that $\mathbf{G}$ contains an $L$-defined Borel subgroup. Denote by $\sigma$ the generator of $\mathrm{Gal}(L/k)\cong \mathbb{Z}/2\mathbb{Z}$ (i.\ e.\ the complex conjugation). By \cite[Lemma 6.17]{Pla-Rap} there exists a Borel subgroup $\mathbf{B}$ defined over $L$ such that $\mathbf{B} \cap \mathbf{B}^{\sigma}=\mathbf{T}$ is a maximal $k$-torus in $\mathbf{G}$, where $\mathbf{B}^{\sigma}$ is the image of $\mathbf{B}$ under the conjugation $\sigma$.

We now follow the discussion in \cite[p. 374]{Pla-Rap}. Denote by $\Sigma$ the root system of $\mathbf{G}$ relative to $\mathbf{T}$. The splitting field of the torus $\mathbf{T}$ (i.\ e.\ the minimal field over which $\mathbf{T}$ is isomorphic to $(\mathbb{G}_m)^4$) is the compositum $EL$, which is a degree $12$ extension of $k$. Consider the Galois automorphism $\rho \in \mathrm{Gal}(EL/k)$ which is the identity on $E$ and coincides with complex the conjugation $\sigma$ on $L$. Since $\rho$ belongs to $\mathrm{Gal}(EL/E)$ and $\mathbf{G}$ becomes an inner form over $E$, it follows that $\rho$ must act on $\Sigma$ via an element of the Weyl group $W$.

On the other hand, the automorphism $\rho$ extends to $EL$ the generator $\sigma$ of $\mathrm{Gal}(L/k)$. Since $\mathbf{T}=\mathbf{B} \cap \mathbf{B}^{\sigma}$, $\sigma$ has to take the positive roots associated to $\mathbf{B}$ under \eqref{eq:Borel-decomposition} to the negative ones. The only element of the Weyl group which exchanges a system of positive roots with a system of negative roots is the antipodal map $a$, which therefore corresponds to the action of $\rho$ on $\Sigma$.

The group $\mathrm{Gal}(EL/k)$ is the direct product generated by $\rho$ and the group $\mathrm{Gal}(EL/L) \cong \mathfrak{S}_3$. The group $\mathrm{Gal}(EL/L)$ acts faithfully on the root system $\Sigma$ by preserving the system $\Delta$ of simple roots associated to $\mathbf{B}$. It follows that the action of $\calG=\mathrm{Gal}(\overline{k}/k)$ on $\Sigma$ factors through the action of $\mathrm{Gal}(EL/k) \cong \mathfrak{S}_3 \times \mathbb{Z}/2\mathbb{Z}$, with $\mathfrak{S}_3$ acting on $\Delta$ and $\mathbb{Z}/2\mathbb{Z}$ acting via the antipodal map. 

Now, suppose that $\Delta=\{\alpha_0,\alpha_1,\alpha_2,\alpha_3\}$, with $\alpha_0$ corresponding to the ``central'' root in the Dynkin diagram (the one connected by an edge to the other three roots, see Figure \ref{fig:Tits-diagrams}). Let us choose any partition $\Delta=\Delta_+ \cup \Delta_-$ (with non-empty $\Delta_-$) which is preserved by the action of the $\mathfrak{S}_3$ factor of $\mathrm{Gal}(EL/k)$. There are only $3$ possible choices:
\begin{align} 
\Delta_-=\{\alpha_0\}, &\; \Delta_+=\{\alpha_1,\alpha_2,\alpha_3\}; \label{eq:root-partition-1} \\ 
\Delta_+=\{\alpha_0\}, &\; \Delta_-=\{\alpha_1,\alpha_2,\alpha_3\}; \label{eq:root-partition-2} \\ 
\Delta_-=\Delta,\;\;\;\; &\; \Delta_+=\emptyset. \label{eq:root-partition-3}
\end{align}
Each choice will determine a partition of $\Sigma=\Sigma_+ \cup \Sigma_-$ into two sets of even and odd roots which is necessarily preserved also by the antipodal map $\alpha$. It follows that the whole of $\calG$ preserves this partition, thus it commutes with  $\mathrm{Inn}(g)$. This implies that $\mathrm{Inn}(g)$ is defined over $k$. The proof follows by taking $\theta=\mathrm{Inn}(g)$. 

We are left with the task of verifying that the fixed point set for the action of $\theta$ on $\mathbb{H}^7$ is $3$-dimensional. We notice first of all that $\theta$ corresponds to an element of $\mathrm{PSO}_{7,1}$, i.\ e.\ to an orientation preserving isometry of $\mathbb{H}^7$. Given that $\theta^2=\mathrm{id}$, it follows that the action of $\theta$ on the Lie algebra $\mathfrak{g}$ is diagonalisable and has eigenvalues $1$ and $-1$ with multiplicities $m(1)$ and $m(-1)$ respectively. The multiplicities can be computed easily: $\theta$ acts as the identity on the $4$-dimensional algebra $\mathfrak{t}$ and on each one-dimensional root space $L_{\alpha}$ for $\alpha \in \Sigma_+$ an even root. It also acts as $-\mathrm{id}$ on each root space $L_{\alpha}$ for $\alpha \in \Sigma_-$ an odd root.

For all the possible choices (\ref{eq:root-partition-1},~\ref{eq:root-partition-2},~\ref{eq:root-partition-3})  for $\Delta_+$ and $\Delta_-$ we obtain that $\Sigma_+$ has $8$ roots and $\Sigma_-$ has $16$ roots. It follows that we have $m(1)=12,\, m(-1)=16$. Since the resulting involution $\theta$ has order $2$, it can be represented by conjugation by a matrix $M \in \mathrm{SO}_{7,1}^{\circ}$ such that $M^2=\mathrm{id}$. Up to conjugacy in $\mathrm{SO}_{7,1}$ we may assume that $M$ is diagonal with $\pm 1$ entries on the diagonal.
The only possibility such that the action on the Lie algebra has $m(1)=12$, $m(-1)=16$ is that $M$ has $4$ entries equal to $1$ and $4$ entries equal to $-1$. Such an $M$ corresponds to a reflection along a $3$ dimensional totally geodesic subspace in $\mathbb{H}^7$. 
\qed

We now proceed to prove Theorem \ref{teo:3-dim-type-III-in-7-dim-type-III}. In what follows we will make essential use of Theorems \ref{theorem:cent}, \ref{theorem:geod} and Remark \ref{prop:invariant-trace-field-inclusion}, whose proofs we postpone to Section \ref{sec:hereditary properties} (the proofs of these theorems do not require any results from Section~\ref{sec:type3}).

\subsubsection{Proof of Theorem \ref{teo:3-dim-type-III-in-7-dim-type-III}}

Let $N$ denote the fc-subspace corresponding to the centraliser of the involution $\theta$. By Theorem \ref{theorem:cent} it has finite volume, and by Theorem \ref{theorem:geod} it is arithmetic. The rest of the proof is devoted to showing that the adjoint trace field of $N$ is not totally real. Proposition~\ref{rem:Nori-detects-type-II} then implies that $N$ is a $3$-dimensional type III arithmetic hyperbolic orbifold.

We carry over the notation from the proof of Proposition \ref{prop:inner-involutions-7-dim-type-III}. In particular, $E$ denotes the minimal field over which $\mathrm{Ext}_{E/k}\mathbf{G}$ becomes an inner form and $L$ is the field over which $\mathrm{Ext}_{L/k}\mathbf{G}$  becomes quasi-split. These fields are linearly independent over $k$, which implies that the map $\mathrm{Gal}(E/k) \times \mathrm{Gal}(L/k) \rightarrow \mathrm{Gal}(EL/k)$ given by 
$$(\phi,\eta)(x \cdot y) \mapsto \phi(x) \cdot \eta(y)$$ for $\phi \in \mathrm{Gal}(E/k)$, $\eta \in \mathrm{Gal}(L/k)$, $x \in E$ and $y \in L$ is an isomorphism. Under this map the groups $\mathrm{Gal}(E/k)$ and $\mathrm{Gal}(L/k)$ are mapped to $\mathrm{Gal}(EL/L)$ and $\mathrm{Gal}(EL/E)$ respectively.  We identify the root system of $\mathbf{G}$ relative to $\mathbf{T}$ with the $24$ vectors in $\mathbb{R}^4$ obtained via permutations in the entries of $(\pm1,\pm1,0,0)$, and the system $\Delta=\{\alpha_0,\alpha_1,\alpha_2,\alpha_3\}$ of simple roots corresponding to the $L$-defined Borel subgroup $\mathbf{B}$ is identified with 
\begin{equation*}
\alpha_0=(0,1,-1,0),\, \alpha_1=(1,-1,0,0),\, \alpha_2=(0,0,1,-1),\, \alpha_3=(0,0,1,1).
\end{equation*}

As a first step, we analyse the action of complex conjugation $\sigma$ on the $D_4$ root system of $\mathbf{G}$ relative to the torus $\mathbf{T}$. Complex conjugation induces a nontrivial automorphism of both factors, since $E$ and $L$ are both imaginary fields. On the group $\mathrm{Gal}(E/k)\cong \mathrm{Gal}(EL/L) \cong \mathfrak{S}_3$, the Galois automorphism $\sigma$ corresponds to the permutation $p$ of the simple root system $\Delta$ that exchanges two non-central roots. We may assume without loss of generality that these are $\alpha_2$ and $\alpha_3$, so that the associated map is a change of sign in the last coordinate. On $\mathrm{Gal}(L/k) \cong \mathrm{Gal}(EL/E)$, $\sigma$ corresponds to the antipodal map $a$. The full action of $\sigma$ on the $D_4$ root system is the composition $p \circ a$ which can be written as 
\begin{equation}\label{eq:conjugation-action}(x,y,z,w) \overset{\sigma}{\mapsto} (-x,-y,-z,w).\end{equation}

Concerning the ``trialitarian'' automorphism $\tau \in \mathrm{Gal}(E/k) \cong \mathrm{Gal}(EL/L)$, its action on the $D_4$ root system is given by a cyclic permutation of the roots $\alpha_1,\alpha_2,\alpha_3$:

\begin{equation}\label{eq:triality-action}
(x,y,z,w) \overset{\tau}{\mapsto} \frac{1}{2}\cdot (x+y+z+w,\;x+y-z-w,\;x-y+z-w,\;-x+y+z-w).
\end{equation}

With this information we are able to describe the Tits index of the centraliser $\mathbf{H}$ of the involution $\theta$, which is naturally a $k$-subgroup of $\mathbf{G}$. Notice that the torus $\mathbf{T}$ constructed in the proof of Proposition \ref{prop:inner-involutions-7-dim-type-III} is a maximal $k$-torus of $\mathbf{H}$.

Let us assume that the choice of the partition $\Delta=\Delta_+ \cup \Delta_-$ is the one given in \eqref{eq:root-partition-1} (the other choices yield exactly the same results). The positive roots in $\Sigma_+$ are those of the form 
\begin{equation}\label{eq:centraliser-root-system}
(\pm1,\pm1,0,0),\; (0,0,\pm1,\pm1)
\end{equation}
which form a root system of type $D_2 \times D_2 = A_1 \times A_1 \times A_1 \times A_1$ for $\mathbb{H}$. Its Dynkin diagram is given by the disjoint union of $4$ vertices labeled $\alpha,\beta,\gamma$ and $\delta$, each vertex corresponding to a pair of opposite roots:
\begin{equation*}
\alpha=\pm(0,0,1,-1),\; \beta=\pm (0,0,1,1),\; \gamma=\pm (1,-1,0,0),\; \delta= \pm(1,1,0,0).
\end{equation*}
With the actions of complex conjugation and triality given in \eqref{eq:conjugation-action} and \eqref{eq:triality-action} we see that the action of $\calG=\mathrm{Gal}(\overline{k}/k)$ on the Dynkin diagram is as in Figure \ref{fig:involution-centraliser-diagram}.

\begin{figure}[h]
    \centering
    \includegraphics[width=2.6cm]{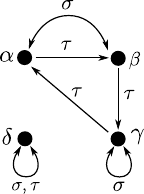}
    \caption{The Tits index (relative to the torus $\mathbf{T}$) of the centraliser $\mathbf{H}$ of the involution $\theta$. The arrows indicate the action on the Dynkin diagram of complex conjugation $\sigma$ and of the trialitarian automorphism $\tau \in \mathrm{Gal}(E/k)$.}
   \label{fig:involution-centraliser-diagram}
\end{figure}

In particular, complex conjugation $\sigma$ preserves the two pairs of roots $\{\alpha,\beta\}$ and $\{\gamma,\delta\}$.
It follows that there exists an $\mathbb{R}$-defined isomorphism of connected adjoint algebraic groups 
$\mathbf{PH}^{\circ} \cong \mathbf{L} \times \mathbf{C}$, where $\mathbf{PH}^{\circ}$ denotes the adjoint group of the identity component of $\mathbf{H}$, the group $\mathbf{L}$ corresponds to the $D_2$ root subsystem $(0,0,\pm1,\pm1)$ spanned by $\alpha$ and $\beta$, and the group $\mathbf{C}$ corresponds to the root subsystem $(\pm1, \pm1,0,0)$ spanned by $\gamma$ and $\delta$. Since the roots $\alpha$ and $\beta$ are swapped by $\sigma$, it follows that $\mathbf{L}$ has Tits symbol ${^2}D^{(1)}_{2,1}$ and is therefore isomorphic to $\mathbf{PSO}_{3,1}$ by \cite[p. 57]{Tits66}. Concerning the group $\mathbf{C}$, we notice that complex conjugation $\sigma$ acts by exchanging each root with its opposite and therefore has Tits symbol ${^1}D^{(1)}_{2,0}$, implying that $\mathbf{C}$ is $\mathbb{R}$-isomorphic to $\mathbf{PSO}_4$.

 Recall that the minimal field $E$ such that $\mathrm{Ext}_{E/k}\mathbf{G}$ is an inner form is an imaginary Galois extension of the totally real field $k$ and $\mathrm{Gal}(E/k)$ is isomorphic to the symmetric group $\mathfrak{S}_3$. Denote by $K$ the fixed field of complex conjugation $\sigma$, now interpreted as an element of $\mathrm{Gal}(E/k)$. The field $K$ is a cubic extension of $k$ and the set $S^{\infty}_{K/k}$ of embeddings of $K$ relative to $k$ contains one real embedding (corresponding to $K$) and two complex conjugate embeddings, which we denote by $\tau(K)$ and $\tau^2(K)$. The action of the trialitarian automorphism $\tau$ permutes these $3$ embeddings cyclically, while complex conjugation $\sigma$ fixes $K$ and exchanges $\tau(K)$ and $\tau^2(K)$.

It follows that the $D_2$ root systems spanned by $\{\alpha,\beta\}$ is preserved by the action of the absolute Galois group $\mathrm{Gal}(\overline{K}/k)$. The minimal field of definition for the projection of $\mathbf{PH}^{\circ}\rightarrow \mathbf{L}$ is precisely $K$ and, as shown in the proof of Theorem \ref{theorem:geod} (see Remark \ref{lemma:atf-and-projection-field-coincide}), this is the adjoint trace field of $N$. 
Since $K$ is not totally real, 
we see that $N$ is a $3$-dimensional type-III arithmetic hyperbolic orbifold.
\qed

We conclude this section with the following supplementary result:
\begin{corollary}
The invariant trace field of the $3$-dimensional type-III orbifold $N=\mathbb{H}^3/\Gamma$ of Theorem \ref{teo:3-dim-type-III-in-7-dim-type-III} is the complex embedding $\tau(K)$ of the adjoint trace field $K$. In particular, it is a cubic extension of a totally real field.
\end{corollary}

\begin{proof}
The adjoint group $\mathbf{PH}^{\circ}$ of $\mathbf{H}^{\circ}$ decomposes over the algebraic closure $\overline{k}$ of $k$ as a direct product:
\begin{equation*}
\mathbf{PH}^{\circ} \cong \mathbf{PGL}_2 \times \mathbf{PGL}_2 \times \mathbf{PSO}_4,
\end{equation*}
where the $\mathbf{PGL}_2$-factors correspond to the roots $\alpha$ and $\beta$ respectively, while the $\mathbf{PSO}_4$-factor corresponds to the pair of roots $\{\gamma,\delta\}$.
By Remark \ref{prop:invariant-trace-field-inclusion}, the invariant trace field of $N$ is the minimal field over which the projection $\mathbf{PH}^{\circ}\rightarrow \mathbf{PGL}_2$ onto the first factor is defined. The analysis of the Tits index of $\mathbf{PH}^{\circ}$ (see Figure \ref{fig:involution-centraliser-diagram}) shows that this field is precisely the fixed field for the automorphism $ \tau \sigma \tau^{-1} \in \mathrm{Gal}(E/k)$, i.\,e.\ the complex embedding $\tau(K)$ of~$K$.
\end{proof}

\section{Two kinds of totally geodesic subspaces}\label{sec:kinds-of-subspaces}

In this section, we introduce two techniques to construct totally geodesic immersions of arithmetic hyperbolic orbifolds into other arithmetic hyperbolic orbifolds. In Section \ref{sec:hereditary properties} we shall show that any totally geodesic immersion of arithmetic hyperbolic orbifolds is obtained through a combination of these two techniques. 

\subsection{Subform subspaces}

We present here a simple generalisation of the method used in \cite[Proposition 5.1]{KRS} to construct codimension-one totally geodesic embeddings of type-I arithmetic lattices.

\begin{prop}\label{sec:subform-space-type-I}
Let $\Lambda<\mathrm{PO}_f(k)$ be a type-I arithmetic lattice associated to an admissible form $f$ of signature $(m,1)$ defined over a totally real field $k$. For any integer $n>m$, there exists an admissible $k$-defined form $g$ of signature $(n,1)$ and a type-I arithmetic lattice $\Gamma < \mathrm{PO}_g(k)$ such that $N=\mathbb{H}^m/\Lambda$ is a totally geodesic suborbifold of $M=\mathbb{H}^n/\Gamma$. 
\end{prop}
\begin{proof}
Consider any $k$-defined form $h$ with the property that $h^{\sigma}$ has signature $(n-m,0)$ for any field embedding $\sigma:k\rightarrow \mathbb{R}$. Define $g$ to be the orthogonal direct sum of the quadratic spaces associated to the forms $f$ and $h$. Then $g$ is an admissible $k$-form and we obtain an inclusion of $\mathbf{O}_f$ into $\mathbf{O}_g$. By applying \cite[Proposition 2.1]{KRS}, we obtain that the arithmetic lattice $\Lambda < \mathbf{O}_f(k)$ is realised as a totally geodesic sublattice of an arithmetic lattice $\Gamma < \mathbf{O}_g(k)$. We conclude by projecting $\Lambda$ and $\Gamma$ to $\mathrm{PO}_f(k)$ and $\mathrm{PO}_g(k)$, respectively.
\end{proof}

Note that all the fc-subspaces arising as fixed point sets of involutions in $\mathrm{PO}_f(k)$ as described in \ref{sec:k-involutions-type1} fit into the description above. 

The construction of subform subspaces has a natural generalisation to type-II lattices. It is very similar to the type-I case, so we skip the details. 

\subsection{Subspaces via Weil restriction of scalars}\label{sec:weil-res-type-I}

The discussion at the end of Section \ref{sec:k-involutions-type1} suggests a technique to construct further examples of totally geodesic immersions of type-I lattices into other type-I lattices, of perhaps much higher dimension.

\begin{prop}\label{prop:res_of_scalars_typeI}
Let $\Lambda<\mathrm{PO}_f(K)$ be a type-I arithmetic lattice associated to an admissible form of signature $(n,1)$ defined over a totally real field $K$. Let $k$ be a subfield of $K$ such that $[K:k]=d$. Then $\Lambda$ is a totally geodesic sublattice of a type-I arithmetic lattice $\Gamma < \mathrm{PO}_g(k)$ associated to an admissible $k$-defined form $g$ of signature $(d(n+1)-1,1)$. 
\end{prop}

\begin{proof}
By the primitive element theorem we can suppose that $K=k(\alpha)$ and that $p(x) \in k[x]$ is the minimal polynomial of $\alpha$ over $k$, with $d$ distinct real roots $\alpha=\alpha_0, \alpha_1,\dots, \alpha_{d-1}$. The field $K$ is then isomorphic to an abstract extension of $k$:
$$K\cong k[x]/(p(x)).$$

Let us denote by $S^{\infty}_{K/k}$ the set of all field embeddings $\sigma:K\rightarrow \mathbb{R}$ which restrict to the identity on $k$. There are precisely $d=[K:k]$ such embeddings, so that $S^{\infty}_{K/k}=\{\sigma_0,\sigma_1,\dots,\sigma_{d-1}\}$. We assume that $\sigma_i(x)=\alpha_i$, so that $\sigma_0=\mathrm{id}|_{K}$ and let $\mathrm{Gal}(\overline{K}/k)$ denote the Galois group of the Galois closure $\overline{K}$ of the extension $K/k$, which naturally acts faithfully and transitively on $S^{\infty}_{K/k}$.

Let $(V,f)$ be a $K$-defined admissible quadratic space, i.e.\ $V$ is an $(n+1)$-dimensional vector space over $K$ and $f$ is an admissible symmetric bilinear form of signature $(n,1)$. By fixing a basis for $V$, we may assume that $V \cong K^{n+1}$.

For each embedding $\sigma \in S^{\infty}_{K/k}$, we build a vector space $V^{\sigma}$ of dimension $n+1$ over $\sigma(K)$ as follows:
\begin{itemize}
    \item $V^{\sigma}=\{v^{\sigma}|v \in V\}$,
    \item $v^{\sigma}+w^{\sigma}=(v+w)^{\sigma}$,
    \item $\sigma(\lambda)\cdot v^{\sigma}=(\lambda \cdot v)^{\sigma}, \lambda \in K$,
\end{itemize} where it is understood that $V=V^{\sigma_0}$.
Note that $V^{\sigma}$ can be naturally interpreted as an $(n+1)$-dimensional vector space over the abstract field $k[x]/(p(x))$ by requiring that $$[q(x)]\cdot v^{\sigma}= q(\sigma(x))\cdot v^{\sigma}.$$

We now build a vector space $W$ over $k[x]/(p(x))$ of dimension $(n+1)\cdot d$ by considering the direct sum
\begin{equation}\label{eq:W-decomposition-typeI}
W= V \oplus V^{\sigma_1} \oplus \dots \oplus V^{\sigma_{d-1}}.\end{equation}

Notice that the group $\mathrm{Gal}(\overline{K}/k)$ acts $k$-linearly on the vector space $W$. If $\sigma \in \mathrm{Gal}(\overline{K}/k)$ and $v^{\sigma_i} \in V^{\sigma_i}$, then $$\sigma(v^{\sigma_i})=v^{\sigma\circ \sigma^i} \in V^{\sigma\circ \sigma^i}.$$

Hence we can define the $k$-subspace $\mathrm{Res}_{K/k}\,V$ of fixed points of the action of $\mathrm{Gal}(\overline{K}/k)$:
$$
\mathrm{Res}_{K/k}\,V=\{v+v^{\sigma_1}+\dots+v^{\sigma_{d-1}}|v \in V\}.
$$

Furthermore, for each $\sigma \in S^{\infty}_{K/k}$, a $\sigma(K)$-defined symmetric  bilinear form $f^{\sigma}$ on $V^{\sigma}$ is given by
$$f^{\sigma}(v^{\sigma},w^{\sigma})=(f(v,w))^\sigma.$$ 

By endowing each factor of the direct sum decomposition \eqref{eq:W-decomposition-typeI} of $W$ with the corresponding form $f^{\sigma_i}$ and imposing the various factors to be pairwise orthogonal, we define a symmetric bilinear form $h$ on $W$ so that
$$(W,h) = (V,f) \oplus (V^{\sigma_1},f^{\sigma_1}) \oplus \ldots \oplus (V^{\sigma_{d-1}},f^{\sigma_{d-1}}).$$

We now claim that the restriction of the bilinear form $h$ to the $k$-subspace $\mathrm{Res}_{K/k}\,V$ is a $k$-defined symmetric bilinear form, which we denote by $\mathrm{Res}_{K/k}\,f$. Thus $(\mathrm{Res}_{K/k}\,V, \mathrm{Res}_{K/k}\,f)$ is a $k$-defined quadratic space which we think of as the restriction of scalars of the $K$-defined quadratic space $(V, f)$ to $k$.
In order to prove the claim it is sufficient to notice that $h:W \times W \rightarrow k[x]/(p(x))$ is $k$-linear  (where $W$ and $k[x]/(p(x))$ are now interpreted as $k$-vector spaces) and that for all $v, w \in V$: 
$$h\left( \sum_{i=0}^{d-1} v^{\sigma_i},\sum_{j=0}^{d-1} w^{\sigma_j} \right)=\sum_{i=0}^{d-1}f^{\sigma_i}(v^{\sigma_i},w^{\sigma_i})=\sum_{i=0}^{d-1}(f(v,w))^{\sigma_i}=\mathrm{tr}(f(v,w)) \in k,$$ where $\tr:K\rightarrow k$ is the trace of the field extension $K/k$.

We note that the real vector space $(\mathrm{Res}_{K/k}\,V)\otimes_k \mathbb{R}$ admits a direct sum decomposition as:
\begin{equation}\label{eq:direct-sum-decomposition} (\mathrm{Res}_{K/k}\,V)\otimes_k \mathbb{R}= (V\otimes_K \mathbb{R})\oplus (V^{\sigma_1}\otimes_{\sigma_1(K)}\mathbb{R}) \oplus \dots\oplus (V^{\sigma_{d-1}}\otimes_{\sigma_{d-1}(K)} \mathbb{R}).
\end{equation}
It is important to notice that the subspace $V^{\sigma_i} \otimes \mathbb{R}$ does not correspond to the tensor product of $\mathbb{R}$ with a $k$-subspace of $\mathrm{Res}_{K/k}\,V$. The spaces $V^{\sigma_i}$, $i=0,\dots,d-1$, are each defined over a different field embedding of $K$, so the direct sum decomposition \eqref{eq:direct-sum-decomposition} can be defined at best over the Galois closure of $K/k$.

Let us denote the $k$-defined form $\mathrm{Res}_{K/k}\,f$ by $g$.
The group of real points of $\mathrm{Res}_{K/k}\mathbf{O}_f$ is isomorphic to
$$\mathbf{O}_f(\mathbb{R})\times \mathbf{O}_{f^{\sigma_1}}(\mathbb{R})\times  \dots\times \mathbf{O}_{f^{\sigma_{d-1}}}(\mathbb{R})\subset\mathbf{O}_{g}(\mathbb{R}),$$ which implies that the group $\mathrm{Res}_{K/k}\mathbf{O}_f(\mathbb{R})$ corresponds to the subgroup of $\mathbf{O}_{g}(\mathbb{R})$ that preserves the direct sum decomposition \eqref{eq:direct-sum-decomposition}.
This fact translates to a $k$-defined inclusion of $\mathrm{Res}_{K/k} \mathbf{O}_f$ into $\mathbf{O}_{g}$, and we can interpret the form $f$ as the restriction to $V$ of the $k$-defined form $\mathrm{Res}_{K/k}\,f$ of signature $(d(n+1)-1,1)$. The admissibility of $\mathrm{Res}_{K/k}\,f$ is easy to check. Indeed, any field embedding $\eta:k \rightarrow \mathbb{R}$ extends to $K$. The form $(\mathrm{Res}_{K/k}\,f)^{\eta}$ can be represented over $\mathbb{R}$ as $f^{\eta}\oplus f^{\eta \circ \sigma_1}\oplus\dots\oplus f^{\eta \circ \sigma_{d-1}}$, which will be positive definite for any $\eta \neq \mathrm{id}|_{k}$ due to the fact that $f$ is admissible.

We thus obtain that $\mathrm{PO}_{g}(\mathcal{O}_k)$ is a type-I arithmetic lattice with field of definition $k$. The stabiliser of the subspace $V\otimes \mathbb{R} \subset \mathrm{Res}_{K/k}(V) \otimes \mathbb{R}$ is commensurable with $\mathrm{Res}_{K/k}\mathbf{O}_f (\mathcal{O}_k)$, and the projection of this group to $\mathbf{PO}_f(\mathbb{R})$ is commensurable with $\mathrm{PO}_f(\mathcal{O}_K)$. More generally, by applying \cite[Proposition 2.1]{KRS} with $G=\mathrm{Res}_{K/k}\mathbf{O}_f$ and $H=\mathbf{O}_{g}$, we see that any arithmetic lattice $\Lambda < \mathbf{O}_f(K)$ is realised as a totally geodesic sublattice of an arithmetic lattice $\Gamma < \mathbf{O}_{g}(k)$. We conclude by projecting $\Lambda$ and $\Gamma$ to $\mathrm{PO}_f(K)$ and $\mathrm{PO}_{g}(k)$, respectively.
\end{proof}

\begin{remark}
The statement of Proposition \ref{prop:res_of_scalars_typeI} contradicts Proposition~9.1 of \cite{Mey17}. Indeed the fields of definition of $N=\mathbb{H}^n/\Lambda$ and $M=\mathbb{H}^m/\Gamma$ are different, and $N$ is not a subform subspace of $M$ as defined in \cite[Construction 4.11]{Mey17}. We will see in Corollary \ref{cor:max-dim-subform-space} that this issue can be corrected by an extra assumption on the codimension of $N$ in $M$.
\end{remark}

\begin{remark}
The condition $\Lambda< \mathrm{PO}_f(K)$ is not very restrictive. By \cite[Lemma 4.5]{ERT}, if $\Lambda<\mathbf{PO}_f (K)$ is a type-I arithmetic lattice, then the finite-index subgroup $\Lambda^{(2)}$ generated by the squares of the elements of $\Lambda$ lies in $\mathrm{PO}_f(K)$.
\end{remark}

Proposition \ref{prop:res_of_scalars_typeI} allows to construct plenty of examples of totally geodesic immersions of compact arithmetic hyperbolic orbifolds into non-compact ones.

\begin{corollary}
Let $N=\mathbb{H}^m/\Lambda$, $m\geq 2$ be a compact type-I arithmetic hyperbolic orbifold such that $\Lambda<\mathbf{O}_f(K)$, with $f$ being an admissible form of signature $(m,1)$ defined over a totally real algebraic number field $K$ such that $[K:\mathbb{Q}]=d>1$. Then $N$ is realised as a totally geodesic immersed suborbifold in a non-compact type-I arithmetic hyperbolic orbifold of dimension $n=d\cdot(m+1)-1$.
\end{corollary}

\begin{proof}
Apply Proposition \ref{prop:res_of_scalars_typeI} with $k=\mathbb{Q}$ in order to build a totally geodesic immersion of $N$ into $M=\mathbb{H}^n/\Gamma$, where $\Gamma<\mathbf{O}_g(\mathbb{Q})$ is arithmetic and $g$ is an admissible, $\mathbb{Q}$-defined form of signature $(d\cdot(m+1)-1,1)$. Since $d\cdot(m+1)\geq 6$ the form $g$ is isotropic by Meyer's theorem, and thus $M$ is non-compact.
\end{proof}

We now turn our attention to the case of type-II lattices and describe embeddings via Weil restriction of scalars in this setting.

\begin{prop}\label{prop:Weil-restriction-type-II}
Let $\Lambda < \mathrm{PU}_F(D)$ be a type I or II arithmetic lattice associated to an admissible skew-Hermitian form $F$ of signature $(2m-1, 1)$ defined over a quaternion algebra $D$ over a totally real number field $K$. Let $k$ be a subfield of $K$ such that $[K:k] = d$ and $D \cong D' \otimes K$ for some quaternion algebra $D'$ over $k$. Then $\Lambda$ is a totally geodesic sublattice of an  arithmetic lattice $\Gamma < \mathrm{PU}_G(D')$ associated to an admissible $k$-defined skew-Hermitian form $G$ of signature $(2 d m - 1, 1)$ defined over $D'$.

Moreover, if $D'\cong M_2(k)$ then $\Gamma$ and $\Lambda$ are type-I lattices. Otherwise, $D'$ is a division algebra and $\Gamma$ is a type-II lattice. In this case, $\Lambda$ is a type-II lattice if $D$ is a division algebra, while it is a type-I lattice if $D \cong M_2(K)$.
\end{prop}

\begin{proof}
We follow the same strategy as in the proof of Proposition \ref{prop:res_of_scalars_typeI}, and carry over the notation. For each embedding $\sigma \in S^{\infty}_{K/k}$, we define the conjugate quaternion algebra $D^{\sigma}$ over $\sigma(K)$ as follows:
\begin{itemize}
    \item $D^{\sigma} = \{a^{\sigma} |\, a \in D \}$,
    \item $a^{\sigma} + b^{\sigma} = (a+b)^{\sigma}, a^{\sigma} \cdot b^{\sigma}= (a\cdot b)^{\sigma}$,
    \item $\sigma(\lambda) \cdot a^{\sigma} = (\lambda \cdot a)^{\sigma}.$
\end{itemize}

Notice that each algebra $D^{\sigma}$ is isomorphic to $D'\otimes \sigma(K)$, and $(q\otimes \lambda)^{\sigma}=q \otimes \sigma(\lambda)$ for all $q \in D', \lambda \in K$.

Also, each $D^{\sigma}$ can be naturally interpreted as a quaternion algebra over the abstract field $k[x]/(p(x))$. Let us consider the direct sum
$$
D \oplus D^{\sigma_1} \oplus \ldots \oplus D^{\sigma_{d-1}},
$$ 
which contains the $k$-subalgebra $\mathrm{Res}_{K/k}\,D$ of elements invariant under the action of $\mathrm{Gal}(\overline{K}/k)$ that maps $a^{\sigma_i} \in D^{\sigma_i}$ to $a^{\sigma \circ \sigma_i} \in D^{\sigma \circ \sigma_i}$ for each $\sigma \in \mathrm{Gal}(\overline{K}/k)$. Observe that $\mathrm{Res}_{K/k}\,D$ has the structure of a right $D'$-module of rank $d$.

Now let us build the right $D^{\sigma}$-module $(D^{\sigma})^m$ of rank $m$ and consider the direct sum
\begin{equation}\label{eq:W-decomposition-typeII}
W = D^m \oplus (D^{\sigma_1})^m \oplus \ldots \oplus (D^{\sigma_{d-1}})^m.
\end{equation}
We define the subset $\mathrm{Res}_{K/k}\,D^m \subset W$ of fixed points under the action of $\mathrm{Gal}(\overline{K}/k)$:
$$
\mathrm{Res}_{K/k}\,D^m = \{x+x^{\sigma_1} + \ldots + x^{\sigma_{d-1}}|\, x \in D^m \}.
$$ 
Note that $\mathrm{Res}_{K/k}\,D^m \cong (\mathrm{Res}_{K/k}\,D)^m$ is naturally a right $\mathrm{Res}_{K/k}(D)$-module of rank $m$ and is thus a right $D'$-module of rank $d \cdot m$.

For each $\sigma \in S^{\infty}_{K/k}$, let $F^\sigma$ be the skew-Hermitian form on $(D^\sigma)^m$ defined by
$$
F^\sigma(a^\sigma, b^\sigma) = (F(a, b))^\sigma.
$$

By endowing each factor of the direct sum decomposition \eqref{eq:W-decomposition-typeII} of $W$ with the corresponding form and imposing the various factors to be pairwise orthogonal we define a skew-Hermitian form $H$ on $W$ with values in a quaternion algebra over the field $k[x]/(p(x))$.

We claim that the restriction of $H$ to the right $D'$-module $\mathrm{Res}_{K/k}\,D^m$ is a skew-Hermitian form, which we interpret as the restriction of scalars of the form $F$ and denote by $\mathrm{Res}_{K/k}\,F$. In order to prove the claim we proceed as follows. It is clear that $H$ is a sesquilinear form on the right $D'$-module $\mathrm{Res}_{K/k}\,D^m$. We claim that on this set it actually takes values in $D'$.

Let us define the following $k$-linear trace function $\mathrm{Tr}:D\rightarrow D'$ by setting $\mathrm{Tr}(q \otimes \lambda) = q \otimes \mathrm{tr}(\lambda)$. 
Let $\sigma_0 = \mathrm{id}$, so that we have
$$
H\left( \sum_{i=0}^{d-1}x^{\sigma_i},\sum_{j=0}^{d-1}y^{\sigma_j} \right) = \sum_{i=0}^{d-1}F^{\sigma_i}(x^{\sigma_i}, y^{\sigma_i}) = \sum_{i=0}^ {d-1}(F(x,y))^{\sigma_i}=\mathrm{Tr}(F(x, y)) \in D'. 
$$

Since the form $F$ is skew-Hermitian and $(q\otimes \lambda)^* = q^* \otimes \lambda$ for all $q \in D'$ and  $\lambda \in K$, it follows that the form $H$ is also skew-Hermitian.

We now notice that $\mathrm{Res}_{K/k}\,D^m \otimes \mathbb{R}$ is a right $D' \otimes \mathbb{R}$-module of rank $d m$ and it admits the following decomposition:
$$\mathrm{Res}_{K/k}\,D^m \otimes_k \mathbb{R} = (D\otimes_K \mathbb{R})^m \oplus (D^{\sigma_1}\otimes_{\sigma_1(K)} \mathbb{R})^m \oplus \ldots \oplus (D^{\sigma_{d-1}}\otimes_{\sigma_{d-1}(K)} \mathbb{R})^m.$$

Denote the form $\mathrm{Res}_{K/k}\, F$ by $G$.
The group of real points of $\mathrm{Res}_{K/k}\mathbf{U}_F$ is isomorphic to
\begin{small}
$$
\mathbf{U}_F(D \otimes \mathbb{R}) \times \mathbf{U}_{F^{\sigma_1}}(D^{\sigma_1} \otimes \mathbb{R})\times \ldots \times \mathbf{U}_{F^{\sigma_{d-1}}}(D^{\sigma_{d-1}} \otimes \mathbb{R}) \subset \mathbf{U}_G(D'\otimes \mathbb{R}),
$$
\end{small}and the inclusion $\mathrm{Res}_{K/k}\,\mathbf{U}_F \subset \mathbf{U}_G$ is defined over $k$. As in the case of type~I lattices, the form $F$ is now interpreted as the restriction to $D^m$ of the form $G=\mathrm{Res}_{K/k}(F)$, which clearly has signature $(2dm-1,1)$ since $F^{\sigma}$ has signature $(2m,0)$ for any non-identity $\sigma \in \mathrm{Gal}(K/k)$.

The admissibility of the form $\mathrm{Res}_{K/k}\,F$ easily follows from the admissibility of $F$. Indeed, if $\eta:k\rightarrow \mathbb{R}$ is a non-identity field embedding of $k$ then it can be extended to $K$, and hence
$$(\mathrm{Res}_{K/k}\,F)^{\eta} = F^{\eta} \oplus F^{\eta \circ \sigma_1}\oplus \dots\oplus F^{\eta \circ \sigma_{d-1}}$$ has signature $(2dm,0)$, since every factor has signature $(2m,0)$.

The conclusion is now straightforward: if $O$ is an order in $D'$, then $O \otimes \mathcal{O}_K$ is an order in $D$ and we see that $\mathrm{U}_G(O)$ is an arithmetic lattice defined over $k$. The stabiliser of the subspace $(D\otimes_K \mathbb{R})^m\subset \mathrm{Res}_{K/k}\,D^m$ is commensurable with $\mathrm{Res}_{K/k}\,\mathbf{U}_F (\mathcal{O}_k)$, and its projection into $\mathbf{U}_F (\mathbb{R})=\mathrm{U}_F(D \otimes \mathbb{R})$ is commensurable with $\mathbf{U}_F (\mathcal{O}_K)$. By applying \cite[Proposition 2.1]{KRS} with $\mathrm{Res}_{K/k}\,\mathbf{U}_F<\mathbf{U}(G)$, we see that any arithmetic lattice $\Lambda<\mathrm{U}_F(D)$ is realised as a totally geodesic sublattice of an arithmetic lattice $\Gamma < \mathrm{U}_G(D')$. The proof of the first part of Proposition \ref{prop:Weil-restriction-type-II} follows by choosing the skew-Hermitian form $G$ to be $\mathrm{Res}_{K/k}\,F$ and by projecting $\Lambda$ and $\Gamma$ to $\mathrm{PU}_F(K)$ and $\mathrm{PU}_G(k)$, respectively.

For the second part of the statement, we notice that if $D'\cong M_2(k)$ then $D\cong D' \otimes K \cong M_2(K)$. From the discussion in Section \ref{sec:unitary-lattices} it follows that in this case both $\Lambda$ and $\Gamma$ are type-I lattices. If $D'$ is a division algebra then $\Gamma$ is a type-II lattice. Thus there are two possible cases:
\begin{enumerate}
    \item $D = D' \otimes \mathbb{K}$ is a division algebra and $\Lambda$ is a type-II lattice; 
    \item $D\cong M_2(K)$ and $\Lambda$ is a type-I lattice.
\end{enumerate}
\end{proof}

\begin{remark}
As in the case of type-I lattices, it is not very restrictive to suppose that the type-II lattice $\Lambda<\mathbf{PU}_F(K)$ is contained in $\mathrm{PU}_F(D)$. Indeed, we have that the finite-index subgroup $\Lambda^{(2)} < \Lambda$ is a subgroup of $\mathrm{PU}_F(D)$ since, by Lemma~\ref{lemma:ERT-type-II}, $g^2 \in \mathrm{PU}_F(D)$ for all $g \in \Lambda$.
\end{remark}
 
\section{Totally geodesic immersions of (quasi-)arithmetic hyperbolic lattices}\label{sec:hereditary properties}
In this section we analyse the relation between the adjoint trace field and the ambient group of a (quasi-)arithmetic lattice and the adjoint trace field and ambient group of a totally geodesic sublattice (Theorem~\ref{theorem:geod}). This will allow us to generalise the examples of totally geodesic immersions constructed in Sections \ref{sec:subform-space-type-I} and \ref{sec:weil-res-type-I}, define the notions of subform subspaces and Weil restriction subspaces of arithmetic hyperbolic lattices, and prove Theorem~\ref{teo:embeddings-hyperbolic-orbifolds}. We begin with proving Theorem~\ref{theorem:cent}, which states that fc-subspaces have finite volume. 

\subsection{Proof of Theorem~\ref{theorem:cent}}\label{sec:5.3}
First, we remark that the fixed point set of any collection of isometries acting on $\HH^n$ is either empty or it is a totally geodesic subspace of $\HH^n$ (possibly a point).
Since by the hypothesis the group $F$ commensurates $\Gamma$, the group $$\Gamma' = \bigcap\limits_{g\in F} g\Gamma g^{-1}$$ is a finite-index subgroup of $\Gamma$. Moreover, it is normalised by the finite group $F$. Hence, the group $\Gamma''$ generated by $\Gamma'$ and $F$ is a lattice in $\Isom(\HH^n)$, commensurable with $\Gamma$, and clearly $F < \Gamma''$.
By \cite[Lemma 4.4]{OV}, the mapping $\phi: U/C_{\Gamma''}(F) \to \HH^n/ \Gamma''$ is proper, where $C_{\Gamma''}(F)$ denotes the centraliser of $F$ in $\Gamma''$ and $U = \Fix(F)$ is a totally geodesic subspace of dimension $m$. Without loss of generality, we can suppose that $F=\{g \in \Gamma'' |\, gx=x\ \forall x \in U\}$.

Since $F$ is finite, the centraliser $C_{\Gamma''}(F)$ has finite index in the normaliser $N_{\Gamma''}(F)$, and the latter is easily seen to be the stabiliser $\mathrm{Stab}_{\Gamma''}(U)$ of $U$ in $\Gamma''$. Hence, also the natural map $\phi': U/\mathrm{Stab}_{\Gamma''}(U) \to \HH^n/ \Gamma''$ is proper. This means that $U$ projects down to a properly immersed totally geodesic suborbifold $S$ in the orbifold $M = \HH^n / \Gamma''$. Since the map $\phi'$ is proper, the cusps of $S$ correspond to cusps of $M$, and there are no accumulation points of $S$ inside of $M$. 

Thus, once $M$ has finite volume, the orbifold $S$ has finite volume as well, except for the excluded case of $m = \mathrm{dim}(U)= 1$, where $S$ could be an infinite geodesic. For one such example, let $\Gamma''$ be a reflection group in the facets of an ideal Coxeter polyhedron $\mathcal{P}$ in $\mathbb{H}^n$ and $F$ be the finite group of reflections in the facets of $\mathcal{P}$ which intersect along an edge $e$ having an ideal vertex. The length of the edge $e$ is clearly infinite.

Finally, by commensurabilty of the lattices $\Gamma$ and $\Gamma''$, we deduce that the stabilisers of $U$ in $\Gamma$ and $\Gamma''$ are commensurable, and therefore also $\mathrm{Stab}_{\Gamma}(U)$ is a lattice acting on $U$. We notice that $\Gamma$ is uniform if and only if it does not contain parabolic elements, and in this case $\mathrm{Stab}_{\Gamma}(U)$ does not contain parabolic elements either. \qed

\subsection{Proof of Theorem~\ref{theorem:geod}} 
We remark that arithmeticity of totally geodesic suborbifolds of arithmetic orbifolds is already known (see \cite[Proposition 15.2.2]{BC-Asterisque}). Here we refine the analysis in order to control the behaviour of the adjoint trace fields.

Let $M = \HH^n/\Gamma$ be an arithmetic hyperbolic $n$-orbifold with adjoint trace field $k$ and ambient group $\G$. 
We are assuming that $M$ contains a proper totally geodesic suborbifold $N=\mathbb{H}^m/\Lambda$ of dimension $\geq 2$ and therefore $M$ cannot be a $3$-dimensional type-III orbifold. As such its adjoint trace field $k$ is totally real and the ambient group $\mathbf{G}$ is admissible (see Section \ref{sec:3-dim-arithmetic} and Corollary \ref{cor:type-III-non-admissible}).
We have that $\Gamma < \G(k)$ and is commensurable with $\G(\OOO_{k})$. Recall that the commensurator of $\Gamma$ in $G = \G(\mathbb{R})\cong \mathbf{PO}_{n,1}(\R)$ is precisely $\G(k)$.

Let $m\geq 2$ and suppose that $N \subset M$  is obtained as the quotient of an $m$-dimensional totally geodesic subspace $U \subset \mathbb{H}^n$. We have that $\mathrm{Stab}_{\Gamma}\,U$ acts as a lattice on $U$.
Let us consider the stabiliser $H=\mathrm{Stab}_{\mathbf{G}(k)}(U)$ of the subspace $U$ in the commensurator of $\Gamma$. We denote by $\mathbf{H}$ its Zariski closure in $\G$, and call this group the {\itshape rational stabilizer} of $N$ in $M$. 
The condition $M \in \mathrm{Stab}_{\mathbf{G}(\mathbb{R})}(U)$ is $\mathbb{R}$-polynomial in the coefficients of $M$, since it corresponds to preserving the $\mathbb{R}$-subspace spanned by $U$. We thus see that $\mathrm{Stab}_{\mathbf{G}(\mathbb{R})}(U)$ is Zariski closed, and therefore it contains $\mathbf{H}(\mathbb{R})$.
By repeating the argument from \cite[Proposition 15.2.2]{BC-Asterisque} verbatim over the field $k$ we obtain that $\mathbf{H}$ is $k$-defined, and is therefore an algebraic $k$-subgroup of $\G$. 

The natural projection map $$\mathrm{Stab}_{\mathbf{G}(\mathbb{R})}(U) \cong \frac{\mathbf{O}_{m,1}(\mathbb{R}) \times \mathbf{O}_{n-m}(\mathbb{R})}{\langle(-\mathrm{id},-\mathrm{id}) \rangle} \rightarrow \mathbf{PO}_{m,1}(\mathbb{R})$$ restricts to a Lie group morphism $\mathbf{H}(\mathbb{R})\rightarrow \mathbf{PO}_{m,1}(\mathbb{R})$. This induces a morphism of real algebraic groups
\begin{equation}\label{eq:Borel-density-projection-map}
\pi:\mathbf{H} \rightarrow \mathbf{PO}_{m,1}
\end{equation} which, by Borel's density theorem \cite{Borel-density}, maps the connected component of the identity $\mathbf{H}^{\circ}$ surjectively onto $\mathbf{PO}_{m,1}^{\circ}$. We denote the kernel of $\pi$ by $\mathbf{C}$. The group $\mathbf{C}(\mathbb{R})$ is a closed subgroup of the compact group 
$$\mathrm{Fix}(U)=\{g\in \G(\mathbb{R})|\,g(x)=x\;\; \forall x \in U \}\cong \mathrm{O}_{n-m},$$ 
and is therefore compact.

We notice that the group $\mathbf{H}$ is not necessarily connected, nor semisimple. In fact it can be shown that its identity component $\mathbf{H}^{\circ}$ is a reductive group. 
Indeed the unipotent radical $\mathrm{R}_u(\mathbf{H}^{\circ})$ of $\mathbf{H}^{\circ}$ projects via the morphism $\pi$ to the unipotent radical of $\mathbf{PO}_{m,1}^{\circ}$. However the latter is trivial, and therefore $\mathrm{R}_u(\mathbf{H}^{\circ})$ is a subgroup of $\mathbf{C}$. Since $\mathbf{C}(\mathbb{R})$ is compact, it contains no unipotent elements, and thus $\mathrm{R}_u(\mathbf{H}^{\circ})$ is trivial.

We therefore have that the commutator subgroup $\mathbf{H}'=[\mathbf{H}^{\circ},\mathbf{H}^{\circ}]$ of $\mathbf{H}^\circ$ is a semisimple $k$-defined subgroup of $\mathbf{H}^{\circ}$, and $\mathbf{H}^{\circ}$ is an almost-direct product of $\mathbf{H}'$ and the identity component $\mathbf{Z}$ of its center \cite[Theorem 2.4]{Pla-Rap}. Since the center of $\mathbf{PO}_{m,1}^{\circ}$ is trivial, we have that $\pi(\mathbf{Z})$ is trivial and therefore the restriction 
\begin{equation}\label{eq:Borel-density-semisimple}\pi:\mathbf{H}'\rightarrow \mathbf{PO}_{m,1}^{\circ}\end{equation} 
of the morphism $\pi$ to $\mathbf{H}'$ is still surjective.

We now observe that the group $\mathbf{H}'$ is admissible. Indeed for any non-identity embedding $\sigma:k\rightarrow \mathbb{R}$ we have that $\mathbf{H}'^{\sigma}(\mathbb{R})$ is realized as a closed subgroup of the compact group $\mathbf{G}^{\sigma}(\mathbb{R})$, and is therefore compact.

Thus, if $\Gamma<\mathbf{G}(k)$ is arithmetic then $\Lambda=\pi(\mathrm{Stab}_{\Gamma}(U))\cap \mathbf{PO}_{m,1}(\mathbb{R})^{\circ}$ is commensurable with the image of $\mathbf{H}'(\mathcal{O}_k)\cap \mathbf{H}'(\mathbb{R})^{\circ} $ under the map $\pi$, and is therefore arithmetic in $\mathbf{PO}_{m,1}(\mathbb{R})$. If $\Gamma<\mathbf{G}(k)$ is only assumed to be quasi-arithmetic we still have that $\Lambda$ is a lattice contained in $\pi(\mathbf{H}'(k)\cap \mathbf{H}'(\mathbb{R})^{\circ})$, and thus is quasi-arithmetic.

We now prove that when $N$ is arithmetic of type I or type II we have that $k \subset K$, where $K$ is the adjoint trace field of $N$. We call the adjoint group $\mathbf{PH}^{\circ}$ of $\mathbf{H}^{\circ}$ the {\itshape semisimple rational stabilizer} of $N$ in $M$, and we
consider its decomposition as a direct product of $\mathbb{R}$-simple adjoint groups \cite[p.\ 46]{Tits66}.
The kernel of the restriction of the map $\pi$ to $\mathbf{H}^{\circ}$ is a normal subgroup, which implies that its projection to a factor $\mathbf{F}$ of the decomposition of $\mathbf{PH}^{\circ}$ is either trivial or all of $\mathbf{F}$. 
It follows that the kernel of $\pi$ projects to the product of a finite set of compact simple factors $\mathbf{C}_1 \times \cdots \times \mathbf{C}_l$. Since the group $\mathbf{PO}_{m,1}^{\circ}$ is $\mathbb{R}$-simple, we see that it is a factor of the decomposition of $\mathbf{PH}^{\circ}$ into $\mathbb{R}$-simple factors. In fact, it is the unique factor whose real points form a noncompact Lie group, and we obtain an isomorphism of algebraic $\mathbb{R}$-groups
\begin{equation}\label{eq:R-simple-decomposition}
    \mathbf{PH}^{\circ} \cong \mathbf{C}_1 \times \cdots \times \mathbf{C}_l \times \mathbf{PO}_{m,1}^{\circ},
\end{equation} with the map $\pi$ inducing the projection onto the last factor.

Now we consider the decomposition of the connected, adjoint group $\mathbf{PH}^{\circ}$  as a direct product of $k$-simple adjoint groups. There is precisely one factor $\mathbf{H}_0$ whose group of $\mathbb{R}$-points is noncompact and contains the $\mathbf{PO}_{m,1}^{\circ}$-factor as a normal subgroup. Moreover, by \cite[\S 3.1.2.]{Tits66} there exists a finite field extension $K'\supset k$ and an absolutely simple $K'$-group $\mathbf{D}$ such that $\mathbf{H}_0 = \mathrm{Res}_{K'/k}\, \mathbf{D}$. Since $\mathbf{H}_0$ is admissible for $\mathrm{PO}_{m,1}^{\circ}$ also $\mathbf{D}$ is, and moreover we see that the lattice $\Lambda$ is naturally identified with a subgroup of $\mathbf{D}({K'})$.

If the dimension $m$ of $N$ is different from $3$ we have that $\mathbf{PO}_{m,1}^{\circ}$ is absolutely simple. Then we see that $\mathbf{D}$ is a $K'$-form of $\mathbf{PO}_{m,1}^{\circ}$ and by \cite[Lemma 2.6]{PR} or Proposition~\ref{prop:adj-trace-of-QA} (which applies to the quasi-arithmetic case as well) we have that the adjoint trace field $K$ of $\Lambda$ equals $K'$ and that $\mathbf{D}$ is $K$-isomorphic to the identity component $\mathbf{L}^{\circ}$ of the real ambient group $\mathbf{L}$ of $\Lambda$. This implies that $K\supset k$.

If $m=3$ we see that $\mathbf{D}$ is a $K'$-form of $\mathbf{PGL}_2$, and again by \cite[Lemma 2.6]{PR} or Proposition~\ref{prop:adj-trace-of-QA} (in the quasi-arithmetic case) we have that the invariant trace field $L$ of $\Lambda$ equals $K'$ and $\mathbf{D}$ is $L$-isomorphic to the complex ambient group $\mathbf{PGL}_A$ of $\Lambda$. This implies that $L\supset k$. If $N=\mathbb{H}^3/\Lambda$ is not a type-III orbifold then by Corollary \ref{cor:subfield} we have that the invariant trace field $L$ is an imaginary quadratic extension of the totally real field $K=L\cap \mathbb{R}$, and moreover $K$ is precisely the adjoint trace field of $\Lambda$. Since $k$ is a real field and $k\subset L$ we immediately obtain that $k\subset K$. \qed

\medskip

We now record two by-products of the proof of Theorem \ref{theorem:geod}. We will use the second one in the proof of Theorem \ref{teo:embeddings-hyperbolic-orbifolds}.

\begin{remark}\label{prop:invariant-trace-field-inclusion}
Let $N=\mathbb{H}^3/\Lambda$ be a totally geodesic suborbifold of an $n$-dimensional arithmetic orbifold $M=\mathbb{H}^n/\Gamma$. The adjoint trace field $k$ of $M$ is contained in the invariant trace field $L$ of $N$, and the latter (which is defined up to complex conjugation) equals the minimal field of definition of the natural projection map from the identity component $\mathbf{H}^{\circ}$ of the rational stabilizer into one of the two absolutely simple factors of the semisimple rational stabilizer $\mathbf{PH}^{\circ}$ that are isomorphic to $\mathbf{PGL}_2$. Moreover, the complex ambient group $\mathbf{PGL}_A$ of $N$ is naturally identified with this $\mathbf{PGL}_2$-factor, and the group $\mathrm{Res}_{L/k}\,\mathbf{PGL}_A$ is $k$-isomorphic to a $k$-simple factor of $\mathbf{PH}^{\circ}$.  
\end{remark}

\begin{remark}\label{lemma:atf-and-projection-field-coincide}
Let $K$ be the adjoint trace field of the totally geodesic suborbifold $N \subset M$ in Theorem \ref{theorem:geod} and let $\mathbf{L}$ be its real ambient group. Then $K$ equals the minimal field of definition of the projection $\pi:\mathbf{H}\rightarrow \mathbf{PO}_{m,1}^{\circ}$, the group $\mathbf{L}^{\circ}$ is naturally identified with $\mathbf{PO}_{m,1}^{\circ}$, and the group $\mathrm{Res}_{K/k}\,\mathbf{L}^{\circ}$ is $k$-isomorphic to a $k$-simple factor of the semisimple group $\mathbf{PH}^{\circ}$. This is obvious if $m \neq 3$ because $\mathbf{PSO}_{m,1}$ is absolutely simple. If $m=3$, let us denote by $L$ the invariant trace field of $N$ and by $A$ the invariant quaternion algebra. Since $N$ is type-I or II, we have the sequence of field inclusions $k\subset K \subset L$, where $K=L\cap \mathbb{R}$ is an index two subfield of $L$ (see Section~\ref{sec:3-dim-types}). Remark~\ref{prop:invariant-trace-field-inclusion} states that $\mathrm{Res}_{L/k}\,\mathbf{PGL}_A$ is $k$-isomorphic to a $k$-simple factor of $\mathbf{PH}^{\circ}$. By Corollary \ref{cor:subfield} we have that $\mathrm{Res}_{L/K}\mathbf{PGL}_A$ is $K$-isomorphic to $\mathbf{L}^{\circ}$, and from this we obtain a $k$-isomorphism between $\mathrm{Res}_{K/k}\,\mathbf{L}^{\circ}$ and $\mathrm{Res}_{L/k}\,\mathbf{PGL}_A$.

\end{remark}

\begin{remark}\label{rem:GPS-smaller-trace-field}
The statement of Theorem \ref{theorem:geod} about the inclusion of adjoint trace fields cannot hold for totally geodesic immersions of pseudo-arithmetic hyperbolic orbifolds. Rather, there are examples where the opposite inclusion holds. As the work of Emery and Mila \cite{EM20} shows, if $M$ is a Gromov--Piatetski-Shapiro non-arithmetic manifold, then its adjoint trace field $k$ is a multiquadratic extension of the adjoint trace field $K$ of its building blocks. Thus any connected component of the gluing locus of the blocks is a totally geodesic submanifold of $M$ whose adjoint trace field $K$ is a proper subfield of the adjoint trace field $k$ of $M$. By combining this fact with Theorem \ref{theorem:geod}, we obtain a simple proof of non-quasi-arithmeticity of the Gromov--Piatetski-Shapiro manifolds.
\end{remark}

Theorem \ref{theorem:geod} naturally suggests that the case where $k=K$, i.e.\ the adjoint trace field of the totally geodesic sublattice coincides with the adjoint trace field of the ambient lattice, is special. Motivated by this, we give the following definition:

\begin{definition}\label{def:subform-space}
Let $N=\mathbb{H}^m/\Lambda$, $m\geq 2$, be a totally geodesic subspace of a quasi-arithmetic orbifold $M=\mathbb{H}^n/\Gamma$. We say that $N$ is a {\itshape subform subspace} if the adjoint trace field $K$ of $\Lambda$ coincides with the adjoint trace field $k$ of $\Gamma$. In this setting, we say that $\Lambda$ is a {\itshape subform sublattice} of $\Gamma$.
\end{definition}

\begin{remark}
Our definition of a subform subspace naturally extends the definition given by Meyer \cite[Construction 4.11]{Mey17} for subspaces of type-I orbifolds to subspaces of arithmetic orbifolds of any type.
\end{remark}
We have already encountered many examples of subform subspaces in arithmetic hyperbolic orbifolds: indeed all the involutions in $\mathrm{PO}_f(k)$ described in Section~\ref{sec:k-involutions-type1} and the involutions in $\mathrm{PU}_F(D)$ described in Section~\ref{sec:k-involutions-type2} give rise to subform sublattices: the fixed point set $\mathbf{U}$ for the action of the involution on $\mathbb{R}^{n+1}$ corresponds to a $k$-subspace (resp.\ a $D$-submodule) and the restriction of the quadratic form $f$ (resp.\ the skew-Hermitian form $F$) to $\mathbf{U}$ will be admissible and defined over $k$ (resp.\ defined on $D$, where $D$ is a quaternion algebra over $k$).

We begin characterising subform subspaces by proving the following proposition:
\begin{prop}\label{prop:subform-space-is-fc}
Let $N=\mathbb{H}^m/\Lambda$ be a subform subspace of an arithmetic orbifold $M=\mathbb{H}^n/\Gamma$. Then $N$ is an fc-subspace associated to a single involution in the commensurator of $\Gamma$.\end{prop}

\begin{proof}
Denote by $\G$ and $\mathbf{L}$ the ambient groups of $M$ and $N$ respectively, which are both defined over the same adjoint trace field $k$. Notice that $\G$ (resp.\ $\mathbf{L}$) is a $k$-form of $\mathbf{PO}_{n,1}$ (resp.\ $\mathbf{PO}_{m,1}$). 
Here we opt to work with forms of the real group $\mathbf{O}_{n,1}(\mathbb{R})$. Up to $k$-isomorphism there is a unique algebraic $k$-group $\widetilde{\mathbf{G}}$ in the $k$-isogeny class of ${\mathbf{G}}$ whose real points are isomorphic to $\mathbf{O}_{n,1}(\mathbb{R})$ \cite[\S 2.6]{Tits66}. The group $\widetilde{\G}$ is obtained as a central extension of $\G$ by $\mathbb{Z}/2\mathbb{Z}$, and the nontrivial element of the center corresponds to $\mathrm{-id}\in \mathrm{O}_{n,1}$. If $\mathbf{G}$ is of absolute type $B_n$, then $\widetilde{\mathbf{G}}=\mathbf{G} \times \mathbb{Z}/2\mathbb{Z}$ is simply a direct product. On the other hand, if $\mathbf{G}$ is of absolute type $D_n$, the extension is induced by a non-trivial $k$-isogeny of $\widetilde{\mathbf{G}}^{\circ}$ onto $\mathbf{G}^{\circ}$. The same conclusions hold for the $k$-from $\widetilde{\mathbf{L}}$ of $\mathbf{O}_{m,1}$.

We repeat the same argument as in the proof of Theorem \ref{theorem:geod} using the forms of $\mathbf{O}_{n,1}$.  Suppose that the totally geodesic suborbifold $N$ is the projection of a totally geodesic subspace $U\subset \mathbb{H}^m$, and denote by $\mathcal{U}$ the vector subspace of $\mathbb{R}^{n+1}$ such that $\mathcal{U} \cap \mathbb{H}^n=U$. We define $\widetilde{\mathbf{H}}$ as the preimage in $\widetilde{\mathbf{G}}$ of the rational stabilizer $\mathbf{H}<\mathbf{G}$ and notice that its real points form a subgroup of $\mathrm{Stab}_{\widetilde{\mathbf{G}}(\mathbb{R})}(\mathcal{U})= \mathrm{O}_{m,1} \times \mathrm{O}_{n-m}$.

By Remark \ref{lemma:atf-and-projection-field-coincide} and the fact that $K=k$, the group $\mathbf{L}^\circ$ is $k$-isomorphic to a $k$-simple factor of the semisimple rational stabilizer $\mathbf{PH}^{\circ}$. Hence the projection $\pi:\mathbf{H}\rightarrow \mathbf{L}$ lifts to a surjective morphism $\widetilde{\pi}:\widetilde{\mathbf{H}}\rightarrow \widetilde{\mathbf{L}}$.
In particular, we obtain that $\widetilde{\mathbf{H}}$ decomposes over $k$ as a product 
$\widetilde{\mathbf{H}} = \widetilde{\mathbf{L}} \times \mathbf{C}$ where $\mathbf{C}$ is the kernel of $\widetilde{\pi}$. Moreover, we notice that the image of $(-\mathrm{id},-\mathrm{id}) \in \widetilde{\mathbf{H}}(k)$ under $\widetilde{\pi}$ is $-\mathrm{id} \in \widetilde{\mathbf{L}}(k)$.

We conclude by noticing that the element $\theta=(\mathrm{id},-\mathrm{id}) \in \mathrm{O}_{m,1} \times \mathrm{O}_{n-m}$ belongs to $\widetilde{\mathbf{H}}(k) < \widetilde{\mathbf{G}}(k)$, since it is the product of $(-\mathrm{id},-\mathrm{id})\in \widetilde{\mathbf{H}}(k)$ with $(-\mathrm{id},\mathrm{id})\in \widetilde{\mathbf{L}}(k)< \widetilde{\mathbf{H}}(k)$. Its projection to $\mathbf{G}$ lies in $\mathbf{G}(k)$, and therefore corresponds to an order-$2$ element in the commensurator of the lattice $\Gamma$. Moreover, $\theta$ corresponds to the identity element in $\mathbf{L}(\mathbb{R})=\mathbf{PO}_{m,1}(\mathbb{R})$, meaning that it acts on the subspace $U$ by fixing it pointwise.
\end{proof}

We now combine Proposition \ref{prop:subform-space-is-fc} with the characterization of involutions in the ambient groups of type I and II arithmetic lattices given in Sections \ref{sec:k-involutions-type1} and \ref{sec:k-involutions-type2} to prove the following:

\begin{prop}
Let $M$ be a type-I (resp. type-II) arithmetic hyperbolic orbifold, and $N \subset M$ a subform subspace in $M$ of dimension $\geq 2$. Then $N$ is a type-I (resp. type-II) arithmetic hyperbolic orbifold.
\end{prop}

\begin{proof}
We first deal with the case where $M$ is a type-I arithmetic orbifold. Let $k$ be its adjoint trace field and $\mathbf{G}=\mathbf{PO}_f$ its ambient group, where $f$ is an admissible form defined over $k$. By Proposition \ref{prop:subform-space-is-fc}, the subspace $N$ is the projection of the fixed point set $U=\mathcal{U} \cap \mathbb{H}^n$ of an involution $\theta \in \mathbf{PO}_f (k)$.

By the discussion in Section \ref{sec:k-involutions-type1}, there are two possibilities:
\begin{enumerate}
    \item the involution $\theta$ belongs to the image $\mathrm{PO}_f(k)$ of $\mathrm{O}_f(k)$ under the projection map $\mathbf{O}_f\rightarrow \mathbf{PO}_f$;
    \item the involution $\theta$ belongs to $\mathbf{PO}_f (k) \setminus \mathrm{PO}_f(k)$.
\end{enumerate}
Case (2) is not possible: by the discussion in Section \ref{sec:k-involutions-type1}, the adjoint trace field of $N$ would be a totally real quadratic extension of $k$ of the form $K=k(\sqrt{\mu})$, and thus $N$ would not be a subform subspace.

Therefore, $\theta$ belongs to $\mathrm{PO}_f(k)$, the fixed point set $\mathcal{U}$ of $\theta$ is a $k$-subspace, and the adjoint trace field of $N$ is precisely $k$. The ambient group of the fc-subspace associated to $N$ is $\mathbf{PO}_g$, where $g$ is the ($k$-defined and admissible) restriction of $f$ to $\mathcal{U}$.

The proof works similarly in the case where $M$ is a type-II orbifold with adjoint trace field $k$ and ambient group of the form $\mathbf{PU}_F$. The subspace $N$ is the fixed point set of an involution $\theta \in \mathbf{PU}_F (k)$. By the discussion in Section~\ref{sec:k-involutions-type2},  the involution $\theta$ cannot belong to $\mathbf{PU}_F (k) \setminus \mathrm{PU}_F(D)$, or the adjoint trace field of $N$ would again be a totally real quadratic extension $K=k(\sqrt{\mu})$ of $k$. Thus $\theta$ belongs to $\mathrm{PU}_F(D)$, its fixed-point-set is a right $D$-submodule $D_N$ of $D^m$, the adjoint trace field of $N$ is simply $k$ and its ambient group is $\mathbf{PU}_G$, where $G$ is the (admissible) restriction of $F$ to $D_N$.
\end{proof}

We finally have the tools to prove Theorem \ref{teo:embeddings-hyperbolic-orbifolds}, which is meant to be an algebraic characterisation of totally geodesic immersions of arithmetic hyperbolic orbifolds. Here we make an essential use of Theorem~\ref{theorem:cent} on the finiteness of volume of fc-subspaces.

\subsection{Proof of Theorem~\ref{teo:embeddings-hyperbolic-orbifolds}}
The last part of the statement is obvious: if there were a proper subform subspace $S'$ of $S$ which contains $N$, then the adjoint trace field of $S'$ would be the same as that of $M$ and thus $S'$ would be a subform subspace of $M$, contradicting the minimality of $S$.

The subform subspace $S$ can be constructed explicitly. In what follows, we carry over definitions and notations from the proof of Theorem \ref{theorem:geod}. In particular $k$ (resp.\ $K$) denotes the adjoint trace field of $M=\mathbb{H}^n/\Gamma$ (resp.\ $N=\mathbb{H}^m/\Lambda$), $\G$ (resp.\ $\mathbf{L}$) denotes its ambient group and $\mathbf{H}<\mathbf{G}$ denotes the rational stabiliser of $N$ in $M$. 

Let us substitute as in the proof of Proposition \ref{prop:subform-space-is-fc} the groups $\mathbf{L}$ and $\mathbf{G}$ with the corresponding forms $\widetilde{\mathbf{L}}$ and $\widetilde{\mathbf{G}}$ of $\mathbf{O}_{m,1}$ and $\mathbf{O}_{n,1}$ respectively, and take $\widetilde{\mathbf{H}}$ to be the preimage of $\mathbf{H}$ in $\widetilde{\mathbf{G}}$. Notice how $\widetilde{\mathbf{G}}(\mathbb{R})\cong \mathrm{O}_{n,1}$ admits a natural linear representation on the vector space $\mathbb{R}^{n+1}$ as the orthogonal group of the standard form of signature $(n,1)$, and this restricts to a representation of the group $\widetilde{\mathbf{H}}(\mathbb{R})$. Since this latter group is reductive, this representation decomposes as a direct sum of irreducible representations.

By Remark \ref{lemma:atf-and-projection-field-coincide} the group $\mathrm{Res}_{K/k}\,\mathbf{L}^\circ$ is $k$-isomorphic to a $k$-simple factor of the semisimple rational stabilizer $\mathbf{PH}^{\circ}$. By the same argument used in the proof of Proposition \ref{prop:subform-space-is-fc} we see that the group $\mathrm{Res}_{K/k}\,\widetilde{\mathbf{L}}$ is a $k$-simple factor of $\widetilde{\mathbf{H}}$, i.e. we have the following $k$-defined decompositions:
\begin{align}\label{eq:Bergeron-Clozel-contains-res-of-scalars-2}
\widetilde{\mathbf{H}}\cong \mathbf{C} \times \mathrm{Res}_{K/k}\,\widetilde{\mathbf{L}},\\
\mathbf{H} \cong \frac{\mathbf{C} \times \mathrm{Res}_{K/k}\,\widetilde{\mathbf{L}}}{\langle(-\mathrm{id},-\mathrm{id})\rangle}\end{align} 
where $\mathbf{C}$ is the (possibly trivial) kernel of the natural projection map \linebreak
\mbox{$\widetilde{\pi}:\widetilde{\mathbf{H}}\rightarrow \mathrm{Res}_{K/k}\,\widetilde{\mathbf{L}}$}. These decompositions completely characterise the arithmetic structure of the geodesic immersion of $N$ into $M$ and determine the irreducible factors of the representation of $\widetilde{\mathbf{H}}(\mathbb{R})$ on $\mathbb{R}^{n+1}$.

Indeed we see that
$\mathrm{Res}_{K/k}\,\widetilde{\mathbf{L}}(\mathbb{R})^{\circ} \cong \mathrm{SO}_{m,1} \times \prod_{i=1}^{d-1} \mathrm{SO}_{m+1}$, where $d$ is the degree of $K/k$. Each factor corresponds to a field embedding $\sigma\in S^{\infty}_{K/k}$ and acts irreducibly on an $(m+1)$-dimensional subspace $\mathcal{U}^{\sigma}$ of $\mathbb{R}^{n+1}$ (equipped with the standard Minkowski form of signature $(n,1)$) while acting trivially on its orthogonal complement.

Each compact factor acts on a positive definite subspace $\mathcal{U}^{\sigma}$, while the single non-compact factor acts precisely on the lift $\mathcal{U}^{\sigma_0}=\mathcal{U}$ of the subspace $N$ to $\mathbb{R}^{n+1}$.  The compact group $\mathbf{C}(\mathbb{R})$ acts on some (possibly trivial) positive definite subspace of dimension $h\geq 0$ and trivially on its orthogonal complement, and we get that $\mathbb{R}^{n+1}$ decomposes as an orthogonal direct sum of all these subspaces, i.e. $n+1=d\cdot(m+1)+h$.

If the group $\mathbf{C}$ is trivial, we simply set $S=M$. If $\mathbf{C}$ is non-trivial, we notice that $i=[(-\mathrm{id},\mathrm{id})]$ is an involution in $\mathbf{H}(k)<\G(k)$, and thus belongs to the commensurator of $\Gamma$. We take $S$ to be the fc-subspace associated to this involution. Notice that $S$ has finite volume due to Theorem \ref{theorem:cent}.

In the first case, $S=M$ is trivially a subform subspace of $M$. In the second case, we can repeat the construction and define the corresponding groups for the immersion $S \subset M$. More specifically, we denote by 
$$
\mathcal{U}_S=\mathcal{U} \oplus \mathcal{U}^{\sigma_1} \oplus \dots \oplus \mathcal{U}^{\sigma_{d-1}}
$$
the subspace that is pointwise fixed by the involution $(-\mathrm{id},\mathrm{id}) \in \widetilde{\mathbf{H}}(k)$, by $\mathbf{H}_S\subset \G$ the rational stabiliser $\mathbf{H}_{S}$ of $S$ and by $\mathbf{L}_S$ the ambient group of $S$.

We obtain that $\mathbf{C}$ is contained in the kernel $\mathbf{C}_S$ of the projection map $$\widetilde{\pi}_{S}:\widetilde{\mathbf{H}}_{S}\rightarrow  \widetilde{\mathbf{L}}_S.$$
Indeed, $\mathbf{C}(\mathbb{R})$ acts trivially on the fixed point set $\mathcal{U}_S$ of the involution $i$, which is the direct sum of all subspaces of $\mathbb{R}^{n+1}$ associated to the factors of $\mathrm{Res}_{K/k}\,\widetilde{\mathbf{L}}(\mathbb{R})$. Since each of these  $d$ subspaces has dimension $m+1$, we obtain that $\mathrm{dim}(\mathcal{U}_S)=d(m+1)$ and thus $\mathrm{dim}(S)=d(m+1)-1$.

Moreover, it is not difficult to see that $\mathbf{C}=\mathbf{C}_S$: any element $g\in\mathbf{C}_S(\mathbb{R})$ fixes $\mathcal{U}_S$ pointwise and therefore has to fix pointwise the lift $\mathcal{U}\subset \mathcal{U}_S$ of $N$ as well. Therefore, the element $g$ belongs to the group $\widetilde{\mathbf{H}}(\mathbb{R})$. Since $g$ acts trivially on $\mathcal{U}_S=\mathcal{U} \oplus \mathcal{U}^{\sigma_1} \oplus \dots \oplus \mathcal{U}^{\sigma_{d-1}}$, it projects to the identity element in $\mathrm{Res}_{K/k}\,\widetilde{\mathbf{L}}(\mathbb{R})$ and thus it belongs to $\mathbf{C}(\mathbb{R})$.

We therefore obtain that $\mathbf{C}_S$ is a $k$-subgroup of $\widetilde{\mathbf{H}}_S$, meaning that the decomposition

$$\widetilde{\mathbf{H}}_S \cong \mathbf{C}_S \times \widetilde{\mathbf{L}}_S $$ is defined over $k$. By applying Remark \ref{lemma:atf-and-projection-field-coincide} we obtain that the adjoint trace field of $S$ is equal to $k$ and therefore $S$ is a subform subspace.

The minimality of $S$ as a subform subspace containing $N$ can be proven as follows: by Remark \ref{lemma:atf-and-projection-field-coincide}  the real points of the group $\widetilde{\mathbf{L}}_{S'}$ associated to a subform subspace $S'$ containing $N$ must contain the group $\mathrm{Res}_{K/k}\,\widetilde{\mathbf{L}}(\mathbb{R})^{\circ} \cong \mathrm{SO}_{m,1} \times \prod_{i=1}^{d-1} \mathrm{SO}_{m+1}$ as a $k$-subgroup, and thus a lift $\mathcal{U}_{S'}$ of $S'$ has to contain $\mathcal{U} \oplus \mathcal{U}^{\sigma_1} \oplus \dots \oplus \mathcal{U}^{\sigma_{d-1}}=\mathcal{U}_S$. Consequently, we see that $S\subseteq S'$. \qed

\begin{remark}\label{rem:geometric-characterization-subform}
The property of being a subform subspace can be characterized geometrically. If $N=\mathbb{H}^m/\Lambda$ is a geodesic submanifold in $M=\mathbb{H}^m/\Gamma$ and $U$ denotes a lift of the universal cover of $N$ to $\mathbb{H}^n$, then the stabilizer $\Sigma=\mathrm{Stab}_{\Gamma}(U)$ of $U$ in $\Gamma$ is naturally a lattice inside the stabilizer $$\mathrm{Stab}_G(U)\cong\frac{\mathrm{O}_{m,1}\times \mathrm{O}_{n-m}}{\langle (-\mathrm{id},-\mathrm{id})\rangle}$$ of $U$ in $G=\mathrm{Isom}(\mathbb{H}^n)$. Denoting by $\widetilde{\Sigma}$ the preimage of $\mathrm{Stab}_{\Gamma}(U)$ in $\mathrm{O}_{m,1}\times \mathrm{O}_{n-m}$, we see that $N$ is a subform subspace of $M$ if and only if $\widetilde{\Sigma}$ is a reducible lattice in $\mathrm{O}_{m,1}\times \mathrm{O}_{n-m}$, in the sense that it is commensurable with a product $\Lambda_1 \times \Lambda_2$ where $\Lambda_1<\mathrm{O}_{n,1}$ and $\Lambda_2<\mathrm{O}_{n-m}$ are lattices, and $\Lambda_1 \sim \Lambda$.

Indeed, if $N$ is a subform subspace in $M$, then the decomposition in equation \eqref{eq:Bergeron-Clozel-contains-res-of-scalars-2} takes the form $\widetilde{\mathbf{H}}\cong \mathbf{C} \times \widetilde{\mathbf{L}}$ with $\widetilde{\mathbf{L}}(\mathbb{R})= \mathrm{O}_{m,1}$ and $\mathbf{C}(\mathbb{R})<\mathrm{O}_{n-m}$ and we may take $\Lambda_1=\widetilde{\Sigma} \cap \widetilde{\mathbf{L}}$ and $\Lambda_2=\widetilde{\Sigma} \cap \mathbf{C}$.

On the other hand, if the trace field inclusion $k\subset K$ is proper then the projection of $\widetilde{\Sigma}$ to the $\mathrm{O}_{n-m}$-factor is not finite, but rather a dense subgroup of the group $\prod_{i=1}^{d-1} \mathrm{SO}_{m+1}<\mathrm{O}_{n-m}$ corresponding to the product of the compact factors of $\mathrm{Res}_{K/k}\,\widetilde{\mathbf{L}}(\mathbb{R})$. Moreover, in this case the intersection $\widetilde{\Sigma} \cap \widetilde{\mathbf{L}}$ is trivial.

\end{remark}

\begin{remark}\label{rem:1-dimensional-not-in-3-dim-type-III}
We can define a notion of adjoint trace field also for a one-dimensional totally geodesic subspace (i.e.\ a closed, immersed geodesic)  of a hyperbolic manifold $M$. As a word of warning, the orientation-preserving isometry group of $\mathbb{H}^1 \cong \mathbb{R}$ is isomorphic to $\mathbf{SO}_{1,1}(\mathbb{R})^{\circ}=\mathrm{SO}_{1,1}^+\cong \mathbb{R}$, and $\mathbf{SO}_{1,1}\cong \mathbb{G}_m$ is a one-dimensional $\mathbb{R}$-split torus (thus an abelian group). The adjoint action is therefore trivial, so one has to define the adjoint trace field differently. However, we may embed $\mathbf{SO}_{1,1}$ as a maximal $\mathbb{R}$-split torus inside the simple $\mathbb{R}$-defined group $\mathbf{SO}_{2,1}\cong \mathbf{PGL}_2$, and compute the adjoint traces of an element of $\mathbf{SO}_{1,1}(\mathbb{R})$ via its action on the Lie algebra $\mathfrak{sl}_2$ of $\mathbf{PGL}_2$. An easy computation shows that for a hyperbolic translation $\gamma$ of length $x$, this adjoint trace is equal to $$\mathrm{tr}(\gamma)+1=2\cosh(x)+1=e^x + e^{-x}+1,$$ where $\mathrm{tr}(\gamma)$ is the trace of $\gamma$ as an element of $\mathrm{SO}_{1,1}^+$. The adjoint trace field of the one-dimensional lattice $\langle \gamma \rangle \cong \mathbb{Z}<\mathrm{SO}_{1,1}^+$ is then defined as $\mathbb{Q}(\mathrm{tr}(\gamma))$ and it is a commensurability invariant. 

It is worth noting that if $\gamma$ is realized as a purely hyperbolic element inside an arithmetic lattice $\Gamma<\mathrm{PSL}_2(\mathbb{F})$ where $\mathbb{F}=\mathbb{R} \text{ or }\mathbb{C}$, then by \cite[Remark 7]{Vin95} the adjoint trace field of $\gamma$ equals the adjoint trace field $k$ of $\Gamma$. Moreover, we have that $\Gamma$ is necessarily of type I if $\mathbb{F}=\mathbb{R}$ and of type I or II if $\mathbb{F}=\mathbb{C}$ and thus $k$ is necessarily totally real. The group $\langle \gamma\rangle$ is then contained in the group $\mathbf{G}(k)$ of $k$-points of the real ambient group $\mathbf{G}$ of $\Gamma$, and its Zariski closure is a maximal $k$-torus $\mathbf{L}<\mathbf{G}$ that splits over $\mathbb{R}$ and which takes the role of the ambient group of $\langle \gamma \rangle$. Besides, since $\mathbf{G}$ is admissible we have that $\mathbf{L}^{\sigma}(\mathbb{R})\cong \mathrm{SO}_2/\langle \pm \mathrm{id} \rangle$ is compact for every non-trivial embedding $\sigma:k\rightarrow \mathbb{R}$, meaning that $\mathbf{L}$ is an admissible $k$-form of $\mathrm{SO}_{1,1}^+$ and $\langle\gamma\rangle$ is commensurable with the group of integer points $\mathbf{L}(\mathcal{O}_k)$. In essence, we may regard $\langle \gamma \rangle$ as a one-dimensional arithmetic hyperbolic lattice.

With this in mind, the statement about the inclusion of the adjoint trace fields in Theorem~\ref{theorem:geod} also holds when the totally geodesic subspace $N$ is one-dimensional, provided that $M$ is not a $3$-dimensional type-III orbifold.

Indeed, to prove the inclusion $k\subseteq K$ we only use the surjectivity of the projection map $\pi:\mathbf{H}(\mathbb{R})^{\circ} \rightarrow \mathbf{L}(\mathbb{R})^{\circ}$ and that the ambient group $\mathbf{G}$ of $M$ is admissible (which implies that its $k$-defined subgroup $\mathbf{H}$ is admissible).
The first condition is always guaranteed by the fact that lattices in $\mathrm{SO}_{1,1}^+\cong \mathbb{R}$ (which is not semi-simple) are Zariski dense. The second condition follows from the fact that $M$ is not $3$-dimensional type III.

Consequently, the statement of Theorem~\ref{teo:embeddings-hyperbolic-orbifolds} always applies when $N$ is one-dimensional and $M$ is not a $3$-dimensional type-III orbifold. In particular, we can speak of subform subspaces of dimension one in type-I, II and $7$-dimensional type-III lattices, and Proposition~\ref{prop:subform-space-is-fc} applies in this case too.
\end{remark}

The following two remarks concern the possible obstructions to extending Theorems \ref{teo:embeddings-hyperbolic-orbifolds} and \ref{theorem:geod} to the settings of quasi- or pseudo-arithmetic lattices, respectively.

\begin{remark}
We are currently unable to prove Theorem \ref{teo:embeddings-hyperbolic-orbifolds} under the assumption that the lattice $\Gamma$ is properly quasi-arithmetic. We do obtain the decomposition given in \eqref{eq:Bergeron-Clozel-contains-res-of-scalars-2} for the $k$-group $\mathbf{H}$. However, we cannot directly apply Theorem~\ref{theorem:cent} to conclude that the minimal subform subspace $S$ has finite volume. The main obstruction here is that the involution $i=[(-\mathrm{id},\mathrm{id})]$ might not belong to the commensurator of $\Gamma$, which has infinite index in $\mathbf{G}(k)$. Thus, $S$ might simply be the (infinite-volume) quotient of $\mathcal{U}_S=\mathcal{U} \oplus \mathcal{U}^{\sigma_1} \oplus \dots \oplus \mathcal{U}^{\sigma_{d-1}}$ under the discrete group $\mathrm{Stab}_{\Gamma}(\mathcal{U}_S)$.
\end{remark}

\begin{remark}
We cannot prove that if $M=\mathbb{H}^n/\Gamma$ is pseudo-arithmetic and $N\subset M$ is a totally geodesic subspace, then $N$ has to be pseudo-arithmetic as well. The main problem here is proving that the group $\mathbf{H}$, defined as the Zariski closure of  $\mathrm{Stab}({\mathcal{U}})\cap \mathrm{Comm}(\Gamma)$, is defined over the field of definition $k$ of the ambient group $\mathbf{G}$ of $M$. Notice that the adjoint trace field $k'$ of $M$ is now a multiquadratic extension of $k$ and $\Gamma<\mathbf{G}(k')$. It is indeed true that the Zariski closure of $\mathrm{Comm}(\Gamma)$ is $k$-defined, since it is equal to $\mathbf{G}$ by Borel's density theorem. It could however be possible that the Zariski closure of $\mathrm{Stab}(\mathcal{U})$ is only defined over $k'$.

We do remark however that by the angle rigidity theorem \cite[Theorem 4.1]{FLMS18} and \cite{FLMS-erratum} it follows that if $M$ is an $n$-dimensional Gromov--Piatetski-Shapiro manifold then all of its totally geodesic submanifolds of dimension $> (n-1)/2$ are pseudo-arithmetic. In this setting, $\mathbf{G}=\mathbf{O}_f(k')$, where $f$ is admissible over $k$ and  $\mathcal{U}$ is a $k$-subspace. It follows that both the Zariski closure of $\mathrm{Stab}_{\Gamma}\mathcal{U}$ and the group $\mathbf{H}$ are $k$-defined. Explicit examples of totally geodesic embeddings of Gromov--Piatetski-Shapiro manifolds are described in \cite{KRiS}.
\end{remark}

By Theorem \ref{teo:embeddings-hyperbolic-orbifolds}, we see that $N$ is a subform subspace of $M$ precisely when $N=S$. At this point it becomes natural to consider the other possible extremal case, when $S=M$.

\begin{definition}
Let $N = \mathbb{H}^m/\Lambda$ be a totally geodesic subspace of an arithmetic orbifold $M=\mathbb{H}^n/\Gamma$. Suppose that $N$ is not $3$-dimensional type III. We say that $N$ is a {\itshape Weil restriction subspace} of $M$ if the minimal subform subspace $S$ of $M$ which contains $N$ is precisely $M$. In this setting, we say that $\Lambda$ is a {\itshape Weil restriction}  sublattice of $\Gamma$.
\end{definition}

This definition also applies to one-dimensional subspaces of type-I, -II and $7$-dimensional type-III lattices by Remark \ref{rem:1-dimensional-not-in-3-dim-type-III}.
Notice how all the totally geodesic immersions built using Propositions \ref{prop:res_of_scalars_typeI} and \ref{prop:Weil-restriction-type-II} give rise to Weil restriction sublattices. Moreover, we have an explicit description of the ambient group of the lattice into which we construct the embedding: these are given by groups of the form $\mathbf{O}_{\mathrm{Res}_{K/k}(f)}$ (in the type-I case) or $\mathbf{U}_{\mathrm{Res}_{K/k}(F)}$ (in the type-II case). Theorems \ref{theorem:geod} and \ref{teo:embeddings-hyperbolic-orbifolds} simply state that the two techniques introduced in Section \ref{sec:kinds-of-subspaces} are all that is needed to construct all totally geodesic immersions of arithmetic hyperbolic lattices, except when one of $M$ or $N$ is $3$-dimensional type III. More precisely, if $M=\mathbb{H}^n/\Gamma$  is an arithmetic hyperbolic orbifold not belonging to the exceptional family in dimension $n=3$ and $N=\mathbb{H}^m/\Lambda$ is a totally geodesic subspace of $M$, then:

\begin{enumerate}
    \item $N$ is arithmetic, its adjoint trace field $K$ is contains the adjoint trace field $k$ of $M$ and $[K:k] = d \geq 1$;
    \item If $N$ is not $3$-dimensional type-III the field $K$ is totally real, the real ambient group $\mathbf{L}$ of $N$ is an admissible $K$-form of $\mathbf{PO}_{m,1}$, and the group $\mathrm{Res}_{K/k}(\mathbf{L})$ is isogenous to a closed subgroup of an admissible $k$-form $\mathbf{L}_S$ of $\mathbf{PO}_{d(m+1)-1,1}$, which corresponds to an arithmetic hyperbolic orbifold $S$ of dimension $d(m+1)-1$  such that $N \subseteq S$;
    \item $S$ is realised as a subform subspace in $M$, and $N$ is a Weil restriction subspace in $S$. In particular $\mathrm{dim}(S)=d(m+1)-1\leq n$.
\end{enumerate}
    
We also see that all totally geodesic subspaces of low enough codimension in $M$ are subform subspaces:

\begin{corollary}\label{cor:max-dim-subform-space}
Let $M=\mathbb{H}^n/\Gamma$ be an arithmetic hyperbolic orbifold with $M$ either type I, II or $7$-dimensional type III. Then all totally geodesic subspaces of dimension $m > (n-1)/2$ that are not $3$-dimensional type III are subform subspaces. 
\end{corollary}
\begin{proof}
The dimension $m$ of a subspace of $M$ which is not a subform subspace is subject to the constraint $d(m+1)-1 \leq n$, where $d=[K:k]>1$. The maximum possible dimension is achieved when $d=2$, so we get that $2(m+1)-1 \leq n$. This implies $m \le (n+1)/2 -1 = (n-1)/2$.
\end{proof}
\medskip

\subsection{Three-dimensional arithmetic submanifolds of type III}
Finally, we can discuss the obstruction to extending Theorem \ref{teo:embeddings-hyperbolic-orbifolds} to the case where the totally geodesic subspace $N$ is $3$-dimensional type-III. We first notice that the adjoint trace field of a $3$-dimensional type-III lattice can never coincide with the one of an arithmetic hyperbolic lattice of dimension $\geq 4$, since the former is not totally real while the latter is. As such, it is impossible for a $3$-dimensional type-III orbifold to be geodesically immersed as a subform space in a higher dimensional manifold in the sense of Definition~\ref{def:subform-space}. Moreover, there is in general no reason to expect an inclusion of the adjoint trace field of the ambient manifold $M$ in the adjoint trace field of $N$ such as the one described in the setting of Theorem \ref{theorem:geod}.

However, we stress that when working with a $3$-dimensional type-III lattice $\Lambda$, the right field to look at is the invariant trace field $L$, and the right group to look at is the complex ambient group. Indeed, the arithmetic structure of $\Lambda$ (namely, it being commensurable with the group of integer points of an admissible form of $\mathbf{PSO}_{3,1}$ defined over some number field) is only evident when one considers the complex ambient group $\mathbf{PGL}_A$, with $A$ the invariant quaternion algebra of $\Lambda$. If one instead considers the adjoint trace field $K$ and the real ambient group $\mathbf{L}$ of $\Lambda$, the latter is not admissible and $\Lambda<\mathbf{L}(K)$ can never be commensurable with its group of integer points, as this would give a higher-rank irreducible lattice and therefore violate Margulis' Superrigidity Theorem.

Assuming that a $3$-dimensional type-III orbifold $N=\mathbb{H}^3/\Lambda$ is geodesically immersed in an arithmetic orbifold $M=\mathbb{H}^n/\Gamma$, by Remark \ref{prop:invariant-trace-field-inclusion} we have that the adjoint trace field $k$ of $\Gamma$ is a subfield of the invariant trace field $L$ of $\Lambda$. Moreover, denoting by $\mathbf{H}$ the rational stabilizer of $N$ in $M$, we have that $L$ is the minimal field of definition of the natural projection map $\mathbf{H}\rightarrow \mathbf{PGL}_2$. Again by Remark \ref{prop:invariant-trace-field-inclusion}, we see that the group $\mathrm{Res}_{L/k}\, \mathbf{PGL}_A$ is a $k$-simple factor of the adjoint group $\mathbf{PH}^{\circ}$.  Unlike the case where $N$ is type I or II, there is no field inclusion of the form $k<K<L$, so it is not possible to obtain the real ambient group $\mathbf{L}$ of $N$ via restriction of scalars from $L$ to $K$ of $\mathbf{PGL}_A$.

Since $L$ has only one pair of conjugate complex embeddings and is not an imaginary quadratic extension of $k$, we see that there is at least one field embedding $\sigma \in S^{\infty}_{L/k}$ such that $\sigma(L) \subset \mathbb{R}$, and therefore the $\mathbb{R}$-simple factors of $\mathrm{Res}_{L/k}\, \mathbf{PGL}_A$ consist of a single factor isomorphic to $\mathbf{PSO}_{3,1}$ and a non-empty set of factors isomorphic to $\mathbf{PSO}_3$.

The contribution of every $\mathbf{PSO}_3$-factor to the total codimension of $N$ in $M$ depends on the factorization of the representation of $\mathrm{Res}_{L/k}\, \mathbf{GL}_A(\mathbb{R})< \mathbf{SO}_{n,1}(\mathbb{R})$ as a direct sum of irreducible representations. Recall that the real Lie group $\mathrm{Spin}_4$ is isomorphic to $\mathrm{Spin}_3 \times \mathrm{Spin}_3$ via left- and right- quaternion multiplication, and this isomorphism induces an isogeny between $\mathbf{SO}_4$ and $\mathbf{SO}_3 \times \mathbf{SO}_3$. The two resulting representations of $\mathrm{Spin}_3$ as groups of isometries of the $4$-dimensional Euclidean space are both irreducible. Therefore, each $\mathbf{PSO}_3$-factor corresponds to an irreducible action of the preimage $\widetilde{\mathbf{H}}$ of $\mathbf{H}$ on either a $3$-dimensional or a $4$-dimensional subspace of $\mathbb{R}^{n+1}$, and moreover a pair of distinct $\mathbf{PSO}_3$-factors may correspond to irreducible actions on the same $4$-dimensional subspace. For example, in the trialitarian case described in Section \ref{sec:type3}, there is a single $\mathbf{PSO}_3$-factor that yields an irreducible action of $\widetilde{\mathbf{H}}$ on a $4$-dimensional subspace (and the codimension of the totally geodesic $3$-dimensional subspace inside the $7$-dimensional trialiatrian orbifold is indeed equal to $4$).

In particular, the discussion above leads to the following:
\begin{remark}\label{embeddings-3-dim-type-III}
Let $N=\mathbb{H}^3/\Lambda$ be a $3$-dimensional type-III totally geodesic suborbifold of an arithmetic hyperbolic orbifold $M=\mathbb{H}^n/\Gamma$. Then $N$ has codimension $n-3$ at least $3$ in $M$, i.e.\ $n\geq 6$. Moreover, the stabilizer $\Sigma=\mathrm{Stab}_{\Gamma}(U)$ of a lift $U$ of the universal cover of $N$ is an irreducible lattice in the stabilizer 
$\mathrm{Stab}_{G}(U)\cong(\mathrm{O}_{3,1}\times \mathrm{O}_{n-3})/\langle (-\mathrm{id},-\mathrm{id})\rangle$
of $U$ in $G=\mathrm{Isom}(\mathbb{H}^n)$.
\end{remark}

Notice how the case where $N$ is $3$-dimensional of type I or II is different. By Remark \ref{lemma:atf-and-projection-field-coincide} we may consider the sequence of field inclusions $k\subset K \subset L=K(\sqrt{\alpha}) \subset \overline{L}$, where $[L:K]=2$ and $\overline{L}$ denotes the Galois closure of the extension $L/k$. Since $L$ has a single complex place, we may assume that the element $\alpha \in K$ is negative, but it becomes positive under any non-identity embedding of $K$ relative to $k$. 

We obtain an isomorphism 
$$\mathcal{G}=\mathrm{Gal}(\overline{L}/k)\cong (\mathbb{Z}/2\mathbb{Z})^d \rtimes \mathrm{Gal}(\overline{K}/k),$$ where $\overline{K}$ is the Galois closure of $K/k$ and $d=[K:k]$. The $d$ entries of the $(\mathbb{Z}/2\mathbb{Z})^d$ correspond to the field embeddings of $K$ relative to $k$ (with the first entry corresponding to the identity embedding), and the action of $\mathrm{Gal}(\overline{K}/k)$ on $(\mathbb{Z}/2\mathbb{Z})^d$ via permutations of the entries is induced by the action of $\mathrm{Gal}(\overline{K}/k)$ on these embeddings.
Denoting by $\sigma \in \mathrm{Gal}(\overline{L}/K)<\mathrm{Gal}(\overline{L}/k)$ the automorphism of $\overline{L}$ given by complex conjugation (which exchanges $\sqrt{\alpha}$ and $-\sqrt{\alpha}$), we see that it corresponds to $(-1,1,\dots,1) \in (\mathbb{Z}/2\mathbb{Z})^d$. The $\mathcal{G}$-conjugates of $\sigma$ act by permuting exactly one pair of $\mathbf{PSO}_3$-factors of $\mathrm{Res}_{L/k}\mathbf{PGL}_A$ (those corresponding to elements of the form $\pm\sqrt{\tau(\alpha)}$ for $\tau \in S^{\infty}_{K/k}$). It follows that the action of each factor of $$\mathrm{Res}_{K/k}\,\widetilde{\mathbf{L}}\cong \prod_{\tau \in S^{\infty}_{K/k}} \widetilde{\mathbf{L}}^{\tau}$$ is $\mathrm{Gal}(\overline{K}/k)$-conjugate to the irreducible action of $\widetilde{\mathbf{L}}$ on $\mathbb{R}^4$ via elements of $\mathrm{O}_{3,1}$. This implies that each factor of the form $\widetilde{\mathbf{L}}^{\tau}$ acts on a $4$-dimensional subspace of $\mathbb{R}^{n+1}$.

\section{Finite centraliser subspaces and arithmeticity}\label{sec:proof-of-main-theorems}

In this section, we prove our arithmeticity criterion in terms of finiteness/infiniteness of fc-subspaces, and provide examples of nonarithmetic hyperbolic orbifolds which contain non-fc-subspaces of codimension one. We also exhibit examples of non-fc subspaces in arithmetic orbifolds.

\subsection{Proof of Theorem~\ref{theorem-fc}.}\label{sec:fc-spaces} \textit{(i).} Let $\Gamma < \mathbf{PO}_{n,1}(\R)$ be an arithmetic group defined over a totally real field $k$. Then $\Gamma$ is either a type I, type II or type III lattice. According to the discussion in Sections~\ref{sec:k-involutions-type1}, \ref{sec:k-involutions-type2}, \ref{sec:k-involutions-dim3} and by  Proposition \ref{prop:inner-involutions-7-dim-type-III}, there exists a $k$-involution $\theta \in \mathrm{Comm}(\Gamma)$. In all cases, we can choose the involution $\theta$ in such a way that the fixed point set of its action on $\mathbb{H}^n$ has positive dimension. Conjugating by the elements of $\mathrm{Comm}(\Gamma)=\mathbf{G}(k)$, which is dense in $\G(\R)^{\circ}$ (cf. \cite[Proposition 5.1.8]{WM}), we obtain countably many different fc-subspaces.

\subsubsection*{(ii)} If $\Gamma$ is arithmetic and not $3$-dimensional type III, by combining Corollary \ref{cor:max-dim-subform-space} with Proposition \ref{prop:subform-space-is-fc} we prove that if $M=\mathbb{H}^n/\Gamma$  and $n$ is odd (resp.\ even), then all subspaces of dimension $m\geq (n+1)/2$ (resp.\ $m\geq n/2$) which are not $3$-dimensional type III are fc-subspaces. However, this bound can be improved when $n$ is odd. We now consider this case more carefully.

Let $N$ be a Weil restriction subspace in $M$ and suppose that $[K:k]=2$ where $K$ and $k$ denote the adjoint trace fields of $N$ and $M$ respectively. Notice that these conditions imply that $n=\mathrm{dim}(M)$ is odd and that $\mathrm{dim}(N)=m=(n+1)/2-1=(n-1)/2$. 

We claim that under the hypotheses above $N$ is necessarily an fc-subspace.
Let us denote by $U$ a lift of $N$ to $\mathbb{H}^n$ and by $\mathbf{H}$ the Zariski closure of $\mathrm{Stab}_{\mathbf{G}(k)}(U)$ in the ambient group $\mathbf{G}$ of $M$. From the proof of Theorem \ref{teo:embeddings-hyperbolic-orbifolds} we have that $\mathbf{H}$ is $k$-defined and that 
\begin{equation}\label{eq:Weil-res-BC-group}
    \mathbf{H}\cong \frac{\mathrm{Res}_{K/k}\,\widetilde{\mathbf{L}}}{\langle-\mathrm{id}\rangle}\cong \frac{\widetilde{\mathbf{L}} \times \widetilde{\mathbf{L}}^{\sigma}}{\langle(-\mathrm{id},-\mathrm{id})\rangle} ,\end{equation} 
where $\widetilde{\mathbf{L}}$ denotes the $K$-form of $\mathbf{O}(m,1)$ isogenous to the ambient group $\mathbf{L}$ of $N$ and $\sigma:K\rightarrow K$  the nontrivial element of $\mathrm{Gal}(K/k)\cong \mathbb{Z}/2\mathbb{Z}$. The action of the Galois automorphism $\sigma$ maps a pair $(A,B) \in \widetilde{\mathbf{L}}_K \times \widetilde{\mathbf{L}}^{\sigma}_K$ to the pair $(B^{\sigma},A^{\sigma})$. Since $(-\mathrm{id},-\mathrm{id})$ is preserved by $\sigma$, this action descends to the group $\mathbf{H}$. Moreover, in $\mathbf{H}$ we have that $[(\mathrm{id},-\mathrm{id})]=[(-\mathrm{id},\mathrm{id})]$, therefore the element $[(\mathrm{id},-\mathrm{id})] \in \mathbf{H}(K)$ is fixed by the action of $\sigma$ and thus is an involution in $\mathbf{H}(k)\subset\mathbf{G}(k)=\mathrm{Comm}(\Gamma)$ with $U$ as its fixed point set. This implies that $N$ is an fc-subspace, and thus all subspaces of $M$ of dimension $m\geq (n-1)/2$, which are not $3$-dimensional type-III, are fc-subspaces.

In the case where $\Gamma$ is $3$-dimensional type III, we only need to prove that all $1$-dimensional totally geodesic subspaces are fc-subspaces. This follows by applying Proposition \ref{prop:Jorgensen-involutions}.

\subsubsection*{(iii)} By Margulis' theorem \cite[Chapter IX, Th. B]{Margulis-book}, if $\Gamma$ is non-arithmetic, then $\Gamma' = \mathrm{Comm}(\Gamma)$ is the maximal lattice containing $\Gamma$ with finite index. Since $\Gamma'$ is a lattice, it contains only finitely many conjugacy classes of finite subgroups (this fact is well known for word hyperbolic groups, and hence for cocompact hyperbolic lattices; for a general result we refer to \cite{Sam14}). Now assume that two finite subgroups $F_1$, $F_2 < \mathrm{Comm}(\Gamma) = \Gamma'$ are conjugate in $\Gamma'$, so that $F_1 = \gamma F_2\gamma^{-1}$, $\gamma \in \Gamma'$. Then their fixed point sets $H_i = \Fix(F_i)$ in $\HH^n$ satisfy $H_1 = \gamma H_2$.

Let $\Gamma_i = \mathrm{Stab}_\Gamma(H_i) = \{ \alpha\in\Gamma \mid \alpha H_i = H_i\}, \, i = 1, 2$, be the corresponding stabilisers. Then $$\alpha\in\Gamma_1 \Leftrightarrow \alpha H_1 = H_1\Leftrightarrow \alpha \gamma H_2 = \gamma H_2\Leftrightarrow \gamma^{-1}\alpha\gamma H_2 = H_2.$$
Hence $\Gamma_1 = \gamma \Gamma_2 \gamma^{-1}$ and therefore we have only finitely many $\Gamma'$-conjugacy classes of the stabilisers. Since $\Gamma$ has finite index in $\Gamma' = \mathrm{Comm}(\Gamma)$, we have finitely many $\Gamma$-conjugacy classes. Therefore, there are only finitely many fc-subspaces.

\subsubsection*{(iv)} In \cite{Sam14}, Samet obtained an effective upper bound for the number of conjugacy classes of finite subgroups of a lattice in terms of covolume. Let $\Gamma' =  \mathrm{Comm}(\Gamma)$ be the commensurator of the non-arithmetic lattice $\Gamma$. As discussed in part (iii), the fc-subspaces of $M = \HH^n/\Gamma$ 
correspond to the strata in the natural orbifold stratification of $M' = \HH^n/\Gamma'$ defined by 
$$M'_{[F]} = H/\mathrm{Stab}_{\Gamma'}(H),$$
where $F<\Gamma'$ is a finite subgroup, $H = \mathrm{Fix}(F)$, and any finite subgroup conjugate to $F$ in $\Gamma'$ gives the same stratum.

By \cite[Theorem~1.3]{Sam14} the number of such strata is bounded above by $c\cdot\mathrm{vol}(\HH^n/\Gamma')$, with $c = \mathrm{const}(n)$, a constant that depends on $n$ only. This quantity bounds the number of fc-subspaces of $M$ up to conjugation in $\Gamma'$. To count them up to conjugation in $\Gamma$ we need to multiply the former bound by the index $[\Gamma':\Gamma]$, which gives $c\cdot\mathrm{vol}(\HH^n/\Gamma')[\Gamma':\Gamma] = c\cdot\mathrm{vol}(\HH^n/\Gamma).$

\medskip
As for the main statements of the theorem, we have that (1) follows from (i)--(ii), and (2) follows from (iii)--(iv). \qed

\subsection{Examples of non-fc-subspaces and proof of Theorem~\ref{teo:existence-not-fc-2}}\label{sec:non-fc}

By Theorem \ref{theorem-fc}, all codimension $1$ totally geodesic suborbifolds in an arithmetic orbifold are fc. For non-arithmetic lattices, Theorem~\ref{teo:existence-not-fc-2} gives examples of maximal totally geodesic subspaces of codimension $1$ which are not fc. Note that by \cite{BFMS20} the number of such subspaces is always finite. 

These examples come from the hybrid non-arithmetic lattices constructed by Gromov and Piatetski-Shapiro in \cite{GPS87}. We briefly recall the construction as described by Vinberg in \cite{Vin14} for the reader's convenience.

Suppose that $\Gamma_1,\Gamma_2<\Isom(\mathbb{H}^n)$ are two lattices both containing a reflection $r$ in a hyperplane $H \subset \mathbb{H}^n$ which satisfy the following conditions:

\begin{enumerate}
    \item for any $\gamma \in \Gamma_i$, $i=1,2$, either $\gamma(H)=H$ or $\gamma(H) \cap H = \emptyset$,
    \item $N_{\Gamma_1}(r)=N_{\Gamma_2}(r)= \langle r \rangle \times \Gamma_0$, where $\Gamma_0$ leaves invariant the two half-spaces bounded by $H$. Here $N_{\Gamma_i}(r)$ denotes the normaliser of $r$ in $\Gamma_i$.
    
\end{enumerate} This means that $H$ projects to the same embedded fc-subspace in both quotient orbifolds of $\mathbb{H}^n$ by $\Gamma_1$ and by $\Gamma_2$.

Consider the set $\{\gamma(H)| \gamma \in \Gamma_i\}$ of translates of $H$ for a fixed $i$. This set decomposes the space $\mathbb{H}^n$ into a collection of closed pieces transitively permuted by $\Gamma_i$, and each of these pieces is a fundamental domain for the action of $N_i = \langle \gamma r \gamma^{-1} \mid \gamma \in \Gamma_i \rangle$, the normal closure of $r$ in $\Gamma_i$. Let $D_i$ be one of these pieces, and let $\Delta_i=\{\gamma \in \Gamma_i | \gamma(D_i)=D_i\}$. Then the group $\Gamma_i$ admits a decomposition as a semidirect product $\Gamma_i=N_i\rtimes \Delta_i$. Moreover, by choosing $D_1$ and $D_2$ to lie on opposite sides of the hyperplane $H$, we can ensure that $\Delta_1 \cap \Delta_2=\Gamma_{0}$.

The following result from \cite{GPS87} shows that we can build a new hybrid lattice out of $\Gamma_1$ and $\Gamma_2$ as above.

\begin{theorem}[Gromov and Piatetski-Shapiro]\label{th:GPS}
The group $\Gamma = \langle \Delta_1, \Delta_2 \rangle$ is a lattice in $\mathrm{Isom}(\mathbb{H}^n)$. If $\Gamma_1$ and $\Gamma_2$ are incommensurable arithmetic groups, then the lattice $\Gamma$ is non-arithmetic. 
\end{theorem}

As an abstract group, $\Gamma=\Delta_1 *_{\Gamma_0} \Delta_2$ (the amalgamated product of $\Delta_1$ and $\Delta_2$ along their common subgroup $\Gamma_{0}$).
It is not difficult to build, for all $n\geq 3$, incommensurable torsion-free arithmetic lattices of simplest type $\Gamma_1, \Gamma_2 < \mathrm{Isom}(\mathbb{H}^n)$ satisfying the conditions above (see, for instance, \cite{GPS87} or \cite{KRS}). 

It is fairly easy to check that, with the construction above, the hyperplane $H$ projects to a totally geodesic subspace $H/\Gamma_0$ in the orbifold $\mathbb{H}^n/\Gamma$. We claim that this is not an fc-subspace.

First, we observe that $\Gamma_i = \langle \Delta_i, r \rangle$. Indeed, the hyperplane $H$ projects onto a connected totally geodesic subspace in the orbifold $M_i = \HH^n/\Gamma_i$. Removing this connected subspace gives an infinite-volume orbifold $M_i^+$ whose universal cover is precisely the domain $D_i$. In fact, $M_i^+ = D_i/\Delta_i$. The deck group of any covering acts transitively on the fibers, and therefore $N_i$ is generated by $\Delta_i$-conjugates of $r$. Hence, $\Gamma_i = \langle \Delta_i, r \rangle$. 

Arguing by contradiction, we assume that the reflection $r$ in the hyperplane $H$ commensurates the lattice $\Gamma$. Theorem~\ref{th:GPS} together with the Margulis commensurator rigidity implies that $\mathrm{Comm}(\Gamma)$ is itself a lattice. Moreover, $\mathrm{Comm}(\Gamma)$ contains $r$ and each $\Delta_i$. Hence, lattices $\Gamma_i = \langle \Delta_i, r \rangle$ are contained in the lattice $\mathrm{Comm}(\Gamma)$, whence they are finite-index subgroups of $\mathrm{Comm}(\Gamma)$. This implies that $\Gamma_1$ and $\Gamma_2$ are commensurable. The proof of Theorem \ref{teo:existence-not-fc-2} is now complete. \qed

\begin{remark}
As shown by Vinberg \cite{Vin14}, the above construction of hybrid lattices can also be used to construct examples of finite-volume non-arithmetic Coxeter polytopes. The discussion above applies without modifications to these examples, providing non-arithmetic Coxeter lattices which contain codimension-one non-fc sublattices.
\end{remark}

\subsection{Proof of Theorem \ref{teo:existence-not-fc-twist-knots}}\label{sec:Le-Palmer}
Here we provide examples of $3$-dimensional, non-arithmetic hyperbolic manifolds $M_k$, $k\geq 1$, with the property that all their codimension one totally geodesic subspaces are non-fc and there are precisely $k$ of them. As far as we can see, unlike the examples by Gromov and Piatetski-Shapiro discussed above, these subspaces do not arise as the glueing locus of a hybrid lattice. These manifolds $M_k$ are obtained by Le and Palmer \cite{LP20} as the $k$-sheeted cover of some hyperbolic $3$-manifold $N_j$ described below.

\begin{prop}[Le and Palmer \cite{LP20}]
Let $N_j=\mathbb{S}^3 \setminus K_j$ be the (hyperbolic) complement in $\mathbb{S}^3$ of a twist knot $K_j$ with $j$ half-twists. If the number $j$ is an odd prime, then $N_j$ contains a unique immersed totally geodesic surface $S$ homeomorphic to a thrice-punctured sphere (see Fig.~\ref{fig:twist-knots}).
\end{prop}

\begin{figure}
    \centering
    \includegraphics[width=6.5cm]{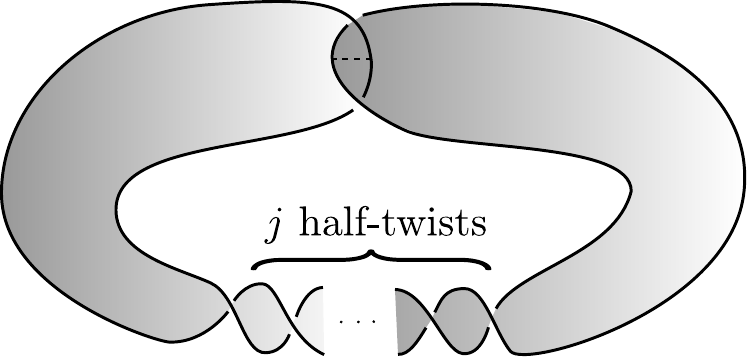}
    \caption{A twist knot in the $3$-sphere, together with the immersed thrice-punctured sphere $S$ (shaded). The self-intersection of $S$ is drawn as a dashed line.}
    \label{fig:twist-knots}
\end{figure}

Notice that if $j=2$, the manifold $N_j$ is the complement of the figure-eight knot, which is the only arithmetic hyperbolic knot complement by the work of Reid \cite{Reid}. For all $j\neq 2$ the manifold $N_j$ is therefore non-arithmetic. We claim that in the latter case the surface $S$ is non-fc.

\begin{prop}
Let $N_j$ be the (hyperbolic) complement in $\mathbb{S}^3$ of a twist knot $K_j$ with $j$ half-twists, where $j$ is an odd prime. The totally geodesic thrice-punctured sphere $S$ is not an fc-subspace.
\end{prop}

\begin{proof}
This follows from the work of Reid and Walsh \cite[Theorem 3.1]{RW}, which proves that if the complement in $\mathbb{S}^3$ of a $2$-bridge knot is non-arithmetic, then it admits no hidden symmetries. Translated in our terminology, this means that if $N = \mathbb{H}^3/\Gamma$ is a non-arithmetic $2$-bridge knot, then $\mathrm{Comm}(\Gamma)$ coincides with the normaliser of $\Gamma$ in $\mathrm{Isom}(\mathbb{H}^3)$. Equivalently, the elements of $\mathrm{Comm}(\Gamma)$ correspond to symmetries of $N$ (rather than symmetries of some finite-index cover of~$N$).

Twist knots are particular examples of $2$-bridge knots, so the above result applies to $N = N_j$, $j\neq 2$. It therefore suffices to prove that there is no symmetry of $N_j$ that fixes $S$ pointwise. Note that such a symmetry would have to be a reflection in $S$, and therefore an orientation reversing element of $\mathrm{Isom}(N_j)$.

However, the Jones polynomial of the twist knot $K_j$ , $j$ odd, is
\begin{equation*}\label{eq:jones}
V(q)=\frac{1+q^{-2}+q^{-j}+q^{-j-3}}{q+1}.
\end{equation*}
Since $V(q) \neq V(q^{-1})$, the twist knot complement $N_j$ is chiral, i.e.\ it admits no orientation-reversing symmetries.
\end{proof}

\medskip

As shown in \cite{LP20}, for any $N_j$ as above there exists a degree $k$ cover $M_k$ that contains exactly $k$ lifts of $S$, and these are the only  totally geodesic surfaces in $M_k$. Then Theorem~\ref{teo:existence-not-fc-twist-knots} follows. \qed

\medskip
\begin{remark}
All twist knot complements $N_j$ can be obtained through Dehn filling on one component of the Whitehead link complement $\mathbb{S}^3\setminus W$, which is hyperbolic and arithmetic. Moreover, in this case the thrice-punctured sphere $S$ does not intersect the core geodesic of the filling. The surface $S$ corresponds to a totally geodesic thrice-punctured sphere $S'\subset \mathbb{S}^3\setminus W$ which is an fc-subspace by Theorem \ref{theorem-fc}. We therefore obtain that the Dehn filling turns an fc-subspace $S'$ into the non-fc subspace $S$. 
\end{remark}

\subsection{Non-fc subspaces in arithmetic orbifolds}\label{sec:non-fc-arithmetic}
We conclude this section by providing examples of non-fc Weil restriction subspaces in arithmetic lattices of type I. We begin by choosing an extension $K/k$ of totally real fields of odd degree $d$ and an admissible quadratic form $f$ of signature $(m,1)$ defined over $K$. We denote by $\overline{K}$ the Galois closure of the extension $K/k$. Let us denote by $g$ the form $\mathrm{Res}_{K/k}\,f$.  By Proposition \ref{prop:res_of_scalars_typeI} we can choose an arithmetic lattice $\Lambda<\mathrm{PO}_f(K)$ and realise it as a totally geodesic sublattice in an arithmetic lattice $\Gamma<\mathrm{PO}_g(k)$. Notice that the form $\mathrm{Res}_{K/k}\,f$ is $k$-defined, admissible and has signature $(n,1)$ with $n=d\cdot(m+1)-1$.  We claim that $N=\mathbb{H}^m/\Lambda$ is a non-fc subspace of $M=\mathbb{H}^n/\Gamma$.

Let us denote by $\mathcal{U}\cong \mathbb{H}^m$ a lift of $N$ to $\mathbb{H}^n$, and suppose that $\gamma \in \mathbf{O}_g(\mathbb{R})$ commensurates $\Gamma$ and fixes $\mathbf{U}$ pointwise. It follows that $[\gamma] \in \mathbf{PO}_g(k)$ and moreover 
\begin{equation*}
[\gamma] \in \frac{\mathbf{O}_f \times \mathbf{O}_{f^{\sigma_1}}\times \dots \times \mathbf{O}_{f^{\sigma_{d-1}}}}{\langle(-\mathrm{id},-\mathrm{id},\dots,-\mathrm{id})\rangle}
\end{equation*} is invariant under the natural action of $\mathrm{Gal}(\overline{K}/k)$ and is mapped by the projection onto the first factor to the identity element $[\pm \mathrm{id}] \in \mathbf{PO}_f$. Since $\mathrm{Gal}(\overline{K}/k)$ acts transitively on the various factors of the form $\mathbf{O}_{f^{\sigma}}$ and $(\pm\mathrm{id})^{\sigma}=\pm \mathrm{id}$, it follows that necessarily 
\begin{equation}\label{eq:g_signs}
\gamma=(\pm \mathrm{id},\pm \mathrm{id},\dots \pm \mathrm{id}) \in \mathbf{O}_f \times \mathbf{O}_{f^{\sigma_1}}\times \dots \times \mathbf{O}_{f^{\sigma_{d-1}}}.
\end{equation}

We claim that $\gamma=\pm(\mathrm{id}, \mathrm{id},\dots,\mathrm{id})$, and thus $[\gamma]=\mathrm{id} \in \mathbf{PO}_g$. Indeed, denote by $n_+$ the number of $+\mathrm{id}$ entries and by $n_-$ the number of $-\mathrm{id}$ entries in \eqref{eq:g_signs}. It is clear that the action of $\mathrm{Gal}(\overline{K}/k)$ does not change the values of $n_+$ and $n_-$, as it simply permutes the entries, while multiplication by $(-\mathrm{id},-\mathrm{id},\dots,-\mathrm{id})$ changes $\mathrm{id}$ to $-\mathrm{id}$ and vice-versa, thus exchanging the values of $n_+$ and $n_-$. Since we require that $\gamma^{\sigma}$ be equal to $\pm \gamma$ for any $\sigma \in \mathrm{Gal}(\overline{K}/k)$, the only possibilities are the following:
\begin{enumerate}
    \item $n_+=d$, $n_-=0$, $\gamma^{\sigma}=\gamma$ for all $\sigma \in \mathrm{Gal}(\overline{K}/k)$, and $$[\gamma]=\mathrm{id} \in \mathbf{PO}_g;$$
    \item $n_+=0$, $n_-=d$, $\gamma^{\sigma}=\gamma$ for all $\sigma \in \mathrm{Gal}(\overline{K}/k)$, and $$[\gamma]=\mathrm{id} \in \mathbf{PO}_g;$$
    \item both $n_+$ and $n_-$ are non-zero, and this implies that $\gamma^{\sigma}=-\gamma$ for some $\sigma \in \mathrm{Gal}(\overline{K}/k)$. 
\end{enumerate}
Case (3) is impossible. Indeed, this would imply that $n_+=n_-$, but in our case, $d = n_+ + n_-$ is odd. This proves the claim.

We conclude by noticing that $N$ is by construction a proper Weil restriction subspace of $M$. Moreover we have just proven that $\gamma \in \mathrm{Comm}(\Gamma)$ fixes $\mathcal{U}$ pointwise if and only if $[\gamma]$ is the identity element in the group $\mathbf{PO}_g(\mathbb{R}) \cong\mathrm{Isom}(\mathbb{H}^n)$. This implies that there is no proper fc-subspace of $M$ that contains $N$, and therefore $N$ itself is not an fc-subspace.

\section{Examples}\label{sec:example-section}
\subsection{Involutions and quadratic extensions}\label{sec:invol-quadr-extension}
We now provide explicit examples of involutions in $\mathbf{PO}_f (k)$ which cannot be represented by an element of $\mathrm{O}_f(k)$, as discussed in Section~\ref{sec:k-involutions-type1}. Here $k=\mathbb{Q}$ and $f$ is a (trivially) admissible form of signature $(n,1)$, with $n$ odd. An analogous example in the setting of arithmetic lattices in $\mathbf{PGL}_2(\mathbb{C})$ is described in \cite[p. 1314]{KRS}.

Let $k=\mathbb{Q}$, $V=\mathbb{Q}^{n+1}$ and consider the symmetric bilinear form whose matrix representation with respect to the standard basis $(e_0, e_1, \dots, e_n)$ is given by the following block-diagonal matrix:

$$A=\begin{bmatrix}0&1 \\ 1 & 0  \end{bmatrix} \oplus_m \begin{bmatrix}2&0 \\ 0 & 2  \end{bmatrix},$$ where $m=(n-1)/2$.

Now, consider the following matrix:
$$M=\begin{bmatrix}0&-1/\sqrt{5} \\ -\sqrt{5} & 0  \end{bmatrix} \oplus_m \begin{bmatrix}2/\sqrt{5}&1/\sqrt{5} \\ 1/\sqrt{5} & -2/\sqrt{5}\end{bmatrix}.$$

The patient reader can manually check that $M$ is invertible and that conjugation of a matrix $N \in \mathrm{GL}_{n+1}$ by $M$ is a polynomial function with \textit{rational} coefficients in the entries of $N$. Moreover, $M^{t}AM=A$ and $M^2=\mathrm{id}$, therefore the matrix $M$ corresponds to an involution in $\mathbf{PO}_f (\mathbb{Q})$.

The positive eigenspace $V^+$ relative to the eigenvalue $1$ for $M$ has dimension $(n+1)/2$, with orthogonal basis $\mathcal{B}_+$ given by:
$$\mathcal{B}_+=(e_0-\sqrt{5}e_1, e_{2i}+(\sqrt{5}-2)e_{2i+1}),\;i=1,\dots,m.$$

The negative eigenspace $V^-$ relative to the eigenvalue $-1$ has the same dimension $(n+1)/2$, with orthogonal basis $\mathcal{B}_-$ given by:
$$\mathcal{B}_-=(e_0+\sqrt{5}e_1, e_{2i}+(-\sqrt{5}-2)e_{2i+1}),\;i=1,\dots,m.$$
Notice how the vectors of $\mathcal{B}_-$ are obtained from those of $\mathcal{B}_+$ simply by applying to each coordinate the non-trivial Galois automorphism $\sigma$ of $\mathbb{Q}(\sqrt{5})$ which maps $\sqrt{5}$ to $-\sqrt{5}$.

The restriction $g$ of the form $f$ to $V^+$ is represented with respect to $\mathcal{B}_+$ by the diagonal matrix with one entry equal to $-2\sqrt{5}$ and all other entries equal to $20-8\sqrt{5}$ (which is positive). Similarly, the restriction $h$ of the form $f$ to $V_-$ is represented with respect to $\mathcal{B}_-$ by the diagonal matrix with one entry equal to $2\sqrt{5}$ and all other entries equal to $20+8\sqrt{5}$.

Thus we see that $g$ has signature $((n-1)/2,1)$ and  $h=g^{\sigma}$ is positive definite, so that $g$ is admissible. Moreover, the group 
$$\mathrm{Res}_{\mathbb{Q}(\sqrt{5})/\mathbb{Q}}\mathbf{O}(g)_{\mathbb{R}}=\mathbf{O}(g)_{\mathbb{R}} \times \mathbf{O}(h)_{\mathbb{R}}$$
is realised as the subgroup of $\mathrm{O}(f,\mathbb{R})$ which preserves the decomposition 
$\mathbb{R}^{n+1}=(V^+\otimes \mathbb{R}) \oplus (V^-\otimes \mathbb{R}).$ 
The space $U=\mathbb{H}^n \cap (V^+\otimes\mathbb{R})$ projects to an arithmetic finite-volume totally geodesic subspace in $\mathbb{H}^n/\mathrm{PO}_f(\Z)$ which is a Weil restriction subspace with adjoint trace field $\mathbb{Q}(\sqrt{5})$ and ambient group $\mathbf{PO}_g$.

\subsection{An embedding of a type-I lattice in a type-II lattice}\label{sec:type-I-in-type-II}

In this subsection, we construct an explicit example of a type-I arithmetic hyperbolic orbifold realised as a Weil restriction subspace of a type-II arithmetic hyperbolic orbifold.

Consider the rational quaternion algebra $D' = \left(\frac{-1,\, 3}{\mathbb{Q}}\right)$. By the discussion in \cite[p.~88]{Mac-Reid}, $D'$ is a division algebra. Let $K = \mathbb{Q}(\sqrt{3})$. Notice that the $K$-algebra
$D = D' \otimes K$ splits since $3$ is a square in $K$. We fix an isomorphism $D\cong \left(\frac{1,\, 1}{\mathbb{Q}(\sqrt{3})}\right)\cong M_2(\mathbb{Q}(\sqrt{3}))$ and denote by $\mathbf{i}, \mathbf{j}$ the standard generators of the split quaternion algebra.

Let us now consider the admissible $K$-form of signature $(2m-1, 1)$ given by 
$$
f(\mathbf{x})=-\sqrt{3}x_0^2+x_1^2+\ldots+x_{2m-1}^2.
$$ 

The form $f$ can be interpreted as a form on $D^m_{+}=\{x \in D^m | x\,\mathbf{i} = x\}$, which is a $K$-subspace of dimension $2m$ of the right $D$-module $D^m$ as in Section \ref{sec:unitary-lattices}: it sufficient to fix a basis for $D^m_{+}$ in order to identify it with $K^{2m}$. 

We now extend the form $f$ to a skew-Hermitian form $F$ on $D^m$ by setting
\begin{multline*}
F(x_1 + y_1\, \mathbf{j}, x_2 + y_2\, \mathbf{j} ) = f(x_1, x_1)(\I-1)\J + f(x_1, y_2)(\I-1) + \\
+ f(x_2, y_1)(\I+1) + f(y_1, y_2)(\I+1)\J
\end{multline*} 
for all $x_1, y_1, x_2, y_2 \in D_+^m$.
The admissibility of $F$ follows directly from the admissibility of the initial form $f$.

Let us now consider an arithmetic lattice $\Lambda < \mathrm{U}_F(D)$. Since $D \cong M_2(K)$, we have that $\Lambda$ is a type-I lattice by the discussion in Section \ref{sec:unitary-lattices}. By applying Proposition \ref{prop:Weil-restriction-type-II} we realise $\Lambda$ as a totally geodesic sublattice in $\Gamma<\mathrm{U}_G(D')$, where $G=\mathrm{Res}_{K/k}\,F$ is an admissible skew-Hermitian form on $(D')^{2m}$. It follows that $\Gamma$ is a type-II lattice and that the type-I orbifold $M=\mathbb{H}^{2m-1}/\Lambda$ is realised as a Weil restriction subspace in the type-II orbifold $N=\mathbb{H}^{4m-1}/\Gamma$.

\subsection{One curious  example: a $5$-dimensional simplex}\label{sec:examples}

We turn our attention to the totally geodesic suborbifolds of the orbifold $\mathbb{H}^5/\Gamma$ corresponding to the group $\Gamma$ generated by reflections in the faces of a $5$-dimensional hyperbolic simplex $S$ with Coxeter diagram represented in Figure~\ref{fig:5-simplex}.

\begin{figure}[h]
    \centering
    \includegraphics[width=3.5cm]{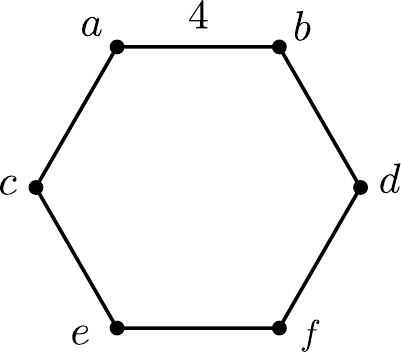}
    \caption{The Coxeter diagram of a non-compact hyperbolic $5$-simplex, with labels for its facets.}
    \label{fig:5-simplex}
\end{figure}

This simplex is non-compact and has $2$ ideal vertices. In \cite{Vin67}, Vinberg showed that the corresponding reflection group $\Gamma$ is a non-arithmetic lattice defined over the field $K = \Q(\sqrt{2})$. More recently, it was shown that $\Gamma$ is not quasi-arithmetic either (it is a pseudo-arithmetic lattice, see \cite{FLMS18} and \cite{EM20}), and that it is not commensurable with any lattice obtained by gluing arithmetic pieces, such as GPS \cite{GPS87} or ABT \cite{BT11} lattices.

Now we exhibit a $2$-dimensional arithmetic fc-subspace of this orbifold. This embedding is interesting because it does not arise through the standard techniques described, for instance, in \cite {KRS} and \cite{KRiS}. The first step consists in characterising the maximal lattice $\Gamma'$ corresponding to the commensurator of $\Gamma$. Notice that there is an isometric involution of $S$ which acts on the facets as the permutation $\tau = (a,b)(c,d)(e,f)$.

\begin{prop}\label{proposition:commensurator-5-simplex}
We have that $\Gamma' = \Comm(\Gamma) = \langle \Gamma, \tau \rangle$.
\end{prop}
\begin{proof}
The group $\Gamma$ can be easily shown to be a maximal co-finite reflection group in $\mathbb{H}^5$. Suppose that there exists a co-finite reflection group  $\Lambda$ which contains $\Gamma$. Then, according to \cite{FT10}, $\Lambda$ must be a reflection group associated with a non-arithmetic hyperbolic Coxeter $5$-simplex. However, $S$ is the only hyperbolic non-arithmetic Coxeter $5$-simplex \cite{JKTR}, therefore $\Lambda = \Gamma$. Then, it follows from \cite{Vin67} that $\Gamma' = \Gamma \rtimes \langle \tau \rangle$ is a maximal lattice.
\end{proof}

Notice that the discussion above implies that there is a splitting short exact sequence:
$$1 \rightarrow \Gamma \rightarrow \Gamma' \rightarrow \mathbb{Z}/2\mathbb{Z} \rightarrow 1,$$
therefore $\Gamma'$ decomposes as a semidirect product $\Gamma \rtimes \mathbb{Z}/2\mathbb{Z}$, where $\mathbb{Z}/2\mathbb{Z}$ is generated by the isometric involution $\tau$ and acts by conjugation on $\Gamma$ through the permutation $\tau$ of the generators.

Now, consider any codimension $k$ face $F$ of the simplex $S$ which is not an ideal vertex. Since $S$ is a simple polytope, $F$ will lie at the intersection of $k$ facets $F_1,\dots,F_k$ of $S$. The subgroup $G_F$ of $\Gamma$ generated by the reflections in these facets is finite and $\Fix(F)$ is precisely the totally geodesic subspace $H$ of $\mathbb{H}^5$ which supports the face $F$.  By Theorem \ref{theorem:cent}, the stabiliser of $H$ in $\Gamma$ acts as a lattice on $H$, and therefore defines a totally geodesic sublattice of $\Gamma$. 

By applying Theorem \ref{theorem:cent}, it is easy to see that there is another totally geodesic sublattice in $\Gamma$ which corresponds to the fixed point set of the involution $\tau$. Let us prove the following fact.
\begin{prop}
The fixed point set of the involution $\tau$ is a hyperbolic plane $H \cong \mathbb{H}^2 \subset \mathbb{H}^5$. The stabiliser of $H$ in $\Gamma$ acts on $H$ as the arithmetic $(2,4,8)$ triangle reflection group. 
\end{prop}
\begin{proof}
Realise the simplex $S$ combinatorially as the projectivisation of the positive orthant in $\mathbb{R}^6$. Up to an appropriate identification of the vertices of $S$ with the vectors of the standard basis of $\mathbb{R}^6$, we can suppose that the involution $\tau$ is realised by the following permutation of the vectors of the standard basis: $(e_1,e_2)(e_3,e_4)(e_5,e_6)$.

The corresponding matrix has eigenvalue $1$ with multiplicity $3$, and the corresponding $3$-dimensional eigenspace intersects the positive orthant in the subset $$\{x_1=x_2,\, x_3=x_4,\, x_5=x_6,\, x_i \ge 0\}.$$

By projecting onto the hyperboloid, we see that the fixed points set of the involution $\tau$ in $S$ is a hyperbolic triangle $T$ whose sides $s_1$, $s_2$, $s_3$ lie respectively in $a \cap b$, $c \cap d$ and  $e \cap f$. We denote by $H$ the fixed points set in $\mathbb{H}^5$ of the involution $\tau$: this will be a hyperbolic plane tessellated by copies of $T$.

The centraliser $C(\tau)$ of the involution $\tau$ in $\Gamma'$ coincides with the stabiliser of $H$, and can easily be seen to be generated by $\tau$ together with $r_1 = abab$, $r_2 = cd$ and $r_3 = efe$, where each  $r_i$ acts on $H$ as a reflection in the side $s_i$ of $T$, and $\tau$ acts as the identity. In order to describe both the geometry of $T$ and the action of $C(\tau)$ on $H$ it is sufficient to find the order of the products of two reflections in the sides of $T$. An easy computation with finite Coxeter groups shows that the orders of $r_1\cdot r_2$, $r_2\cdot r_3$, and $r_1 \cdot r_3$ are $8$, $4$, and $2$, respectively. Thus, the stabiliser of $H$ is the $(2,4,8)$-triangle group, which is arithmetic by \cite{Takeuchi}.
\end{proof}


\begin{thebibliography}{0}

\bibitem{Ag06}
I.~Agol, {\it Systoles of hyperbolic 4-manifolds}, arXiv preprint math/0612290, (2006).

\bibitem{AMN} D. Alekseevski, P.~Michor, Y.~Neretin, {\itshape Rolling of Coxeter polyhedra along mirrors}, Geometric Methods in Physics, XXXI Workshop, eds. Kielanowski et. al, Birkha\"user (2012), pp. 67--86.

\bibitem{Al13} D.~Allcock, {\itshape Reflection centralizers in Coxeter groups}, Transformation Groups {\bf 18}, no. 3 (2013), pp. 599--613.

\bibitem{BFMS20}
U.~Bader, D.~Fisher, N.~Miller, M.~Stover, {  \itshape Arithmeticity, superrigidity, and totally geodesic submanifolds}, Ann. of Math. (2), 2021, {\bf 193}, no. 3, pp. 837--861.

\bibitem{BFMS23}
U. Bader, D. Fisher, N. Miller, and M. Stover, {\itshape Arithmeticity, superrigidity and totally geodesic submanifolds of complex hyperbolic manifolds}, Invent. Math., {\bf 233} (2023), pp. 169--222.

\bibitem{BU23} G.~Baldi, E.~Ullmo, {\itshape Special subvarieties of non-arithmetic ball quotients and Hodge theory}, Ann. of Math. (2), {\bf 197} (2023), pp. 159--220.

\bibitem{BT11}
M.~Belolipetsky, S.~A.~Thomson, {\itshape Systoles of hyperbolic manifolds}, Algebr. Geom. Topol. {\bf 11} (2011), pp. 1455--1469. 


\bibitem{BC-Asterisque}
N.~Bergeron, L.~Clozel. 
{\itshape Spectre automorphe des variétés hyperboliques et applications topologiques},
Asterisque, tome 303 (2005), pp. 1--218.

\bibitem{BC13} N.~Bergeron, L.~Clozel. 
{\itshape Quelques cons\'equneces des travaux d'Arthur pour le spectre et la topologie des vari\'et\'es hyperboliques},
Invent. Math. {\bf 192} (2013), pp. 505--532.

\bibitem{BC17}
N.~Bergeron, L.~Clozel. 
{\itshape Sur la cohomologie des variétés hyperboliques de dimension 7 trialitaires} (French. French summary) [On the cohomology of trialitary hyperbolic manifolds of dimension 7],
Israel J. Math. {\bf 222} (2017), no. 1, pp. 333--400.

\bibitem{BG-rigidity} N.~Bergeron, T.~Gelander. {\itshape A note on local rigidity}, Geom. Dedicata {\bf 107} (2004), pp. 111--131.

\bibitem{BHW11}
N.~Bergeron, F.~Haglund, and D.\,T.~Wise, {\it Hyperplane sections in arithmetic hyperbolic manifolds}, J. Lond. Math. Soc. (2), {\bf 83} (2011), pp. 431--448.

\bibitem{BK20}
N.~V.~Bogachev, A.~Kolpakov, {\itshape On faces of quasi-arithmetic Coxeter polytopes}, Int. Math. Res. Notices, Vol. 2021, Issue {\bf 4} (2021), pp. 3078--3096.

\bibitem{Borel}
A.~Borel, {\itshape Linear algebraic groups}, 2nd ed., Graduate Texts in Mathematics, vol. {\bf 126}, Springer-Verlag, New York, 1991.

\bibitem{Borel-adjoint} A.~Borel, {\itshape Density and maximality of arithmetic subgroups}, J. Reine Angew. Math. {\bf 224} (1966), pp. 78--89.

\bibitem{Borel-density} A.~Borel, {\itshape Density properties for certain subgroups of semisimple groups without compact components}, Ann. of Math. (2) {\bf 72} (1960), pp. 179–-188.

\bibitem{BHC62} 
A.~Borel and Harish--Chandra, {\itshape Arithmetic subgroups of algebraic groups}, Ann. Math. {\bf 75} (1962), pp. 485--535.

\bibitem{Borel-Serre} A.~Borel, J.-P.~Serre, {\itshape Th\'eor\`emes de finitude en cohomologie galoisienne}, Comment. Math. Helv. {\bf 39} (1964), pp. 119--171.

\bibitem{CLR} D.~Cooper, D.~Long. A.~W.~Reid, {\itshape On the virtual Betti number of arithmetic hyperbolic $3$-manifolds}, Geom. Topol. {\bf 11}(4) (2007), pp. 2265--2276.


\bibitem{Cor92} K.~Corlette, {\itshape Archimedean superrigidity and hyperbolic geometry}, Ann. of Math. (2) {\bf 135}:1 (1992), pp. 165--182.

\bibitem{D23} M.~Deraux, {\itshape Mirror stabilizers for lattice complex hyperbolic triangle groups}, {\texttt arXiv:2301.07387 [math.GT]}.

\bibitem{Emery}
V.~Emery, {\itshape Du volume des quotients arithm\'{e}tiques de l’espace hyperbolique}, PhD Thesis, University of Fribourg (2009)

\bibitem{EM20}
V.~Emery, O.~Mila, {\itshape Hyperbolic manifolds and pseudo-arithmeticity}, Trans. Amer. Math. Soc. Ser. B {\bf 8} (2021), pp. 277--295.

\bibitem{ERT}
V.~Emery, J.~G.~Ratcliffe, S.~T.~Tschantz, {\itshape Salem numbers and arithmetic hyperbolic groups}, Trans. Amer. Math. Soc. {\bf 372} (2019), pp. 329--355.

\bibitem{EG18}
H.~Esnault, M.~Groechenig, {\it Cohomologically rigid local systems and integrality}, Selecta Math. (N.S.), {\bf 24} (2018), pp. 4279--4292.

\bibitem{FT10} A.~Felikson, P.~Tumarkin, {\itshape Reflection subgroups of Coxeter groups}, Trans. Amer. Math. Soc. {\bf 362} (2010), pp. 847--858.

\bibitem{Fenchel} W.~Fenchel, {\itshape Elementary geometry in hyperbolic space}, De Gruyter Studies in Mathematics {\bf 11} (1989).


\bibitem{FLMS18}
D.~Fisher, J.F.~Lafont, N.~Miller, M.~Stover, {\itshape Finiteness of maximal totally geodesic submanifolds in hyperbolic hybrids},  J. Eur. Math. Soc. {\bf 23} (2021), no. 11, 3591--3623, \texttt{arXiv:1802.04619}.

\bibitem{FLMS-erratum}
D.~Fisher, J.F.~Lafont, N.~Miller, M.~Stover, {\itshape Corrigendum to ``Finiteness of maximal totally geodesic submanifolds in hyperbolic hybrids''},  J. Eur. Math. Soc. (2025), DOI: 10.4171/JEMS/1711.

\bibitem{Garibaldi} R.~S.~Garibaldi, {\itshape Isotropic trialitarian algebraic groups}, Journal of Algebra, {\bf 210} (1998), pp. 385--418.


\bibitem{God62}
R.~Godement, {\itshape Domaines fondamentaux des groupes arithmetiques}, Semin.
Bourbaki 1962--1963, 1964, {\bf 15}, no. 3, exp. 257, pp. 119--131.

\bibitem{GPS87}
M.~Gromov, I.~Piatetski-Shapiro, {\itshape Non--arithmetic groups in Lobachevsky spaces}, Inst. Hautes \'Etudes Sci. Publ. Math. {\bf 66} (1987), pp. 93--103.


\bibitem{GS92}
M.~Gromov, R.~Schoen, {\itshape Harmonic maps into singular spaces and $p$-adic superrigidity for lattices in groups of rank one}, Inst. Hautes \'{E}tudes Sci. Publ. Math.. {\bf 76} (1992), pp. 165--246.


\bibitem{Humphreys} J.~E.~Humphreys, \textit{Introduction to Lie algebras and representation theory}, Graduate Texts in Mathematics \textbf{9} (1972), 173 pp.


\bibitem{JKTR} N.~Johnson, R.~Kellerhals, J.~Ratcliffe, S.~Tschantz, {\itshape The size of a hyperbolic Coxeter simplex}, Transformation Groups {\bf 4} (1999), pp. 329--353.

\bibitem{INV} M.~A.~Knus, A.~Merkurjev, M.~Rost, J.~P.~Tignol, {\itshape The book of involutions}, Colloqium publications {\bf 44}, Amer. Math. Soc., (1998).

\bibitem{KRS} A.~Kolpakov, A.~W.~Reid, L.~Slavich, {\itshape Embedding arithmetic hyperbolic manifolds}, Math.~Res.~Lett. \textbf{25} (2018), pp. 1305--1328.

\bibitem{KRiS} A.~Kolpakov, S.~Riolo, L.~Slavich, {\itshape Embedding non-arithmetic hyperbolic manifolds}, Math.~Res.~Lett. \textbf{29} No.1 (2022), pp. 247--274. \texttt{arXiv:2003.01707}. 

\bibitem{LLR08} M.~Lackenby, D.~Long, A.~W.~Reid, 
{\itshape Covering spaces of arithmetic 3-orbifolds},
Int. Math. Res. Not. IMRN, Volume 2008, Art. ID rnn036, 38 pp.

\bibitem{LP20} K.~Le, R.~Palmer, {\itshape Totally geodesic surfaces in twist knot complements},
Pacific J. Math., {\bf Vol. 319}, No. 1 (2022), pp. 153--179, DOI: 10.2140/pjm.2022.319.153, \texttt{arXiv:2009.04637}.

\bibitem{Li-Millson} J.~S.~Li, J.~J.~Millson, {\itshape On the first Betti number of a hyperbolic manifold with an arithmetic fundamental group}, Duke Math. J. {\bf 71}(2) (1993), pp. 365--401.

\bibitem{Mac-Reid} C.~Maclachlan, A.~W.~Reid, {\itshape The arithmetic of hyperbolic $3$-manifolds}, Graduate Text in Math. {\bf 219}, Springer--Verlag (2003).

\bibitem{Margulis77} G.~A.~Margulis, {\itshape Discrete groups of motions of manifolds of non--positive curvature}, Amer. Math. Soc. Transl. {\bf 109} (1977), pp. 33--45.

\bibitem{Mar84} G.~A.~Margulis, {\itshape Arithmeticity of the irreducible lattices
in the semi-simple groups of rank greater than 1}, Invent. Math. (1984), {\bf 76}, pp. 93--120.

\bibitem{Margulis-book}
G.~A.~Margulis, {\itshape Discrete Subgroups of Semisimple Lie Groups}, Band 17 von 3. Folge, Ergebnisse der Mathematik und ihrer Grenzgebiete. Springer: Berlin -- Heidelberg (1991), 388 pp. 

\bibitem{MM20}
G.~A.~Margulis, A.~Mohammadi, {\itshape Arithmeticity of hyperbolic 3-manifolds containing infinitely many totally geodesic surfaces}, Ergod. Th. \& Dynam. Systems, {\bf 42} (2022), no. 3, 1188--1219.

\bibitem{Mey17}
J.~S.~Meyer, {\itshape Totally geodesic spectra of arithmetic hyperbolic spaces}, Trans. Amer. Math. Soc. {\bf 369} (2017), no. 11, pp. 7549--7588.


\bibitem{WM} D.~W.~Morris, {\itshape Introduction to arithmetic groups}, Deductive Press (2015), 492 pp.

\bibitem{OV}
A.~L.~Onishchik, \`{E}.~B.~Vinberg, (Eds.), {\itshape Lie groups and Lie algebras II}, Encyclopaedia of Mathematical Sciences {\bf 21}, Springer: Berlin -- Heidelberg (2010), p. 224. 

\bibitem{Pierce}
R.~S.~Pierce, {\itshape Associative algebras}, Graduate Texts in Mathematics {\bf 88}, Springer: New York (1982), 436 pp.

\bibitem{Pla-Rap} V.~Platonov, A.~Rapinchuk, \textit{Algebraic groups and Number Theory}, Academic Press (1993), 614 pp.


\bibitem{PR}
G.~Prasad, A.~Rapinchuk. {\itshape Weakly commensurable arithmetic groups and isospectral locally symmetric spaces}, Publ. Math. IH\'{E}S, {\bf 109} (2009), pp. 113--184.

\bibitem{Raghunathan}
M.~S.~Raghunathan, {\itshape Discrete subgroups of Lie groups}, Berlin--Heidelberg, Springer (1972), 227 pp.

\bibitem{Reid} A.~W.~Reid, {\itshape Arithmeticity of knot complements}, J. London Math. Soc.  {\bf s2--43} Issue 1 (1991), pp. 171--184.

\bibitem{RW} A.~W.~Reid, G.~S.~Walsh, {\itshape Commensurability classes of $2$-bridge knot complements}, Algebr. Geom. Topol. {\bf 8} (2008), no. 2, pp. 1031--1057.

\bibitem{Sam14} I.~Samet, {\itshape On the number of finite subgroups of a lattice}, 
Comment. Math. Helv. {\bf 89} (2014), no. 3, pp. 759--781.

\bibitem{Schwerner} J.~Schwermer, {\itshape Geometric cycles, arithmetic groups and their cohomology}, Bull. Amer. Math. Soc. {\bf 47} (2010), pp. 187--279.

\bibitem{Serre}
J-P.~Serre, {\itshape A course in Arithmetic}, Graduate Texts in Math. {\bf 7}, Springer-Verlag (1973). 

\bibitem{Serre-Gal-cohomology}
J-P.~Serre, {\itshape Galois cohomology}, Springer-Verlag, Berlin (1997).

\bibitem{Takeuchi} K.~Takeuchi, {\itshape Arithmetic triangle groups}, J. Math. Soc. Japan {\bf 29}(1) (1977), pp. 99--106.

\bibitem{Tho16}
S.~Thomson, {\it Quasi-arithmeticity of lattices in $\PO(n, 1)$}, Geom. Dedicata, {\bf 180} (2016), pp. 85--94

\bibitem{Tits66}
J.~Tits, {\itshape Classification of algebraic semisimple groups}, Algebraic Groups and Discontinuous Subgroups (Proc. Sympos. Pure Math., Boulder, Colo., 1965), Amer. Math. Soc., Providence, RI, 1966, pp. 33--62.

\bibitem{Vin67}
\`E.~B.~Vinberg, {\itshape Discrete groups generated by reflections in Lobachevski spaces}, Math. Sb. {\bf 72} (1967), pp. 471--488.

\bibitem{Vin71}
\`E.~B.~Vinberg, {\itshape Rings of definition of dense subgroups of semisimple linear groups}, Math. of USSR--Izvestiya (1971), {\bf 5} (1), pp. 45--55.

\bibitem{Vin-Geometry-II}
\`E.~B.~Vinberg, (Ed.), {\itshape Geometry II}, Encyclopaedia of Mathematical Sciences {\bf 29}. Springer-Verlag, Berlin, 1993.

\bibitem{Vin95} \`E.~B.~Vinberg, {\itshape The smallest field of definition of a subgroup of the group $\mathrm{PSL}_2$},  Mat. Sb., {\bf184}:10 (1993), pp. 53–-66; Russian Acad. Sci. Sb. Math., {\bf80}:1 (1995), pp. 179--190.

\bibitem{Vin14} \`E.~B.~Vinberg, {\itshape Non--arithmetic hyperbolic reflection groups in higher dimensions},  Moscow Math. J. {\bf 15} (2015), pp. 593--602.

\end{thebibliography}
\end{document}